\documentclass[a4paper]{article}

\usepackage[latin1]{inputenc}
\usepackage[T1]{fontenc}
\usepackage{amsmath,amssymb,amsfonts,amscd,amsthm,mathrsfs,mathtools}
\usepackage[all]{xy}
\usepackage{appendix,
listings}
\usepackage[bookmarks=false]{hyperref}
\usepackage{color}
\usepackage[margin=3cm]{geometry}
\usepackage{fullpage}

\DeclareMathOperator{\Aut}{Aut}
\DeclareMathOperator{\End}{End}
\DeclareMathOperator{\id}{id}

\DeclareMathOperator{\im}{Im}
\DeclareMathOperator{\Ker}{Ker}
\DeclareMathOperator{\Sym}{Sym}
\DeclareMathOperator{\SSym}{\mathbb{S}ym}

\DeclareMathOperator{\GL}{GL}
\DeclareMathOperator{\Sp}{Sp}
\DeclareMathOperator{\ch}{ch}
\DeclareMathOperator{\rk}{rk}
\DeclareMathOperator{\ad}{ad}
\DeclareMathOperator{\td}{td}
\DeclareMathOperator{\pr}{pr}
\DeclareMathOperator{\supp}{supp}
\DeclareMathOperator{\discr}{discr}
\DeclareMathOperator{\Vect}{Vect}
\DeclareMathOperator{\coker}{coker}
\DeclareMathOperator{\Fix}{Fix}
\DeclareMathOperator{\codim}{Codim}
\DeclareMathOperator{\Supp}{Supp}
\DeclareMathOperator{\NS}{NS}

\DeclareMathOperator{\Mon}{Mon}
\DeclareMathOperator{\Bl}{Bl}
\DeclareMathOperator{\Diff}{Diff}

\renewcommand{\rho}{\varrho}
\newcommand{\hilb}[1]{^{[#1]}}
\newcommand{\ie}{{\it i.e. }}

\newcommand{\vac}{|0\rangle}

\newcommand{\tors}{{\rm{tors}}}
\newcommand{\DP}{:}

\renewcommand{\d}{\mathfrak{d}}
\newcommand{\p}{\mathfrak{p}}
\newcommand{\G}{\mathfrak{G}}
\newcommand{\q}{\mathfrak{q}}

\newcommand{\defIs}{\vcentcolon=}
\newcommand{\kum}[2]{K_{ #2 }( #1 )}
\newcommand{\X}{\kum{A}{2}}
\newcommand{\cc}{c_2(\X)}

\newcommand{\C}{\mathbb{C}}
\renewcommand{\H}{\mathbb{H}}
\newcommand{\R}{\mathbb{R}}
\newcommand{\Q}{\mathbb{Q}}
\newcommand{\Z}{\mathbb{Z}}
\newcommand{\F}{\mathbb{F}_{2}}

\newcommand{\kq}{\mathfrak{q}}
\newcommand{\vect}[1]{\left( \begin{smallmatrix} #1 \end{smallmatrix} \right)}
\newcommand{\plan}[2]{\left< \vect{ #1 }, \vect{ #2 } \right>}

\newcommand{\incl}[1][r]
  {\ar@<-0.2pc>@{^(-}[#1] \ar@<+0.2pc>@{-}[#1]}
 
\theoremstyle{plain}
\newtheorem{theorem}{Theorem}[section]
\newtheorem{thm}[theorem]{Theorem}
\newtheorem{lemma}[theorem]{Lemma}
\newtheorem{lemme}[theorem]{Lemma}
\newtheorem{proposition}[theorem]{Proposition}
\newtheorem{prop}[theorem]{Proposition}
\newtheorem{corollary}[theorem]{Corollary}
\newtheorem{cor}[theorem]{Corollary}
\newtheorem{defipro}[theorem]{Definition-Proposition}
\theoremstyle{definition}
\newtheorem{definition}[theorem]{Definition}
\newtheorem{defi}[theorem]{Definition}
\newtheorem{notation}[theorem]{Notation}
\theoremstyle{remark}
\newtheorem{remark}[theorem]{Remark}
\newtheorem{rmk}[theorem]{Remark}

\begin{document}

\title{\bf Integral cohomology of the generalized\\Kummer fourfold}

\author{Simon Kapfer and Gr\'egoire Menet}



\maketitle
\begin{abstract}
We describe the integral cohomology of the generalized Kummer fourfold giving an explicit basis, using Hilbert scheme cohomology and tools developed by Hassett and Tschinkel.
Then we apply our results to a IHS variety with singularities, obtained by a partial resolution of the generalized Kummer fourfold quotiented by a symplectic involution. 
We calculate the Beauville--Bogomolov form of this new variety, presenting the first example of such a form that is odd.
\end{abstract}

\section{Introduction}
In algebraic geometry \emph{irreducible holomorphic symplectic (IHS) manifolds} became important objects of study in recent years, after fundamental results by Beauville \cite{Beauville} and Huybrechts \cite{Huybrechts2}.
Among all the developments concerning this field, integral cohomology plays an inescapable role. 
This is primarily due to the \emph{Beauville--Bogomolov form} which is a non-degenerate symmetric integral and primitive bilinear pairing on the second cohomology group with integral coefficients. 
This form endows the second cohomology group with a lattice structure establishing lattice theory as a fundamental tool omnipresent in all the last developments. 
As examples, we can cite works on classifications of automorphisms \cite{Mongardi}, \cite{MongWanTari}, \cite{BCS} or the important survey of Markman \cite{Markmansurvey} with results on the Kähler cone and the monodromy. 
In a more modest term, the fourth integral cohomology group is also quite useful. As examples, we can underline Theorem 1.2 of \cite{BNS} providing formulas which apply for the classification of automorphism on IHS manifolds of $K3^{[2]}$-type, particularly used in \cite{BCS}; furthermore Theorem 1.10 of \cite{Markman2} provides a description of the monodromy group of the IHS manifolds of $K3^{[n]}$-type; we can also cite \cite{Lol2}, where the second author provides the Beauville--Bogomolov lattice of the Markushevich--Tikhomirov varieties constructed in \cite{Markou}. 
Taking $X$ a IHS manifold of $K3^{[2]}$-type, in all these papers a description of the torsion group $\frac{H^{4}(X,\Z)}{\Sym^2(H^{2}(X,\Z))}$ was essential.

Until now, no complete description of the \emph{integral cohomology of the generalized Kummer fourfold} was existing. In particular, the relation between the fourth cohomology group and the image of the symmetric power of the second cohomology group via cup-product was not known. For all reasons mentioned above, it appeared to us that it was an interesting gap to fill.

Let $\kum{A}{2}$ be the generalized Kummer fourfold over a torus $A$. There are three main theorems in this paper. Two of them describe the integral cohomology of the generalized Kummer fourfold:
\begin{itemize}
\item\textbf{Theorem \ref{integralbasistheorem}}
which provides an integral basis of $H^4(K_2(A),\Z)$ in terms of $\Sym^2(H^2(K_2(A),$
$\Z))$ and certain classes of Brian\c con subschemes with support on three-torsion points, introduced in \cite{Hassett}.
\item\textbf{Theorem \ref{thetaTheorem}}
which states that the pullback from the Hilbert scheme of points on the torus $\theta^*:H^{*}(A^{[3]},\Z)\rightarrow H^{*}(K_2(A),\Z)$ is surjective except in degree 4. Moreover it provides an integral basis of $\im \theta^*$ 
and shows that the kernel of $\theta^*$ is the ideal generated by $H^1(A\hilb{3},\Z)$.
\end{itemize}

The third theorem is related to \emph{irreducible symplectic V-manifolds}; it can be seen as an application of Theorem \ref{integralbasistheorem} and a generalization of \cite{Lol2}. A V-manifold is a compact analytic complex space with at worst finite quotient singularities. A V-manifold will be called symplectic if its nonsingular locus is endowed with an everywhere non-degenerate holomorphic 2-form which extends to a resolution of singularities. 
A symplectic V-manifold will be called irreducible if it is complete, simply connected, and if the holomorphic 2-form is unique up to $\mathbb{C}^*$. Such varieties are good candidates to generalize the short list of known IHS manifolds, since some aspects of the theory were already generalized in \cite{Nanikawa} and \cite{Mat}, for instance the Beauville--Bogomolov form, the local Torelli theorem and the Fujiki formula. 

In \cite{Nanikawa}, Namikawa proposes a definition of the Beauville-Bogomolov form for some singular irreducible symplectic varieties. He assumes that the singularities are only $\mathbb{Q}$-factorial with a singular locus of codimension $\geq 4$. Under these assumptions, he proves a local Torelli theorem. 
This result was completed by a generalization of the Fujiki formula by Matsushita in \cite{Mat} (see also Theorem 1.2.4 of \cite{Lol} for a summaring satement).

These results were further generalized by Kirschner for symplectic complex spaces in \cite{Tim}. 
In \cite[Theorem 2.5]{Lol2} the first concrete example of Beauville--Bogomolov lattice for a singular irreducible symplectic variety has appeared. 
The variety studied in \cite{Lol2} is a partial resolution of an irreducible symplectic manifold of $K3^{[2]}$-type quotiented by a symplectic involution. The objective of this paper is to provide a new example of a Beauville--Bogomolov lattice replacing the manifold of $K3^{[2]}$-type by a fourfold of Kummer type. 
Knowing the integral basis of the cohomology group of the generalized Kummer provided by Theorem \ref{integralbasistheorem}, this calculation becomes possible. 
Moreover, the calculation will be much simpler as in \cite{Lol2} because of the general techniques for calculating integral cohomology of quotients developed in \cite{Lol} and the new technique using monodromy developed in Lemma \ref{Ddelta}. 
The other techniques developed in \cite{Lol2} are also contained in \cite{Lol}, so to simplify the reading, we will only cite \cite{Lol2} in the rest of the section.

Concretely, let $X$ be an irreducible symplectic fourfold of Kummer type and $\iota$ a symplectic involution on $X$. Theorem \ref{SymplecticInvo} establishes that the fixed locus of $\iota$ is the union of 36 points and a K3 surface $Z_0$. Then the singular locus of $K:=X/\iota$ is the union of a K3 surface and 36 points. The singular locus is not of codimension four. We will lift to a partial resolution of singularities,
$K'$ of $K$, obtained by blowing up the image of $Z_0$. By Section 2.3 and Lemma 1.2 of \cite{Fujiki2}, the variety $K'$ is an irreducible symplectic V-manifold which has singular locus of codimension four.

\begin{thm}\label{theorem}
Let $X$ be an irreducible symplectic fourfold of Kummer type and $\iota$ a symplectic involution on $X$.
Let $Z_0$ be the K3 surface which is in the fixed locus of $\iota$.
We denote $K=X/\iota$ and $K'$ the partial resolution of singularities of $K$ obtained by blowing up the image of $Z_0$.
Then the Beauville--Bogomolov lattice $H^2(K',\Z)$ is isomorphic to $U(3)^{3}\oplus\left(
\begin{array}{cc}
-5 & -4\\
-4 & -5 
\end{array} \right)$, and the Fujiki constant $c_{K'}$ is equal to $8$.
\end{thm}
We remark that it is the first example of a Beauville--Bogomolov form which is not even.


The paper is organised as follows. In Section \ref{OddHilb2} we describe the odd integral cohomology of $A^{[2]}$ the Hilbert scheme of two points on a surface $A$ with torsion free cohomology. 
Then, after recalling some notions on Nakajima operators in Section \ref{Section_Hilbert}, 
we are able to provide an integral basis of the Hilbert scheme of two points on an abelian surface in term of Nakajima operators (Proposition \ref{A2Basis}).
Section \ref{Section_GeneralKummer} studies the integral cohomology of generalized Kummer varieties in any dimension.
In Section \ref{Middle}, we use all these preliminary results and monodromy technique developed in \cite{Hassett} to find an integral basis of the cohomology of the generalized Kummer fourfold $K_2(A)$.
As a consequence, in Section \ref{Involution}, we are able 
to end the classification of symplectic involutions on $K_2(A)$ as a corollary of the lattice classification by Mongardi, Tari and Wandel in \cite{MongWanTari}.
Finally, Section \ref{BeauvilleForm} is dedicated to the proof of Theorem \ref{theorem}.
~\\


\section{Odd cohomology of \texorpdfstring{$A\hilb{2}$}{the Hilbert scheme of two points}}\label{OddHilb2}
Let $A$ be a smooth compact surface with torsion free cohomology and $A\hilb{2}$ the Hilbert scheme of two points. 
It can be constructed as follows: Consider the direct product $A\times A$. Denote 
$$b: \Bl_\Delta(A\! \times\! A) \rightarrow A\times A $$ 
the blow-up along the diagonal $\Delta \cong A$ with exceptional divisor $E$.
Let $j: E\rightarrow \Bl_\Delta(A\! \times\! A) $ be the embedding. 
The action of $\mathfrak{S}_2$ on $A\times A$ lifts to an action on $\Bl_\Delta(A\! \times\! A)$. 
We have the pushforward $j_*:H^*(E,\Z)\rightarrow H^*(\Bl_\Delta(A\! \times\! A),\Z) $.

The quotient by the action of $\mathfrak{S}_2$ is 
$ \pi:\Bl_\Delta(A\! \times\! A)\rightarrow A\hilb{2}$.
Now, $A\hilb{2}$ is a compact complex manifold with torsion-free cohomology,~\cite[Theorem~2.2]{Totaro}.
In this section, we want to prove the following proposition.
\begin{proposition} \label{Alpha35}
Let $A$ be a smooth compact surface with torsion free cohomology.
Then
\begin{itemize}
\item[(i)]
$H^{3}(A\hilb{2},\Z)=\pi_*(b^{*}(H^{3}(A\times A,\Z)))\oplus \frac{1}{2}\pi_*j_*b_{|E}^{*}(H^1(\Delta,\Z))$,
\item[(ii)]
$H^{5}(A\hilb{2},\Z)=\pi_*(b^{*}(H^{5}(A\times A,\Z)))\oplus \pi_*j_*b_{|E}^{*}(H^3(\Delta,\Z))$.
\end{itemize}
\end{proposition}
We adopt the following notation.
\begin{notation} \label{TorusClasses}
We denote the generators of $H^1(A,\Z)$ by $a_i$, and their respective duals by $a_i^*\in H^3(A,\Z)$. 
We denote the generator of the top cohomology $H^4(A,\Z)$ by $x$.
A basis of $H^2(A,\Z)$ will be denoted by $(b_i)$.
\end{notation}
The rest of this section is dedicated to the proof of this proposition. 
This proposition is proved using techniques developed in \cite{Lol}, for another approach see \cite{Totaro}.
The proof is organized as follows. First we recall some notions on integral cohomology endowed with the action of an involution in Section \ref{IntegralTools}. 
Then Section \ref{Prelemma} is devoted to calculate the torsion of $H^{3}(A\hilb{2}\smallsetminus E,\Z)$ (Lemma~\ref{3}) using equivariant cohomology techniques. Then this knowledge allow us to deduce (i) using the exact sequence (\ref{exactutile}) of Section \ref{=} and (ii) using the unimodularity of the lattice $H^{3}\left(\Bl_\Delta(A\! \times\! A),\Z\right)\oplus H^{5}\left(\Bl_\Delta(A\! \times\! A),\Z\right)$.
\subsection{Integral cohomology under the action of an involution}\label{IntegralTools}
The main references of this subsection are \cite{Lol} and \cite{BNS}.

Let $G=\left\langle \iota\right\rangle$ be the group generated by an involution $\iota$ on a complex manifold $X$.
As denoted in \cite[Section 5]{BNS}, let $\mathcal{O}_{K}$ be the ring $\Z$ with the following $G$-module structure:
$\iota\cdot y=-y$ for $y\in \mathcal{O}_{K}$. For $a\in \Z$, we also denote by $(\mathcal{O}_{K},a)$ the module $\Z\oplus\Z$ whose $G$-module structure is defined by $\iota\cdot(y,k)=(-y+ka,k)$. We also denote by $N_{2}$ the $\mathbb{F}_{2}[G]$-module $(\mathcal{O}_{K},a)\otimes\mathbb{F}_{2}$.
We recall Definition-Proposition 2.2.2 of \cite{Lol}.
\begin{defipro}\label{defiprop}
Assume that $H^{*}(X,\Z)$ is torsion-free. Then for all $0\leq k \leq 2\dim X$, we have an isomorphism of $\Z[G]$-modules:
$$H^{k}(X,\Z)\simeq \bigoplus_{i=1}^{r}(\mathcal{O}_{K},a_{i})\oplus \mathcal{O}_{K}^{\oplus s}\oplus\Z^{\oplus t},$$
for some odd numbers $a_{i}$ and $(r,s,t)\in\mathbb{N}^3$.
We get the following isomorphism of $\mathbb{F}_{2}[G]$-modules:
$$H^{k}(X,\mathbb{F}_{2})\simeq N_{2}^{\oplus r}\oplus\mathbb{F}_{2}^{\oplus (s+t)}.$$
We denote $l_{2}^k(X):=r$, $l_{1,-}^k(X):=s$, $l_{1,+}^k(X):=t$, $\mathcal{N}_{2}:=N_{2}^{\oplus r}$ and $\mathcal{N}_{1}:=\mathbb{F}_{2}^{\oplus s+t}$.
\end{defipro}
\begin{rmk}
These invariants are uniquely determined by $G$, $X$ and $k$.
\end{rmk}
\begin{prop}\cite[Sect.2~2]{Lol}\label{sarti}
Let $X$ be a compact complex manifold of dimension $n$ and $\iota$ an involution. Assume that $H^{*}(X,\Z)$ is torsion free.
We have:
\begin{itemize}
\item[(i)]
$\rk H^{k}(X,\Z)^{\iota}=l_{2}^k(X)+l_{1,+}^k(X).$
\item[(ii)]
We denote $\sigma:=\id+\iota^*$ and $S^k_\iota:= \Ker \sigma \cap H^{k}(X,\Z)$. 
We have $H^{k}(X,\Z)^{\iota}\cap S^k_\iota=0$ and
$$\frac{H^{k}(X,\Z)}{H^{k}(X,\Z)^{\iota}\oplus S^k_\iota}=\left(\frac{\Z}{2\Z}\right)^{\oplus l_2^k(X)}.$$
\end{itemize}
\end{prop}
\begin{rmk}\label{x+ix}
Note that the elements of $(\mathcal{O}_{K},a_{i})^{\iota}$ are written $y+\iota^{*}(y)$ with $y\in (\mathcal{O}_{K},a_{i})$.
\end{rmk}
Let $\pi: X\rightarrow X/G$ be the quotient map. 
We denote by $\pi^{*}$ and $\pi_{*}$ respectively the pull-back and the push-forward along $\pi$. We recall that 
\begin{equation}
\pi_*\circ \pi^*=2\id \text{ and } \pi^*\circ\pi_* =\id+\iota^*.
\label{pietpi}
\end{equation}
We also recall the commutativity behaviour of $\pi_*$ with the cup product.
\begin{prop}\cite[Lemma 3.3.7]{Lol}\label{commut}
Let $X$ be a compact complex manifold of dimension $n$ and $\iota$ an involution. Assume that $H^{*}(X,\Z)$ is torsion free.
Let $0\leq k \leq 2n$, $m$ an integer such that $km\leq 2n$, and let $(x_{i})_{1\leq i \leq m}$ be elements of $H^{k}(X,\Z)^{\iota}$.
Then $$\pi_{*}(x_{1})\cdot...\cdot \pi_{*}(x_{m})=2^{m-1}\pi_{*}(x_{1}\cdot...\cdot x_{m}).$$
\end{prop}
\subsection{Preliminary lemmas}\label{Prelemma}
We denote $V=\Bl_\Delta(A\! \times\! A)\smallsetminus E$ and $U=V/\mathfrak S_{2}=A\hilb{2}\smallsetminus E$, where $\mathfrak{S}_{2}=\left\langle \sigma_{2}\right\rangle$. 
\begin{lemma}\label{1}
We have: $H^{k}(A\times A,\Z)=H^{k}(V,\Z)$
for all $k\leq 3$.
\end{lemma}
\begin{proof}
We have $V\cong A\times A\smallsetminus \Delta$,
so we get the following natural exact sequence:
$$\xymatrix{ \cdots\ar[r]&H^{k}(A\times A,V,\Z)\ar[r] & H^{k}(A\times A,\Z)\ar[r] & H^{k}(V,\Z)\ar[r]& \cdots}$$
Moreover, by Thom isomorphism $H^{k}(A\times A,V,\Z)=H^{k-4}(\Delta,\Z)=H^{k-4}(A,\Z)$.
Hence $H^{k}(A\times A,V,\Z)=0$ for all $k\leq 3$.
Hence $H^{k}(A\times A,\Z)=H^{k}(V,\Z)$
for all $k\leq 2$. It remains to consider the following exact sequence:
$$\xymatrix{ 0\ar[r]&H^{3}(A\times A,\Z)\ar[r]&H^{3}(V,\Z)\ar[r] & H^{4}(A\times A,V,\Z)\ar[r]^{\rho} & H^{4}(A\times A,\Z)}.$$
The map $\rho$ is given by $\Z \left[\Delta\right] \rightarrow H^{4}(A\times A,\Z)$.
Using Notation~\ref{TorusClasses}, the class $x\otimes 1$ is also in $H^{4}(A\times A,\Z)$ and intersects $\Delta$ in one point.
Hence the class of $\Delta$ in $H^{4}(A\times A,\Z)$ is not trivial and the map $\rho$ is injective.
It follows that 
\begin{equation*}
H^{3}(A\times A,\Z)=H^{3}(V,\Z).
\qedhere
\end{equation*}
\end{proof}
Now we will calculate the invariant $l_{1,-}^{2}(A\times A)$ and $l_{1,+}^{1}(A\times A)$ from Definition-Proposition~\ref{defiprop}.

\begin{lemma}\label{2}
We have: $l_{1,+}^{1}(A\times A)=0$ and $l_{1,-}^{2}(A\times A)=b_1(A)$.
\end{lemma}
\begin{proof}
By Künneth formula we have:
$$H^{1}(A\times A,\Z)=H^{0}(A,\Z)\otimes H^{1}(A,\Z)\oplus H^{1}(A,\Z)\otimes H^{0}(A,\Z).$$
The elements of $H^{0}(A,\Z)\otimes H^{1}(A,\Z)$ and $H^{1}(A,\Z)\otimes H^{0}(A,\Z)$ are exchanged under the action of $\sigma_2$. It follows that $l_{2}^{1}(A\times A)=b_1(A)$ and necessary $l_{1,-}^{1}(A\times A)=l_{1,+}^{1}(A\times A)=0$.
Using Künneth again, we get: 
\begin{align*}
H^{2}(A\times A,\Z)&=H^{0}(A,\Z)\otimes H^{2}(A,\Z)\oplus H^{1}(A,\Z)\otimes H^{1}(A,\Z)\\
&\oplus H^{2}(A,\Z)\otimes H^{0}(A,\Z).
\end{align*}
As before, the elements of $H^{0}(A,\Z)\otimes H^{2}(A,\Z)$ and $H^{2}(A,\Z)\otimes H^{0}(A,\Z)$ are exchanged under the action of $\sigma_2$.
Futhermore, the elements $z\otimes y\in H^{1}(A,\Z)\otimes H^{1}(A,\Z)$ are sent to $-y\otimes z$ by the action of $\sigma_2$. Such an element is anti-invariant by the action of $\sigma_2$ if $z=y$. It follows:
$$l_{2}^{2}(A\times A)=b_2(A)+\frac{b_1(A)(b_1(A)-1)}{2},$$
$$l_{1,-}^{2}(A\times A)=b_1(A),$$
and thus:
\begin{equation*}
l_{1,+}^{2}(A\times A)=0.
\qedhere
\end{equation*}
\end{proof}
\begin{lemma}\label{3}
The torsion part of the group $H^{3}(U,\Z)$ is isomorphic to $(\Z/2\Z)^{\oplus b_1(A)}$.
\end{lemma}
\begin{proof}
Using the spectral sequence of equivariant cohomology, it follows from Proposition 3.2.5 of~\cite{Lol}, Lemma~\ref{1} and~\ref{2}.
\end{proof}
\subsection{Third cohomology group}\label{=}
By Theorem 7.31 of~\cite{Voisin}, we have:
\begin{equation}
H^{3}(\Bl_\Delta(A\! \times\! A),\Z)=b^{*}(H^{3}(A\times A,\Z))\oplus j_*b_{|E}^{*}(H^{1}(\Delta,\Z)).
\label{voisin1}
\end{equation}
It follows that $$H^{3}(A^{[2]},\Z)\supset \pi_{*}b^{*}(H^{3}(A\times A,\Z))\oplus \pi_{*}j_*b_{|E}^{*}(H^{1}(\Delta,\Z)).$$
We want to find an equality. We will proceed as follows: We first prove that $\pi_{*}b^{*}(H^{3}(A\times A,\Z))$ is primitive. Then, in Lemma~\ref{primitive3}, we show that all the elements of $\pi_{*}j_*b_{|E}^{*}(H^{1}(\Delta,\Z))$ are divisible by 2 and finally we remark that this implies that the direct sum $\pi_*b^{*}(H^{3}(A\times A,\Z))\oplus \frac{1}{2}\pi_*j_*b_{|E}^{*}(H^{1}(\Delta,\Z))$ is primitive.  

It follows from the K\"unneth formula that
all elements in $H^{3}(A\times A,\Z)^{\mathfrak{S}_2}$ are written as $y+\sigma_{2}^{*}(y)$ with $y\in H^{3}(A\times A,\Z)$.
Since $\frac{1}{2}\pi_{*}(y+\sigma_{2}^{*}(y))=\pi_{*}(y)$, it follows that $\pi_{*}(b^*(H^{3}(A\times A,\Z)))$ is primitive in $H^{3}(A^{[2]},\Z)$.
Moreover by (\ref{voisin1}), we have the following values which will be used in Section~\ref{d5}:
$$l_2^3(\Bl_\Delta(A\! \times\! A))=\rk H^{3}(A\times A,\Z)^{\mathfrak{S}_2}=b_1(A)(b_2(A)+1).$$
and
\begin{equation}
l_{1,+}^3(\Bl_\Delta(A\! \times\! A))=\rk H^{1}(\Delta,\Z)^{\mathfrak{S}_2}=b_1(A),\ \text{ and }\ l_{1,-}^3(\Bl_\Delta(A\! \times\! A))=0.
\label{l1}
\end{equation}
\begin{lemma}\label{primitive3}
All the elements of the group $\pi_{*}j_*b_{|E}^{*}(H^{1}(\Delta,\Z))$ are divisible by 2 in $H^{3}(A^{[2]},\Z)$.
\end{lemma}
\begin{proof}
We consider the following commutative diagram:
\begin{equation}
\xymatrix@C=10pt{
\ar[d]^{d\pi^{*}} H^{3}(\mathscr{N}_{A^{[2]}/\pi(E)},\mathscr{N}_{A^{[2]}/\pi(E)}\smallsetminus0,\Z)=H^{3}(A^{[2]},U,\Z)
\ar[r]^-{g}&
H^{3}(A^{[2]},\Z)\ar[d]_{\pi^{*}} \\
H^{3}(\mathscr{N}_{\Bl_\Delta(A\! \times\! A)/E},\mathscr{N}_{\Bl_\Delta(A\! \times\! A)/E}\smallsetminus0,\Z)=H^{3}(\Bl_\Delta(A\! \times\! A),V,\Z)
\ar[r]^-{h}&H^{3}(\Bl_\Delta(A\! \times\! A),\Z),
}
\label{ThomII}
\end{equation}
where $\mathscr{N}_{A^{[2]}/E}$ and $\mathscr{N}_{\Bl_\Delta(A\! \times\! A)/E}$ are the normal bundles of $\pi(E)$ in $A^{[2]}$ and of $E$ in $\Bl_\Delta(A\! \times\! A)$, respectively.
By the proof of Theorem 7.31 of~\cite{Voisin}, the map $h$ is injective with image in $H^{3}(\Bl_\Delta(A\! \times\! A),\Z)$ given by $j_*b_{|E}^{*}(H^{1}(\Delta,\Z))$. Hence Diagram (\ref{ThomII}) shows that $g$ is also injective and has image $\pi_{*}j_*b_{|E}^{*}(H^{1}(\Delta,\Z))$ in $H^{3}(A^{[2]},\Z)$.
We obtain:
\begin{equation}
\xymatrix{ 0\ar[r]&H^{3}(A^{[2]},U,\Z)\ar[r]^g&H^{3}(A^{[2]},\Z)\ar[r]&H^{3}(U,\Z)}.
\label{exactutile}
\end{equation}
From Thom's isomorphism, we know that $H^{4}(A^{[2]},U,\Z)$ is torsion free, hence  $\tors \coker g= \tors H^{3}(U,\Z)$, 
where $\tors$ means the torsion part of the groups.
It follows,
by Lemma~\ref{3} that:
$$\tors \coker g=(\Z/2\Z)^{\oplus b_1(A)}.$$
Since $\rk\pi_{*}j_*b_{|E}^{*}(H^{1}(\Delta,\Z))=b_1(A)$ it follows that all the elements of $\pi_{*}j_*b_{|E}^{*}(H^{1}(\Delta,\Z))$ are divisible by 2 in $H^{3}(A^{[2]},\Z)$.
\end{proof}
Now it remains to prove that $\pi_{*}b^{*}(H^{3}(A\times A,\Z))\oplus \frac{1}{2}\pi_{*}j_*b_{|E}^{*}(H^{1}(\Delta,\Z))$ is primitive in $H^{3}(A^{[2]},\Z)$.
This comes from the fact that all elements in $\pi_{*}b^{*}(H^{3}(\Bl_\Delta(A\! \times\! A),\Z)^{\mathfrak{S}_2})$ are divisible by 2, so the relations (\ref{pietpi}) on $\pi_*$ and $\pi^*$ impose the above sum to be primitive.  

More detailed, let $y\in \pi_*b^{*}(H^{3}(A\times A,\Z))$ and $z\in \pi_*j_*b_{|E}^{*}(H^{1}(\Delta,\Z))$. It is enough to show that if $\frac{y+z}{2}\in H^{3}(A^{[2]},\Z)$, then $\frac{y}{2}\in H^{3}(A^{[2]},\Z)$ and $\frac{z}{2}\in H^{3}(A^{[2]},\Z)$.
As we have seen, we can write $y=\frac{1}{2}\pi_*(w+\sigma_{2}^{*}(w))$, with $w\in b^{*}(H^{3}(A\times A,\Z))$ and $z=\frac{1}{2}\pi_*(z')$, with $z'\in j_*b_{|E}^{*}(H^{1}(\Delta,\Z))$. If $$\frac{\frac{1}{2}\pi_*(w+\sigma_{2}^{*}(w))+\frac{1}{2}\pi_*(z')}{2}\in H^{3}(A^{[2]},\Z)$$ then taking the image by $\pi^{*}$ of this element, we obtain $$\frac{w+\sigma_{2}^{*}(w)+z'}{2}\in H^{3}(\Bl_\Delta(A\! \times\! A),\Z).$$ 
Hence from (\ref{voisin1}), necessarily $z'$ is divisible by 2. It follows that $z$ is divisible by 2 and so $y$.

This finishes the proof of (i) of Proposition~\ref{Alpha35}.
\subsection{The fifth cohomology group}\label{d5}
Now we prove (ii) of Proposition~\ref{Alpha35}.
We will need two basic properties from lattice theory that we recall here and can be found for example in
Chapter 8.2.1 of \cite{Dolgachev}.

Let $M$ be a lattice. Let $L\subset M$ be a sublattice of the same rank. 
Then 
\begin{equation}
|M\DP L|=\sqrt{\frac{\discr L}{\discr M}}.
\label{squareDiscr}
\end{equation}
We recall that the \emph{discriminant} $\discr L$ of a lattice $L$
is defined by the absolute value of the determinant of the matrix of its bilinear form. 

If $M$ is unimodular and $L\subset M$ is a primitive embedding, then 
\begin{equation}
\discr L = \discr L^\perp.
\label{discrOrthPrim}
\end{equation}

By Theorem 7.31 of~\cite{Voisin}, we have:
\begin{equation}
H^{5}(\Bl_\Delta(A\! \times\! A),\Z)=b^*(H^{5}(A\times A,\Z))\oplus j_*b_{|E}^{*}(H^{3}(\Delta,\Z)).
\label{voisin2}
\end{equation}
It follows that $$H^{5}(A^{[2]},\Z)\supset \pi_{*}(b^*(H^{5}(A\times A,\Z)))\oplus \pi_{*}j_*b_{|E}^{*}(H^{3}(\Delta,\Z)).$$
As before, by looking at the Künneth formula, $\pi_{*}(b^*(H^{5}(A\times A,\Z)))$ is primitive in $H^{5}(A^{[2]},\Z)$.
Moreover, by (\ref{voisin2}):
$$
l_2^5(\Bl_\Delta(A\! \times\! A))=\rk H^{5}(A\times A,\Z)^{\mathfrak{S}_2}=b_1(A)(b_2(A)+1),
$$
and 
\begin{equation}
l_{1,+}^5(\Bl_\Delta(A\! \times\! A))=\rk H^{3}(\Delta,\Z)^{\mathfrak{S}_2}=b_1(A),\ \text{ and }\ l_{1,-}^5(\Bl_\Delta(A\! \times\! A))=0.
\label{l2}
\end{equation}
\begin{lemma}
The lattice $\pi_{*}(H^{3}(\Bl_\Delta(A\! \times\! A),\Z)\oplus H^{5}(\Bl_\Delta(A\! \times\! A),\Z))$ has discriminant $2^{2b_1(A)}$.
\end{lemma}
\begin{proof}
By Proposition~\ref{sarti} (ii):
\footnotesize
\begin{gather*}
\frac{H^{3}(\Bl_\Delta(A\! \times\! A),\Z)\oplus H^{5}(\Bl_\Delta(A\! \times\! A),\Z)}{H^{3}(\Bl_\Delta(A\! \times\! A),\Z)^{\mathfrak{S}_2}\oplus H^{5}(\Bl_\Delta(A\! \times\! A),\Z)^{\mathfrak{S}_2}\oplus \left(H^{3}(\Bl_\Delta(A\! \times\! A),\Z)^{\mathfrak{S}_2}\oplus H^{5}(\Bl_\Delta(A\! \times\! A),\Z)^{\mathfrak{S}_2}\right)^\bot}\hspace{30pt}
\\\hspace{92pt}= \left(\Z/2\Z\right)^{\oplus\left(l_2^3(\Bl_\Delta(A\! \times\! A))+l_2^5(\Bl_\Delta(A\! \times\! A))\right)}.
\end{gather*}
\normalsize
Since $H^{3}(\Bl_\Delta(A\! \times\! A),\Z)\oplus H^{5}(\Bl_\Delta(A\! \times\! A),\Z)$ is a unimodular lattice, 
it follows from (\ref{discrOrthPrim}) and (\ref{squareDiscr}) that 
$$\discr \left[H^{3}(\Bl_\Delta(A\! \times\! A),\Z)^{\mathfrak{S}_2}\oplus H^{5}(\Bl_\Delta(A\! \times\! A),\Z)^{\mathfrak{S}_2}\right]=2^{l_2^3(\Bl_\Delta(A\! \times\! A))+l_2^5(\Bl_\Delta(A\! \times\! A))}.$$
Then by Proposition~\ref{commut},
\begin{gather*}
\discr \pi_{*}(H^{3}(\Bl_\Delta(A\! \times\! A),\Z)^{\mathfrak{S}_2}\oplus H^{5}(\Bl_\Delta(A\! \times\! A),\Z)^{\mathfrak{S}_2})\hspace{55pt}\\
\hspace{62pt}=2^{l_2^3(\Bl_\Delta(A\! \times\! A))+l_2^5(\Bl_\Delta(A\! \times\! A))+\rk \left[H^{3}(\Bl_\Delta(A\! \times\! A),\Z)^{\mathfrak{S}_2}\oplus H^{5}(\Bl_\Delta(A\! \times\! A),\Z)^{\mathfrak{S}_2}\right]}.
\end{gather*}
Then by Proposition~\ref{sarti} (i):
\begin{equation}\label{odddiscr}
\begin{array}{rl}
\discr \pi_{*}(H^{3}(\Bl_\Delta(A\! \times\! A),\Z)^{\mathfrak{S}_2}\!\!\!\!&\oplus\; H^{5}(\Bl_\Delta(A\! \times\! A),\Z)^{\mathfrak{S}_2})\hspace{55pt}\\
\hspace{70pt} =\!\! & 2^{2\left( l_2^3(\Bl_\Delta(A\! \times\! A))+ l_2^5(\Bl_\Delta(A\! \times\! A))\right)+l_{1,+}^3(\Bl_\Delta(A\! \times\! A))+l_{1,+}^5(\Bl_\Delta(A\! \times\! A))}.
\end{array}
\end{equation}
By Remark~\ref{x+ix} and since $\pi_*(x+\iota^*(x))=2\pi_*(x)$, we have:
$$\frac{\pi_*(H^3(\Bl_\Delta(A\! \times\! A),\Z)\oplus H^5(\Bl_\Delta(A\! \times\! A),\Z))}{\pi_*(H^3(\Bl_\Delta(A\! \times\! A),\Z)^{\mathfrak{S}_2}\oplus H^5(\Bl_\Delta(A\! \times\! A),\Z)^{\mathfrak{S}_2})}=\left(\frac{\Z}{2\Z}\right)^{\oplus \left(l_2^3(\Bl_\Delta(A\! \times\! A))+l_2^5(\Bl_\Delta(A\! \times\! A))\right)}.$$
Then by (\ref{odddiscr}), (\ref{squareDiscr}), (\ref{l1}) and (\ref{l2}):
\begin{equation*}
\discr \pi_{*}(H^{3}(\Bl_\Delta(A\! \times\! A),\Z)\oplus H^{5}(\Bl_\Delta(A\! \times\! A),\Z))=2^{l_{1,+}^3(\Bl_\Delta(A\! \times\! A))+l_{1,+}^5(\Bl_\Delta(A\! \times\! A))}=2^{2b_1(A)}.
\qedhere
\end{equation*}
\end{proof}
The lattice $H^{3}(A^{[2]},\Z)\oplus H^{5}(A^{[2]},\Z)$ is unimodular. Hence by (\ref{squareDiscr}):
$$\frac{H^{3}(A^{[2]},\Z)\oplus H^{5}(A^{[2]},\Z)}{\pi_{*}(H^{3}(\Bl_\Delta(A\! \times\! A),\Z)\oplus H^{5}(\Bl_\Delta(A\! \times\! A),\Z))}=\left(\frac{\Z}{2\Z}\right)^{\oplus b_1(A)}.$$
However, from the last section, we know that $\frac{H^{3}(A^{[2]},\Z)}{\pi_{*}(H^{3}(\Bl_\Delta(A\! \times\! A),\Z))}=(\Z/2\Z)^{\oplus b_1(A)}$.
It follows that
$$
\frac{H^{5}(A^{[2]},\Z)}{\pi_{*}(H^{5}(\Bl_\Delta(A\! \times\! A),\Z))}={0},
$$
which proves (ii) of Proposition \ref{Alpha35}.

\section[Nakajima operators for Hilbert schemes of points on surfaces]{Nakajima operators for Hilbert schemes of points on surfaces%
\sectionmark{Nakajima operators}}
\sectionmark{Nakajima operators}
\label{Section_Hilbert}
Let $A$ be a smooth projective complex surface. 
Let $A\hilb{n}$ the Hilbert scheme of $n$ points on the surface, \ie the moduli space of finite subschemes of $A$ of length $n$.
$A\hilb{n}$ is again smooth and projective of dimension $2n$, cf.~\cite{Fogarty}. 
Their rational cohomology can be described in terms of Nakajima's~\cite{Nakajima} operators. First consider the direct sum
$$
\H \defIs  \bigoplus_{n=0}^{\infty} H^*(A\hilb{n},\Q).
$$
This space is bigraded by cohomological \emph{degree} and the \emph{weight}, which is given by the number of points $n$. The unit element in $H^0(A\hilb{0},\Q) \cong \Q$ is denoted by $\vac$, called the \emph{vacuum}.
\begin{defipro}
There are linear operators $\q_m(a)$, for each $m\geq 1$ and $a \in H^*(A,\Q)$, acting on $\H,$ which have the following properties: They depend linearly on $a$, and if $a\in H^k(A,\Q)$ is homogeneous, the operator $\q_{m}(a)$ is bihomogeneous of degree $k+2(m-1)$ and weight $m$:
$$
\q_{m}(a) : H^l(A\hilb{n}) \rightarrow H^{l+k+2(m-1)}(A\hilb{n+m}).
$$
To construct them, first define incidence varieties $\mathcal Z_m\subset A\hilb{n}\times A\times A\hilb{n+m}$ by
$$
\mathcal Z_m \defIs  \left\{(\xi,x,\xi')\, |\, \xi\subset\xi',\, \supp(\xi') -\supp(\xi) = mx \right\}.
$$
Then $\q_m(a)(\beta) $ is defined as the Poincar\'e dual of 
$$
\pr_{3*}\left( \left(\pr_2^*(\alpha)\cdot \pr_1^*(\beta)\right) \cap [\mathcal Z_m] \right).
$$
\end{defipro}
Consider now the superalgebra generated by the $\q_m(a)$. 
Every element in $\H$ can be decomposed uniquely as a linear combination of products of operators $\q_{m}(a)$, acting on the vacuum. 
In other words, the $\q_m(a)$ generate $\H$ and there are no algebraic relations between them (except the linearity in $a$ and the super-commutativity).

\begin{definition}
To give the cup product structure of $\H$, define operators $\G(a)$ for $a \in H^*(A)$. Let $\Xi_n \subset A\hilb{n}\times A$ be the universal subscheme. Then the action of $\G(a)$ on $H^*(A\hilb{n})$ is multiplication with the class
$$
\pr_{1*}\left( \ch(\mathcal{O}_{\Xi_n})\cdot \pr_2^*(\td(A)\cdot a) \right) \in H^*(A\hilb{n}).
$$
For $a \in H^k(A)$, we define $\G_i(a)$ as the component of $\G(a)$ of cohomological degree $k+2i$. A differential operator $\mathfrak{d}$ is given by $\G_1(1)$. It means multiplication with the first Chern class of the tautological sheaf $\pr_{1*}\left( \mathcal{O}_{\Xi_n}\right)$.
\end{definition}
In~\cite{LehnSorger} and~\cite{LiQinWang} we find various commutation relations between these operators, that allow to determine all multiplications in the cohomology of the Hilbert scheme. First of all, if $X$ and $Y$ are operators of degrees $d$ and $d'$, their commutator is defined in the super sense: 
$$
[X,Y] \defIs  XY - (-1)^{dd'}YX.
$$
The integral on $A\hilb{n}$ induces a non-degenerate bilinar form on $\H$: for classes $\alpha,\,\beta\in H^*(A\hilb{n})$ it is given by
$$
(\alpha,\beta)_{A\hilb{n}} \defIs   \int_{A\hilb{n}}\alpha\cdot\beta.
$$
If $X$ is a homogeneous linear operator of degree $d$ and weight $m$, acting on $\H$, define its adjoint $X^\dagger$ by
$$
(X(\alpha),\beta)_{A\hilb{n+m}}  = (-1)^{d|\alpha|}( \alpha , X^\dagger (\beta))_{A\hilb{n}}.
$$
We put $\q_0(a) \defIs 0$ and for $m<0$, $\q_m(a) \defIs  (-1)^m \q_{-m}(a)^\dagger$. Note that, for all $m\in\Z$, the bidegree of $\q_m(a)$ is $(m,|a| + 2(|m|-1))$. If $m$ is positive, $\q_m$ is called a \emph{creation operator}, otherwise it is called annihiliation operator. Now define
$$
\mathfrak{L}_m(a) \defIs  \left\{ 
\begin{array}{ll}
 \tfrac{1}{2}\sum\limits_{k\in\Z}\sum\limits_{i}\q_k( a_{(1)})\q_{m-k}( a_{(2)}), & \text{ if } m\neq 0, \vspace{4mm}\\
 \sum\limits_{k>0}\sum\limits_{i}\q_k( a_{(1)})\q_{-k}( a_{(2)}), & \text{ if } m= 0. \\
\end{array}
\right.
$$
where $\sum_i a_{(1)}\otimes  a_{(2)}$ is the push-forward of $a$ along the diagonal $\tau_2 :A \rightarrow A\times A$ (in Sweedler notation).
\begin{lemma}\cite[Thm.~2.16]{LiQinWang} Denote $K_A\in H^2(A,\Q)$ the class of the canonical divisor. We have:
\label{commutators}
\begin{align}
[\q_m(a), \q_n(b)] &= m\cdot \delta_{m+n} \cdot \int_A ab \\
\label{qLcommute}
[\mathfrak{L}_m(a),\q_n(b)] &= -n\cdot \q_{m+n}(ab) \\
\label{DiffNaka}
[\mathfrak{d},\q_m(a)] &= m \cdot \mathfrak{L}_m(a) + \tfrac{m(|m|-1)}{2} \q_m(K_A a) \\
[\G_k(a),\q_1(b)] &= \tfrac{1}{k!} \ad(\mathfrak{d})^k(\q_1(a b) ) 
\end{align}
\end{lemma}
\begin{remark}\label{HRep}
Note (cf.~\cite[Thm.~3.8]{LehnSorger}) that (\ref{qLcommute}) together with (\ref{DiffNaka}) imply that 
\begin{equation}\label{NakaDel}
\q_{m+1}(a) = \tfrac{(-1)^m}{m!}(\ad\q')^m\left(\q_1(a)\right),
\end{equation}
so there are two ways of writing an element of $\H$: As a linear combination of products of creation operators $\q_m(a)$ or as a linear combination of products of the operators $\mathfrak{d}$ and $\q_1(a)$. This second representation is more suitable for computing cup-products, but not faithful. 
Equations (\ref{DiffNaka}) and (\ref{NakaDel}) permit now to switch between the two representations.
\end{remark}

\begin{remark}
We adopted the notation from~\cite{LiQinWang}, which differs from the conventions in~\cite{LehnSorger}. Here is part of a dictionary:
\begin{center}
\begin{tabular}{c|c} 
Notation from~\cite{LiQinWang} & Notation from~\cite{LehnSorger} \\\hline
operator of weight $w$ and degree $d$ & operator of weight $w$ and degree $d-2w$\\
$\q_m(a) $ & $\p_{-m}(a)$ \\
$ \mathfrak{L}_m(a) $ & $ - L_{-m}(a)$ \\
$\mathfrak{G}(a)$ & $a\hilb{\bullet}$\\
$ \mathfrak{d} $ & $ \partial $ \\
$\tau_{2*}(a)$& $-\Delta(a)$
\end{tabular}
\end{center}
\end{remark}

By sending a subscheme in $A$ to its support, we define a morphism
\begin{equation}\label{HilbertChow}
\rho : A\hilb{n} \longrightarrow \Sym^n(A),
\end{equation}
called the Hilbert--Chow morphism. The cohomology of $\Sym^n(A)$ is given by elements of the $n$-fold tensor power of $H^*(A)$ that are invariant under the action of the group of permutations $\mathfrak{S}_n$. A class in $H^*(A\hilb{n},\Q)$ which can be written using only the operators $\q_1(a)$ of weight $1$ comes from a pullback along $\rho$:
\begin{equation}
\label{qSym}
\q_1(b_1)\cdots \q_1(b_n)\vac = \rho^*\left( \sum_{\pi\in\mathfrak{S}_n } \pm b_{\pi(1)}\otimes\ldots\otimes b_{\pi(n)} \right), \quad b_i\in H^*(A,\Q),
\end{equation}
where signs arise from permuting factors of odd degrees. In particular,
\begin{gather} \label{q0primitive}
\frac{1}{n!}\q_1(b)^n \vac = \rho^*\big( b\otimes \ldots \otimes b\big),
\\ \label{q1primitive}
\frac{1}{(n-1)!}\q_1(b)\q^{n-1}\vac =\rho^*\Big( b\otimes\! 1\!\otimes\ldots\otimes\! 1\; + \;\ldots\; +\; 1\!\otimes\ldots\otimes\! 1\! \otimes b\Big) .
\end{gather}


\section{On integral cohomology of Hilbert schemes}\label{basisHilb2}

For the study of integral cohomology, first note that if $a \in H^*(A,\Z)$ is an integral class, then $\q_{m}(a) $ maps integral classes to integral classes. Operators satisfying this property are called integral. Qin and Wang studied them in~\cite{QinWang}. We need the following results:

\begin{lemma} \cite[Lem.~3.3, 3.6 and Thm.~4.5]{QinWang}\label{IntegralOperators}
The operators $\frac{1}{n!}\q_1(1)^n$ and $\frac{1}{2}\q_2(1) $ are integral.
Let $b\in H^2(A,\Z)$ be monodromy equivalent to a divisor. Then the operator $\frac{1}{2}\q_1(b)^2 - \frac{1}{2}\q_2(b)$ is integral. 
\end{lemma}
\begin{remark}
Qin and Wang~\cite[Thm. 1.1 et seq.]{QinWang} conjecture that this works even without the restriction on $b\in H^2(A,\Z)$. 
\end{remark}

\begin{corollary} \label{IntegralOperatorsTorus}
If $A$ is a torus, the operator $\frac{1}{2}\q_1(b)^2 - \frac{1}{2}\q_2(b)$ is integral for all $b\in H^2(A,\Z)$. 
\end{corollary}
\begin{proof}
The Nakajima operators are preserved under deformations of $A$. 
Moreover, 
the image $\Mon(A)$ of the monodromy representation on $H^2(A,\Z)$ is given by $O^{+,+}(H^2(A,\Z))$, the group of isometries on $H^2(A,\Z)$ preserving the orientation of the negative and positive definite part of $H^2(A,\R)$.
Indeed, by the last remark in \cite{Borcea}, the subgroup $\Diff(A)$ of $O(H^2(A,\Z))$ induced by the diffeomorphisms of $A$ is equal to  $O^{+,+}(H^2(A,\Z))$. 
Hence $\Mon(A)\subset \Diff(A) = O^{+,+}(H^2(A,\Z))$. Furthermore, by Theorem 1 and 2 in Section 4 and 5 of \cite{Shioda}, the moduli space of marked complex tori have 4 connected components. 
It follows that necessarily, $\Mon(A)$ has at most index 4 in $O(H^2(A,\Z))$. So $\Mon(A)= O^{+,+}(H^2(A,\Z))$.



Suppose now that the N\'eron-Severi group $\NS(A)$ contains a copy of the hyperbolic lattice $U$ (such $A$ exist).
Let us denote $H^2(A,\Z)=U_1\oplus U_2\oplus U_3$ with $\NS(A)=U_1$ and for all $i\in\{1,2,3\}$, $U_i$ is isometric to $U$.
We consider two isometries in $O^{+,+}(H^2(A,\Z))$, $\varphi_2$ and $\varphi_3$, defined in the following way:
$\varphi_2$ exchanges $U_1$ and $U_2$ and acts as $-\id$ on $U_3$ and $\varphi_3$ exchanges $U_1$ and $U_3$ and acts as $-\id$ on $U_2$.
Using these two isometries, all elements of $U_2$ and $U_3$ are monodromy equivalent to a divisor.
Then Lemma~\ref{IntegralOperators} establishes the corollary for that particular $A$. 
Now, since all tori are equivalent by deformation, a general torus can always be deformed to our special $A$. Since the integrality of an operator is a topological invariant, $\frac{1}{2}\q_1(b)^2 - \frac{1}{2}\q_2(b)$ remains integral for all $b\in H^2(A,\Z)$.
\end{proof}

\begin{proposition} Assume that $H^*(A,\Z)$ is free of torsion.
We are using Notation~\ref{TorusClasses}. Denote $b_i^*\in H^2(A,\Z)$ the dual element to $b_i$. 
Modulo torsion, the following classes form a basis of $H^2(A\hilb{n},\Z)$:
\begin{itemize}
 \item[] $\frac{1}{(n-1)!}\q_{1}(b_{i})\q_{1}(1)^{n-1}\vac = \G_0(b_i) 1$,
 \item[] $ \frac{1}{(n-2)!}\q_{1}(a_{i})\q_{1}(a_{j})\q_{1}(1)^{n-2}\vac = \G_0(a_i) \G_0(a_j)1,\  i < j$, 
 \item[] $ \frac{1}{2(n-2)!}\q_{2}(1) \q_{1}(1)^{n-2}\vac$. We denote this class by $\delta$.
\end{itemize}
Their respective duals in $H^{2n-2}(A\hilb{n},\Z)$ are given by
\begin{itemize}
 \item[] $\q_{1}(b_{i}^*)\q_{1}(x)^{n-1}\vac$,
 \item[] $\q_{1}(a_{j}^*)\q_{1}(a_{i}^*)\q_{1}(x)^{n-2}\vac,\  i < j$,
 \item[] $\q_2(x)\q_{1}(x)^{n-2} \vac$.
\end{itemize}
\end{proposition}
\begin{proof} It is clear from the above lemma that these classes are all integral.
G\"ottsche's formula~\cite[p.~35]{Gottsche} gives the Betti numbers of $A\hilb{n}$ in terms of the Betti numbers of $A$: 
$h^1(A\hilb{n}) = h^1(A)$, and $h^2(A\hilb{n}) = h^2(A)+ \frac{h^1(A)(h^1(A)-1)}{2} + 1$. It follows that the given classes span a lattice of full rank.

Next we have to show that the intersection matrix between these classes is in fact the identity matrix. Most of the entries can be computed easily using the simplification from (\ref{qSym}). For products involving $\delta$ (this is the action of $\mathfrak{d}$) or its dual, first observe that $\mathfrak{d}\q_1(x)^m\vac = 0 $ and $ \mathfrak{L}_1(a)\q_1(x)^m\vac =0$ for every class $a$ of degree at least 1. Then compute:

\begin{gather*}
\delta \cdot\q_2(x)\q_{1}(x)^{n-2} \vac = \mathfrak{d}\q_2(x)\q_{1}(x)^{n-2} \vac = 2 \mathfrak{L}_2(x) \q_{1}(x)^{n-2} \vac = \q_{1}(x)^{n}\vac,
\\
\mathfrak{d}\q_{1}(b_{i}^*)\q_{1}(x)^{n-1}\vac =  \mathfrak{L}_1(b_i^*) \q_{1}(x)^{n-1} \vac = 0,
\\
\mathfrak{d}\q_{1}(a_{j}^*)\q_{1}(a_{i}^*)\q_{1}(x)^{n-2}\vac = \left(\mathfrak{L}_1(a_j^*) +\q_{1}(a_{j}^*)\mathfrak{d}\right)\q_{1}(a_{i}^*)\q_{1}(x)^{n-2}\vac = 
  \\ =\left(-\q_1(a_i^*)\mathfrak{L}_1(a_j^*) + \q_{1}(a_{j}^*)\mathfrak{L}_1(a_i^*)\right)\q_{1}(x)^{n-2}\vac  = 0,
\\
\G_0(b_i)\q_2(x)\q_{1}(x)^{n-2} \vac = 0, 
\\
\G_0(a_i)\G_0(a_j)\q_2(x)\q_{1}(x)^{n-2} \vac = 0.
\end{gather*}

\end{proof}

\begin{remark}
If $A$ is a complex torus, a theorem of Markman~\cite{Markman} ensures that $H^*(A\hilb{n},\Z)$ is torsion free.
\end{remark} 

\begin{proposition} \label{A2Basis}
Let $A$ be a complex abelian surface. Using Notation~\ref{TorusClasses}, a basis of $H^*(A\hilb{2},\Z)$ is given by the following classes.
\begin{center}
\begin{tabular}{c|c|l|l}
 degree & Betti number & class & multiplication with class \\\hline
 0 & 1 & $\frac{1}{2}\q_1(1)^2\vac$ & $\id$ \\ \hline
 1 & 4 &  $\q_1(1)\q_1(a_i)\vac$ & $\G_0(a_i)$ \\ \hline
 2 & 13 & $\frac{1}{2}\q_2(1)\vac$ & $\d$ \\ 
   &  & $\q_1(a_i)\q_1(a_j)\vac$ for $i<j$ & $\G_0(a_i)\G_0(a_j)$ \\
   &  & $\q_1(1)\q_1(b_i)\vac$ & $\G_0(b_i)$ \\\hline
 3 & 32 & $\frac{1}{2}\q_2(a_i)\vac$  & $-\G_1(a_i) $ \\
   &  & $\q_1(a_i)\q_1(b_j)\vac$ & $\G_0(a_i)\G_0(b_j)$ \\ 
   &  & $\q_1(1)\q_1(a^*_i)\vac$ & $\G_0(a^*_i)$ \\\hline
 4 & 44 & $\left(\frac{1}{2}\q_1(b_i)^2-\frac{1}{2}\q_2(b_i)\right)\vac$ & $\frac{1}{2} \G_0(b_i)^2 + \G_1(b_i) $ \\
   &  & $\q_1(a_i)\q_1(a^*_j)\vac$ & $\G_0(a_i)\G_0(a^*_j)$ \\
   &  & $ \q_1(b_i)\q_1(b_j)\vac$ for $i\leq j$ &  $\G_0(b_i)\G_0(b_j)$ \\\hline
 5 & 32 & $\q_2(a^*_i)\vac$ & $-2\G_1(a^*_i)$ \\
   &  & $\q_1(a^*_i)\q_1(b_j)\vac$ & $ \G_0(a^*_i)\G_0(b_j)$ \\
   &  & $\q_1(a_i)\q_1(x)\vac$ & $\G_0(a_i)\G_0(x)$ \\\hline
 6 & 13 & $\q_2(x)\vac$ & $-2\G_1(x)$ \\
   &  & $\q_1(a^*_i)\q_1(a^*_j)\vac$ for $i<j$ & $\G_0(a^*_i)\G_0(a^*_j)$ \\
   &  & $\q_1(b_i)\q_1(x)\vac$ & $ \G_0(b_i)\G_0(x)$ \\\hline
 7 & 4 & $\q_1(a^*_i)\q_1(x)\vac$ & $\G_0(a^*_i)\G_0(x) $ \\\hline
 8 & 1 & $\q_1(x)^2\vac$ & $\G_0(x)^2$ 
\end{tabular}
\end{center}
\begin{proof}
The Betti numbers come from G\"ottsche's formula~\cite{Gottsche}.
One computes the intersection matrix of all classes under the Poincar\'e duality pairing and finds that it is unimodular. 
So it remains to show that all these classes are integral. By Lemma~\ref{IntegralOperators} this is clear for all classes except 
those of the form $\frac{1}{2}\q_2(a_i)\vac \in H^3(A\hilb{2},\Z)$.

Evaluating the Poincar\'e duality pairing between degrees 3 and 5 gives:
\begin{gather*}
 \q_2(a_i)\vac \cdot \q_2(a^*_i)\vac = 2, \\
 \q_1(a_i)\q_1(b_j)\vac \cdot  \q_1(a^*_i)\q_1(b^*_j)\vac = 1, \\
 \q_1(1)\q_1(a^*_i)\vac \cdot \q_1(x)\q_1(a_i)\vac = 1,
\end{gather*}
while the other pairings vanish. Therefore, one of $\q_2(a_i)\vac$ and $\q_2(a^*_i)\vac$ must be divisible by $2$. 
With the considerations from Section~\ref{OddHilb2} in mind, we can interpret $\q_2(a_i)\vac\in H^3(A\hilb{2},\Z)$ and $\q_2(a^*_i)\vac\in H^5(A\hilb{2},\Z)$ as classes concentrated on the exceptional divisor, that is, as elements of $\pi_* j_*H^*(E,\Z)$. Indeed,
the pushforward of a class $a\otimes 1 \in H^{k}(E,\Z)$ is given by 
$$
\pi_* j_*(a\otimes 1) = \q_2(a)\vac \in H^{k+2}(A\hilb{n},\Z).
$$
When pushing forward to the Hilbert scheme, the only possibility to get a factor $2$ is in degree $3$, by Proposition~\ref{Alpha35}. 
\end{proof}

\end{proposition}

\section[Generalized Kummer varieties and the morphism to the Hilbert scheme]{Cohomology of generalized Kummer varieties via Hilbert scheme cohomology %
\sectionmark{Generalized Kummer varieties}}
\sectionmark{Generalized Kummer varieties}
\label{Section_GeneralKummer}
\begin{definition}
Let $A$ be a complex projective torus of dimension $2$ and $A\hilb{n}$, $n\geq 1$, the corresponding Hilbert scheme of points. Denote $\Sigma : A\hilb{n} \rightarrow A$ the summation morphism, a smooth submersion that factorizes via (\ref{HilbertChow}) the Hilbert--Chow morphism $A\hilb{n}\stackrel{\rho}{\rightarrow}\Sym^n(A)\stackrel{\sigma}{\rightarrow} A$. Then the generalized Kummer variety $\kum{A}{n-1}$ is defined as the fiber over $0$:
\begin{equation}\label{square}
\begin{CD}
\kum{A}{n-1} @>\theta >> A\hilb{n}\\
@VVV @VV\Sigma V\\
\{0\} @> >> A
\end{CD}
\end{equation}
\end{definition}
\begin{theorem}~\cite[Theorem 2]{Spanier}\label{torsion}
The cohomology of the generalized Kummer, $H^*(\kum{A}{n-1},\Z)$, is torsion free. 
\end{theorem}
Our first objective is to collect some information about the pullback diagram~(\ref{square}). 
We make use of Notation~\ref{TorusClasses}.

\begin{proposition}\label{KummerClass}
Set $\alpha_i \defIs  \frac{1}{(n-1)!}\kq_{1}(1)^{n-1}\kq_1(a_i)\vac = \G_0(a_i)1$. Then the corresponding class of 
$\kum{A}{n-1}$ in $H^4(A\hilb{n},\Z)$ is given by
$$
[\kum{A}{n-1}]=\alpha_1\cdot\alpha_2\cdot\alpha_3\cdot\alpha_4.
$$ 
\end{proposition}
\begin{proof}
Since the generalized Kummer variety is the fiber over a point, its 
class must be the pullback of $x\in H^4(A)$ under $\Sigma$. But $\Sigma^* (x) = \Sigma^*(a_1)\cdot \Sigma^*(a_2)\cdot \Sigma^*(a_3)\cdot \Sigma^*(a_4)$, so we have to verify that $\Sigma^* (a_i) = \alpha_i$. To do this, we want to use the decomposition $\Sigma = \sigma\rho$.
The pullback along $\sigma$ of a class $a\in H^1(A,\Q)$ on $H^1(\Sym^n(A),\Q)$ 
is given by $a\otimes 1\otimes \cdots\otimes 1 + \ldots + 1\otimes \cdots\otimes 1\otimes a$. It follows from (\ref{q1primitive}) that $\Sigma^* (a_i) = \frac{1}{(n-1)!}\kq_{1}(1)^{n-1}\kq_1(a_i)\vac $.
\end{proof}
The morphism $\theta$ induces a homomorphism of graded rings
\begin{equation}
\theta^* :H^*(A\hilb{n})\longrightarrow H^*(\kum{A}{n-1})
\end{equation}
and by the projection formula, we have
\begin{equation}\label{projectionFormula}
\theta_*\theta^*(\alpha)  = [\kum{A}{n-1}]\cdot\alpha.
\end{equation}

\begin{lemma}\label{petitlemmeenplus}
 Let $\beta\in H^*(K_{n-1}(A),\Q)$. Then there is a class $B\in H^{*}(A\hilb{n},\Q)$ such that 
 $$\theta_*(\beta)=\frac{1}{n^4}B\cdot [\kum{A}{n-1}].$$
\end{lemma}
\begin{proof}
For a point $a\in A$, we denote by $t_a$ the morphism on $A\hilb{n}$ induced by the translation by $a$.
Then we consider the morphism $\Theta :\kum{A}{n-1}\times A \longrightarrow A\hilb{n}$ defined by $\Theta(\xi,a)=t_a(\theta(\xi))$. It fits in a pullback diagram
\begin{equation}
\begin{CD}
\kum{A}{n-1}\times A @>\Theta >> A\hilb{n}\\
@VV\pr_2V @VV\Sigma V\\
A @> n\cdot >> A
\end{CD}
\end{equation}
that realizes $\kum{A}{n-1}\times A$ as a $n^4$-fold covering of $A\hilb{n}$ over $A$.
Now, for $\beta\in H^*(K_{n-1}(A),\Q)$ set
$$
B:=\Theta_*(\beta\otimes 1).
$$
Then the projection formula gives
\begin{align*}
B\cdot [K_{n-1}(A)]&= \Theta_*\left(\beta\otimes 1\cdot \Theta^*[\kum{A}{n-1}]\right) \\
&=n^4 \Theta_*\left((\beta\otimes 1)\cdot  (1\otimes x)\right)\\
&=n^4 \Theta_*(\beta \otimes x)\\
&=n^4\theta_*(\beta).
\end{align*}

\end{proof}

\begin{proposition}\label{annihilator}
The kernel of $\theta^*$ is equal to the annihilator of $[\kum{A}{n-1}]$.
\end{proposition}
\begin{proof}
Assume $\alpha\in \ker(\theta^*)$. Then we have
$
[\kum{A}{n-1}]\cdot \alpha = \theta_*\theta^*(\alpha) = 0
$. 
Conversely, if $\alpha\notin \ker(\theta^*)$,
let $\beta\in H^*(\kum{A}{n-1},\Q)$ be the Poincar\'e dual of $\theta^*(\alpha)$, so $\beta\cdot \theta^*(\alpha)\neq 0$.
Then by projection formula:
$
\theta_*(\beta)\cdot \alpha\neq 0.
$
By Lemma~\ref{petitlemmeenplus}, there exists $B\in H^*(A\hilb{n},\Q)$ such that 
$B\cdot [\kum{A}{n-1}]\cdot \alpha\neq 0$. It follows that $ [\kum{A}{n-1}]\cdot \alpha\neq 0$.
\end{proof}

\begin{corollary} \label{KummerEquality}
$\theta^*(\alpha) = \theta^*(\beta)$ if and only if $[\kum{A}{n-1}]\cdot \alpha = [\kum{A}{n-1}]\cdot \beta$. 
\qed
\end{corollary}

\begin{proposition}\label{Annihideal}
The annihilator of $[\kum{A}{n-1}]$ in $H^*(A\hilb{n},\Q)$ is the ideal generated by $H^1(A\hilb{n})$. 
\end{proposition}
First, we need to recall some material on super algebras (see for instance \cite{DeligneMorgan}).
\begin{definition}
Let $V=V^{+}\oplus V^{-}$ be a super vector space and $n\geq 0$. Then the supersymmetric power $\SSym^n(V)$ of $V$ is a super vector space, given by
\begin{gather*}
\SSym^n(V) = \bigoplus_{p+q=n} \Sym^p(V^{+}) \!\otimes\! \Lambda^q(V^{-}), \\
\SSym^n(V)^{+}\! =\! \bigoplus_{\substack{p+q=n \\ q\text{ even} }} \Sym^p(V^{+}) \!\otimes\! \Lambda^q(V^{-}), \ \ 
\SSym^n(V)^{-}\! =\! \bigoplus_{\substack{p+q=n \\ q\text{ odd} }} \Sym^p(V^{+}) \!\otimes\! \Lambda^q(V^{-}).
\end{gather*}
\end{definition}
\begin{remark}
The supersymmetric power $\SSym^n (V)$ can be realized as a quotient of $V^{\otimes n}$ by an action of the symmetric group $\mathfrak S_n$. This action can be described as follows: If $\tau\in \mathfrak S_n$ is a transposition that exchanges two numbers $i<j$, then $\tau$ permutes the corresponding tensor factors in $v_1\otimes  \cdots\otimes v_n$ introducing a sign
$(-1)^{|v_i||v_j|+(|v_i|+|v_j|)\sum_{i<k<j} |v_k|}$.
\end{remark}

Now let $U$ be a vector space over $\Q$ and look at the exterior algebra $H\defIs  \Lambda^* U$. 
Since $H$ is a super vector space, we can construct the supersymmetric power $ \SSym^n( H)$.
We may identify $\SSym^n( H)$ with the space of $\mathfrak S_n$-invariants in $H^{\otimes n}$ by means of the linear projection operator
$$
\pr : H^{\otimes n} \longrightarrow \SSym^n( H) , \quad \pr = \frac{1}{n!}\sum_{\pi \in\mathfrak S_n} \pi.
$$
The multiplication in $H^{\otimes n}$ induces a multiplication on the subspace of invariants, which makes $\SSym^n( H)$ a supercommutative algebra.

Since $H$ is generated as an algebra by $U=\Lambda^1(U)\subset H$, we may define a homomorphism of algebras:
$$ s : H \longrightarrow \SSym^n( H), \quad s(u) = \pr( u \otimes 1\otimes\cdots\otimes 1)\text{ for }u\in U, $$
so $\SSym^n( H)$ becomes an algebra over $H$.
\begin{lemma}
\label{SuperFree}
The morphism $s$ turns $\SSym^n( H)$ into a free module over $H$, for $n\geq 1$.
\end{lemma}
\begin{proof}
We start with the tensor power $H^{\otimes n}$ and the ring homomorphism 
$$
\iota : H \longrightarrow H^{\otimes n},\quad h\longmapsto h\otimes 1\otimes\cdots\otimes 1
$$
that makes $H^{\otimes n}$ a free $H$-module. Note that $\pr \iota \neq s$, since $\pr$ is not a ring homomorphism.
(For example, $\pr(\iota(h))\neq s(h)$ for any nonzero $h\in\Lambda^2(U)$.)
We therefore modify the $H$-module structure of $H^{\otimes n}$:

For some $u\in U$, denote $u^{(i)} \defIs  1^{\otimes i-1}\otimes u\otimes 1^{\otimes n-i+1} \in H^{\otimes n}$. Then $H^{\otimes n}$ is generated as a $k$-algebra by the elements $\{u^{(i)}\,,\,u\in U\}$. Now consider the ring automorphism
$$
\sigma : H^{\otimes n} \longrightarrow H^{\otimes n}, \quad u^{(1)} \longmapsto u^{(1)} +u^{(2)} + \ldots + u^{(n)}, \quad
u^{(i)} \longmapsto u^{(i)} \text{ for } i>1.
$$
Then we have $\sigma\iota = s$ on $\SSym^n( H)$. On the other hand, if $\{b_i\}$ is a $k$-basis of $V$, then $\{b_i^{(j)},\,j>1\} $ is a $\iota$-basis for $H^{\otimes n}$, and $\{\sigma(b_i^{(j)})\}$ is a $\sigma\iota$-basis for $H^{\otimes n}$.
Now if we project the basis elements, we get a set $\{\pr(\sigma(b_i^{(j)}))\}$ that spans $\SSym^n( H)$. Eliminating linear dependent vectors (this is possible over the rationals), we get a $s$-basis of $\SSym^n( H)$.
\end{proof}

\begin{proof}[Proof of Proposition \ref{Annihideal}]
Set $H=H^*(A,\Q)\cong \Lambda^*(H^1(A,\Q))$ and consider the exact sequence of $H$-modules
$$
0 \longrightarrow 
J
\longrightarrow H \stackrel{x\cdot}{\longrightarrow} H.
$$
It is clear that $J$ is the ideal in $H$ generated by $H^{1}(A,\Q)$. 
Now denote $J^{(n)}$ the ideal generated by $H^1(\Sym^n(A),\Q)$ in $H^*(\Sym^n(A),\Q)\cong\SSym^n(H)$.
By the freeness result of Lemma~\ref{SuperFree}, tensoring with $\SSym^n(H)$ yields another exact sequence of $H$-modules
$$
0 \longrightarrow {J}^{(n)} \longrightarrow \SSym^n(H) \xrightarrow{\sigma(x)\cdot} \SSym^n(H).
$$
Now let $\mathfrak{H}$ be the operator algebra spanned by products of $\mathfrak d$ and $\q_1(a)$ for $a\in H^*(A)$. Let $\mathfrak C$ be the graded commutative subalgebra of $\mathfrak H$ generated by $\q_1(a)$ for $a\in H^*(A)$. The action of $\mathfrak H$ on $\vac$ gives $\H$ and the action of $\mathfrak C$ on $\vac$ gives $\rho^*(H^*(\Sym^n(A),\Q))\cong \SSym^n(H)$.
By sending $\mathfrak d$ to the identity, we define a linear map $c : \mathfrak H \rightarrow \mathfrak C$. 
Denote $J\hilb{n}$ the ideal generated by $H^1(A\hilb{n},\Q)$ in $H^*(A \hilb n,\Q)$. We claim that for every $\mathfrak y\in \mathfrak H$:
$$
\mathfrak y\vac \in J\hilb{n} \Leftrightarrow c(\mathfrak y)\vac \in J\hilb{n}.
$$
To see this, we remark that $H^1(A \hilb n,\Q) \cong H^1(A ,\Q)  $ and the multiplication with a class in $H^1(A \hilb n,\Q) $ is given by the operator $\mathfrak G_0(a)$ for some $a\in H^1(A ,\Q)$. Due to the fact that $\mathfrak d$ is also a multiplication operator (of degree 2), $\mathfrak G_0(a)$ commutes with $\mathfrak d$. It follows that for $\mathfrak y =\mathfrak G_0(a) \mathfrak r$ we have $c(\mathfrak y) = \mathfrak G_0(a) c(\mathfrak r)$.

Now denote $\mathfrak k$ the multiplication operator with the class $[\kum{A}{n-1}]$. We have:
$
[\mathfrak k, \mathfrak d] = 0.
$
Now let $y\in H^*(A\hilb{n},\Q)$ be a class in the annihilator of $[\kum{A}{n-1}]$. We can write $y= \mathfrak y\vac$ for a $\mathfrak y\in\mathfrak H$. Choose $\tilde y \in \SSym^n (H)$ in a way that $\rho^*(\tilde y) = c(\mathfrak y) \vac$. Then we have:
$$
0=[\kum{A}{n-1}]\cdot y = \mathfrak k\, \mathfrak y \vac =  \mathfrak k \,c(\mathfrak y)\vac = \rho^*(\sigma^*(x) \cdot \tilde y).
$$
Since $\rho^*$ is injective, $\tilde y$ is in the annihilator of $\sigma^*(x)$, so $\tilde y \in J^{(n)}$. It follows that $c(\mathfrak y)\vac$ and $y$ are in the ideal generated by $H^1(A\hilb{n},\Q)$.
\end{proof}

\begin{theorem}\cite[Th\'eor\`eme 4]{Beauville}
$\kum{A}{n-1}$ is a irreducible holomorphically symplectic manifold. In particular, it is simply connected and the canonical bundle is trivial.
\end{theorem}
This implies that $H^2(\kum{A}{n-1},\Z)$ admits an integer-valued non-degenerate symmetric bilinear form (the Beauville--Bogomolov form) $B_{\kum{A}{n-1}}$ which gives $H^2(\kum{A}{n-1},\Z)$ the structure of a lattice. Looking, for instance, in the useful table from the introduction of~\cite{Rapagnetta}, we know that this lattice is
isomorphic to $U^{\oplus 3}\oplus \left< -2n \right>$, for $n\geq 3$. 
We have the Fujiki formula for $\alpha\in H^2(\kum{A}{n-1},\Z)$:
\begin{equation} \label{fujiki}
\int_{\kum{A}{n-1}} \alpha^{2n-2} = n\cdot(2n-3)!!\cdot B_{\kum{A}{n-1}}(\alpha,\alpha)^{n-1}
\end{equation}

\begin{proposition}\label{H2Sur} Assume $n\geq 3$. Then
$\theta^*$ is surjective on $H^2(A\hilb{n},\Z)$.
\end{proposition}
\begin{proof}
By~\cite[Sect.~7]{Beauville}, $\theta^{\ast} : H^2(A\hilb{n},\C) \rightarrow H^2(\kum{A}{n-1},\C)$ is surjective. 
But by Proposition 1 of~\cite{Britze}, the lattice structure of $\im \theta^*$ is the same as of $H^2(\kum{A}{n-1})$, so the image of $H^2(A\hilb{n},\Z)$ must be primitive. The result follows.
\end{proof}
\begin{notation}\label{BasisH2KA}
 We have seen that, for $n\geq 3$,
 $$
 H^2(\kum{A}{n-1},\Z) \cong H^2(A,\Z) \oplus\left<\theta^*(\delta)\right>.
 $$
We denote the injection $ : H^2(A,\Z) \rightarrow H^2(\kum{A}{n-1},\Z)$ by $j$. It can be described by 
$$
j : a \longmapsto \frac{1}{(n-1)!}\theta^*\left(\q_1(a)\q_1(1)^{n-1}\vac\right).
$$ 
Further, we set $e:=\theta^*(\delta)$. Using Notation~\ref{TorusClasses}, we give the following names for classes in $H^2(\kum{A}{n-1},\Z)$:
\begin{align*}
u_1 &:= j(a_1 a_2), & v_1 &:= j(a_1 a_3), & w_1 &:= j(a_1 a_4), \\ 
u_2 &:= j(a_3 a_4), & v_2 &:= j(a_4 a_2), & w_2 &:= j(a_2 a_3),
\end{align*}
These elements form a basis of $H^2(\kum{A}{n-1},\Z)$ with the following intersection relations under the Beauville-Bogomolov form:
\begin{align*}
B(u_1,u_2) &= 1, & B(v_1,v_2) &= 1, & B(w_1,w_2) &= 1,  &
B(e,e)&= -2n,
\end{align*}
and all other pairs of basis elements are orthogonal.

If $A=E_1\times E_2$ is the product of two elliptic curves, we choose the $a_i$ in a way such that $\{a_1,a_2\}$ 
and $\{a_3,a_4\}$ give bases of $H^1(E_1,\Z)$ and $H^1(E_2,\Z)$ in the decomposition $H^1(A) = H^1(E_1)\oplus H^1(E_2)$, respectively.
\end{notation}

\section{Integral cohomology of the generalized Kummer fourfold}
Now we come to the special case $n=3$, so we study $\kum{A}{2}$, the generalized Kummer fourfolds.
\begin{proposition}
The Betti numbers of $\kum{A}{2}$ are:
$
1,\,0,\,7,\,8,\,108,\,8,\,7,\,0,\,1.
$
\end{proposition}
\begin{proof}
This follows from G\"ottsche's formula~\cite[page 49]{Gottsche}.
\end{proof}

First, we deduce a description of the integral odd cohomology groups of $\X$ from Proposition \ref{A2Basis} and Section \ref{Section_GeneralKummer}.
The middle cohomology $H^4(\X,\Z)$ has been studied by Hassett and Tschinkel in \cite{Hassett}. We recall some of their results in Section \ref{HassetTschinkelSection},
then we proceed by using $\theta^*$ to give a partial description of $H^4(\X,\Z)$ in terms of the well-understood cohomology of $A\hilb{3}$ in Section \ref{SyminH4}. 
Finally in Section \ref{integralbasisH4}, we find a basis of $H^4(\X,\Z)$ using the action of the image of monodromy representation.
In order to use monodromy representation, we start by recalling some notions of monodromy on abelian surfaces in Section \ref{monodromyexplication}. 
We will also need some technical calculation related the the action of the symplectic group over finite fields (Section \ref{Section_Symplectic}).

In all the section, we use Notations \ref{TorusClasses} and \ref{BasisH2KA}.
\subsection{Odd Cohomology of the generalized Kummer fourfold}\label{oddcohoK2}

By means of the morphism $\theta^*$, we may express part of the cohomology of $\kum{A}{2}$ in terms of Hilbert scheme cohomology. We have seen in Proposition~\ref{H2Sur} that $\theta^*$ is surjective for degree $2$ and (by duality) also in degree $6$. 
The next proposition shows that this also holds true for odd degrees.
\begin{proposition}\label{oddcohomology}
A basis of $H^3(\X,\Z)$ is given by:
\begin{gather}
\label{A3_1}
\frac{1}{2}\theta^*\Big( \q_1(a^*_i)\q_1(1)^2\vac \Big), \\
\label{A3_2}
\frac{1}{2}\theta^*\Big(\q_2(a_i)\q_1(1)\vac\Big).
\end{gather}
and a dual basis of $H^5(\X,\Z)$ is given by:
\begin{gather}
\label{A5_1}
 \theta^*\Big( \q_1(a_ia_j)\q_1(a_j^*)\q_1(1) \vac \Big) \ \text{for any } j\neq i, \\
\label{A5_2}
\theta^*\Big( \q_2(a^*_i)\q_1(1)\vac \Big).
\end{gather}
\end{proposition}
\begin{proof}
The classes (\ref{A3_1}) are Poincar\'e dual to (\ref{A5_1}) and the classes (\ref{A3_2}) are Poincar\'e dual to (\ref{A5_2}) by direct computation:
\begin{gather*}
\frac{1}{2}\theta^*\Big( \q_1(a^*_i)\q_1(1)^2\vac \Big)\cdot \theta^*\Big( \q_1(a_ia_j)\q_1(a_j^*)\q_1(1) \vac \Big) \hspace{80pt}
\\ = \frac{1}{2} \theta^*\Big(  \G_0(a^*_i) \q_1(a_ia_j)\q_1(a_j^*)\q_1(1) \vac \Big)\\
 =  \frac{1}{2}[\X]\cdot \q_1(a_ia_j)\q_1(a_j^*)\q_1(a_i^*) = 1, \\
 \frac{1}{2}\theta^*\Big(\q_2(a_i)\q_1(1)\vac\Big)\cdot \theta^*\Big( \q_2(a^*_i)\q_1(1)\vac \Big) =\theta^*\Big( \G_1(a_i)\q_2(a^*_i)\q_1(1)\vac  \Big) \\
 = [\X]\cdot \left( 2\q_3(x) - \q_1(x)^2\q_1(1) \right) \vac = 0-1=-1.
\end{gather*}
It remains to show that all classes are integral.
It is clear from Lemma~\ref{IntegralOperators} that (\ref{A3_1}) is integral, while the integrality of (\ref{A5_1}) and (\ref{A5_2}) is obvious. By Proposition~\ref{A2Basis}, $\frac{1}{2}\q_2(a_i)\vac$ is integral as well. If the operator $ \q_1(1)$ is applied, we get again an integral class.
\end{proof}
\begin{cor}\label{actionH3}
Let $A$ be an abelian surface and let $g$ be an automorphism of $A$. Let $g^{[[3]]}$ be the automorphism induced by $g$ on $K_2(A)$.
By Proposition~\ref{oddcohomology}, $H^3(K_2(A),\Z)\cong H^1(A,\Z)\oplus H^3(A,\Z)$ and the action of $g^{[[3]]}$ on $H^3(K_2(A),\Z)$ is given by the action of $g$ on $H^1(A,\Z)\oplus H^3(A,\Z)$.
\end{cor}
\begin{proof}
By Proposition \ref{oddcohomology}, we have an isomorphism:
$$f:H^1(A,\Z)\oplus H^3(A,\Z)\rightarrow H^3(K_2(A),\Z),$$
given by $f(a_i)=\frac{1}{2}\theta^*\Big(\q_2(a_i)\q_1(1)\vac\Big)$ and $f(a_i^*)=\frac{1}{2}\theta^*\Big( \q_1(a^*_i)\q_1(1)^2\vac \Big)$, for all $i\in\{1,...,4\}$.
We want to prove that $f\circ g^*=g^{[[3]]*}\circ f$. 
To do so, it is enough to show that $f\circ g^*(a_i)=g^{[[3]]*}\circ f(a_i)$ and $f\circ g^*(a_i^*)=g^{[[3]]*}\circ f(a_i^*)$ for all $i\in\{1,...,4\}$.
Let $g^{[3]}$ be the morphism induced by $g$ on $A^{[3]}$. By definition of $g^{[[3]]}$:
\begin{equation}
g^{[[3]]}\circ \theta=\theta\circ g^{[3]}.
\label{commuteThetag}
\end{equation}
Then by (\ref{commuteThetag}) and by defintion of $g^{[3]}$:
\begin{align*}
 g^{[[3]]*}\circ f(a_i)&=g^{[[3]]*}\circ\theta^*\Big(\q_2(a_i)\q_1(1)\vac\Big)\\
 &=\theta^*\circ g^{[3]*}\Big(\q_2(a_i)\q_1(1)\vac\Big)\\
 &=\theta^*\Big(\q_2(g^*(a_i))\q_1(1)\vac\Big)\\
 &=f\circ g^*(a_i).
\end{align*}
We prove $f\circ g^*(a_i^*)=g^{[[3]]*}\circ f(a_i^*)$ with the same method.

\end{proof}

\subsection{A monodromy representation on abelian surfaces and generalized Kummer fourfolds}\label{monodromyexplication}


Let $A$ be an abelian surface. We recall that a \emph{principal polarization} of $A$ is a polarization $L$ such that there exists a basis of $H_1(A,\Z)$, with respect to which the symplectic bilinear form on $H_1(A,\Z)$ induced by $c_1(L)$:
\begin{equation}
\omega_L(x,y)=x\cdot c_1(L)\cdot y,
\label{symplecticprinc}
\end{equation}
is given by the matrix:
$$\left( {\begin{array}{cccc}
   0 & 0 & 1 & 0 \\    0 &  0 & 0 & 1\\ -1 & 0 & 0 & 0\\ 0 & -1 & 0 & 0     
   \end{array} } \right).$$

We recall the following result. 
\begin{prop}
Let $(A,L)$ be a principally polarized abelian surface. The group $H_1(A,\Z)$ is endowed with the symplectic from $\omega_L$ defined in (\ref{symplecticprinc}). Let $\Mon (H_1(A,\Z))$ be the image of monodromy representations on $H_1(A,\Z)$.
Then $\Mon (H_1(A,\Z))\supset\Sp(H_1(A,\Z))$.
\end{prop}
\begin{proof}
It can be seen as follows.
Let $\mathcal{M}_2$ be the moduli space of curves of genus $2$ and $\mathcal{A}_2$ be the moduli space of principally polarized abelian surfaces.
By the Torelli theorem (see for instance~\cite[Theorem 12.1]{Milne}), we have an injection $J:\mathcal{M}_2\hookrightarrow \mathcal{A}_2$ given by taking the Jacobian of the curve endowed with its canonical polarization. Moreover, the moduli spaces $\mathcal{M}_2$ and $\mathcal{A}_2$ are both of dimension 3. 

Now if $\mathscr{C}_2$ is a curve of genus 2, we have by Theorem 6.4 of~\cite{Farb}: 
$$\Mon (H_1(\mathscr{C}_2,\Z))\supset \Sp(H_1(\mathscr{C}_2,\Z)),$$
where the symplectic form on $H_1(\mathscr{C}_2,\Z)$ is given by the cup product. 
Then the result follows from the fact that the lattices $H_1(\mathscr{C}_2,\Z)$ and $H_1(J(\mathscr{C}_2),\Z)$ are isometric.
\end{proof}
\begin{rmk}\label{SPA2}
Let $(A,L)$ be a principally polarized abelian surface and $p$ a prime number. The group $H_1(A,\Z)$ tensorized by $\mathbb{F}_p$ can be seen as the group $A[p]$ of points of $p$-torsion on $A$ and the form $\omega_L\otimes\mathbb{F}_p$ provides a symplectic form on $A[p]$. Then $\Mon (A[p])$, the image of the monodromy representation on $A[p]$ contains the group $\Sp(A[p])$. 
\end{rmk}
Now, we are ready to recall Proposition 5.2 of~\cite{Hassett} on the monodromy of the generalized Kummer fourfold.
\begin{prop}\label{Hassettmonodromy}
Let $A$ be an abelian surface and $K_2(A)$ the associated generalized Kummer fourfold. 
The image of the monodromy representation on $\Pi=\left\langle\left. Z_\tau\right|\ \tau\in A[3]\right\rangle$ contains the semidirect product
$\Sp(A[3])\ltimes A[3]$ which acts as follows:
$$f\cdot Z_\tau= Z_{f(\tau)}\ \text{and}\ \tau'\cdot Z_\tau= Z_{\tau+\tau'},$$
for all $f\in \Sp(A[3])$ and $\tau'\in A[3]$.
\end{prop}
\subsection{Actions of the symplectic group over finite fields}\label{Section_Symplectic}
The aim of this subsection is to provide some special computations used in Section~\ref{integralbasisH4}.

Let $V$ be a symplectic vector space of dimension $n\in 2\mathbb{N}$ over a field $k$ with a nondegenerate symplectic form $\omega : \Lambda^2 V \rightarrow k$. A line is a one-dimensional subspace of $V$ through the origin, a plane is a two-dimensional subspace of $V$. A plane $P\subset V$ is called isotropic, if $\omega (x,y)=0$ for any $x,y\in P$, otherwise non-isotropic.  The symplectic group $\Sp V$ is the set of all linear maps $\phi : V\rightarrow V$ with the property $\omega(\phi(x),\phi(y)) = \omega(x,y)$ for all $ x,y\in V$.
\begin{proposition}\label{transitively}
The symplectic group $\Sp V$ acts transitively on the set of non-isotropic planes as well as on the set of isotropic planes.
\end{proposition}
\begin{proof}
Given two planes $P_1$ and $P_2$, we may choose vectors $v_1,v_2,w_1,w_2$ such that $v_1,v_2$ span $P_1$ and $w_1,w_2$ span $P_2$ and $\omega(v_1,v_2) =\omega(w_1,w_2)$. We complete $\{v_1,v_2\}$ as well as $\{w_1,w_2\}$ to a symplectic basis of $V$.
Then define $\phi(v_1)=w_1$ and $\phi(v_2)=w_2$. 
It is now easy to see that the definition of $\phi$ can be extended to the remaining basis elements to give a symplectic morphism.
\end{proof}
\begin{remark} \label{simplePlanes}
The set of planes in $V$ can be identified with the simple tensors in $\Lambda^2V$ up to multiples. Indeed, given a simple tensor $v\wedge w \in \Lambda^2 V$, the span of $v$ and $w$ yields the corresponding plane. Conversely, any two spanning vectors $v$ and $w$ of a plane give the same element $v\wedge w$ (up to multiples).
\end{remark}
From now on, we assume that $k$ is finite of cardinality $q$.
\begin{remark} \label{PlaneTriple}
 If $k$ is the field with two elements, then the set of planes in $V$ can be identified with the set $\{\{x,y,z\}\;|\;x,y,z\in V\backslash\{0\},\,x+y+z=0\}$. Observe that for such a $\{x,y,z\}$, $\omega(x,y)=\omega(x,z)=\omega(y,z)$ and this value gives the criterion for isotropy.
\end{remark}
\begin{proposition}\label{OrbitesSp}
\begin{align}
&\text{The number of lines in $V$ is }\frac{q^n-1}{q-1}, \\
&\text{the number of planes in $V$ is }\frac{(q^n-1)(q^{n-1}-1)}{(q^2-1)(q-1)}, \\
&\text{the number of isotropic planes in $V$ is }\frac{(q^n-1)(q^{n-2}-1)}{(q^2-1)(q-1)}, \\
&\text{the number of non-isotropic planes in $V$ is }\frac{q^{n-2}(q^n-1)}{q^2-1}.
\end{align}
\end{proposition}
\begin{proof}
A line $\ell$ in $V$ is determined by a nonzero vector. There are $q^n - 1$ nonzero vectors in $V$ and $q-1$ nonzero vectors in $\ell$. A plane $P$ is determined by a line $\ell_1 \subset V$ and a unique second line $\ell_2\in V/\ell_1$. We have $\frac{q^2-1}{q-1}$ choices for $\ell_1$ in $P$. The number of planes is therefore
$$
\frac{ \frac{q^n-1}{q-1} \cdot\frac{q^{n-1}-1}{q-1}}{\frac{q^2-1}{q-1} } = \frac{(q^n-1)(q^{n-1}-1)}{(q^2-1)(q-1)}.
$$
For an isotropic plane we have to choose the second line from $\ell_1^\perp/\ell_1$. This is a space of dimension $n-2$, hence the formula. The number of non-isotropic planes is the difference of the two previous numbers.
\end{proof}

We want to study the free $k$-module $k[V]$ with basis $\{X_i \,|\, i\in V\}$. It carries a natural $k$-algebra structure, given by
$X_i X_j \defIs  X_{i+j}$ with unit $1=X_0$. This algebra is local with maximal ideal $\mathfrak m$ generated by all elements of the form $(X_i-1)$.

We introduce an action of $\Sp (4,k)$ on $k[V]$ by setting $\phi(X_i) = X_{\phi(i)}$. Furthermore, the underlying additive group of $V$ acts on $k[V]$ by $v( X_i) = X_{i+v} =X_iX_v$. 
\begin{definition}
For a line $\ell\subset V$ define $S_\ell \defIs  \sum_{i\in\ell} X_i$. For a vector $0\neq v\in \ell$ we set $S_v\defIs S_\ell$.
\end{definition}
\begin{lemma}\label{SympLemma}
Let $P\subset V$ be a plane and $\ell_1,\ell_2\subset P$ two different lines spanning $P$. Then we have
$$
S_{\ell_1}S_{\ell_2}=\sum_{i\in P}X_i = \sum_{\ell\subset P}S_\ell.
$$
\end{lemma}
\begin{proof}
The first equality is clear. For the second equality observe that every point $i\in P$ is contained in one line, if we count modulo $q$.
\end{proof}

\begin{definition} \label{SymplecticIdeal}
We define two subsets of $k[V]$:
\begin{gather*}
M \defIs  \left\{\sum_{i\in P}X_i \,|\, P\subset V \text{ plane}\right\}, \\
N  \defIs  \left\{\sum_{i\in P}X_i \,|\, P\subset V \text{ non-isotropic plane}\right\}.
\end{gather*}
Let $(M)$ and $(N)$ be the ideals generated by $M$ and $N$, respectively. 
Further, let $D$ be the linear span of $\{v(b) - b \,|\, b\in N, v\in V \}$. Then $D$ is in fact an ideal, namely the product of ideals $\mathfrak m\cdot (N)$.
\end{definition}
\begin{proposition}
We have $(M)=(N)$.
\end{proposition}
\begin{proof}
We have to show that $\sum_{i\in P}X_i \in (N)$ for an isotropic plane $P$. Let $v,w$ be two spanning vectors of $P$ and $u$ a vector with $\omega(u,v)\neq 0$. Denote $P'$ the non-isotropic plane spanned by $u $ and $v$. By Lemma~\ref{SympLemma}, we have
$$
S_uS_vS_w = \sum_{\ell\subset P'} S_{\ell}S_w= \left(S_v+\sum_{\lambda\in k}S_{u + \lambda v}\right)S_w.
$$
Now $w$ spans a non-isotropic plane with every line in $P'$, except one, namely the line that contains $v$. So it follows that
$$
\sum_{i\in P}X_i = S_vS_w = S_uS_vS_w  - \sum_{\lambda\in k} S_{u + \lambda v}S_w, 
$$
and we see that the right hand side is an element of $(N)$.
\end{proof}
For the rest of this section, we assume $\dim_k V=4$. 
\begin{proposition} \label{SymplecticIdealsDimension}
The following table illustrates the dimensions of $(N)$ and $D$ for some $k$.
\vspace{2mm}
\begin{center}
\begin{tabular}{c||c|c|c}
 $k$ & $\dim_k(N)$ & $\dim_k D$ \\
\hline
$\mathbb F_2$   & 11 &  5  \\
$\mathbb F_3$  & 50 & 31  \\
$\mathbb F_5$  &355 &270
\end{tabular}
\end{center}
\end{proposition}
An easy way to get these numbers is to count elements in the respective vector spaces using a computer.
\begin{rmk}\label{c2}
We remark that $X\defIs \sum_{i\in V}X_i\in D$. Indeed, let $P$, $P'$ be two non-isotropic planes with $P \cap P' = 0$. Then $X = \left( \sum_{i\in P}X_i\right)\left(  \sum_{i\in P'}X_i \right)$ and both factors are contained in $(N)\subset \mathfrak m$, so $X\in \mathfrak m \cdot (N) = D$.
\end{rmk}
Let us now consider some special orthogonal sums. Take two vectors $v,w\in V$ with $\omega(v,w)=1$ and set $x\defIs  (v\wedge w)^2\in \Sym^2 (\Lambda^2V)$. Denote $P$ the plane spanned by $v$ and $w$ and set $y\defIs  \sum_{i\in P}X_i\in  k[V]$.

We consider now the action of $\Sp V$ on $\Sym^2 (\Lambda^2V)\oplus k[V]$. 
Denote $O$ the vector space spanned by the elements $\phi(x)\oplus \phi(z),$ for $\phi \in \Sp V,$ $z \in (y)$
and $U$ the vector space spanned by the elements $\phi(x),$ for $\phi \in \Sp V$.
\begin{proposition}\label{CombinedSymplectic}
Then we have by numerical computation:
\vspace{2mm}
\begin{center}
\begin{tabular}{c||c|c}
 $k$ & $\dim_k O$  & $\dim_k U$\\
\hline
$\mathbb F_2$ & 11 & 6 \\
$\mathbb F_3$ & 51  & 20 \\
$\mathbb F_5$ & 375  & 20 
\end{tabular}
\end{center}
\end{proposition}
Now we prove the following lemma that we will need for a divisibility argument in Section~\ref{integralbasisH4}.
\begin{lemme}\label{cleffinclassesdiv}
We assume that $k=\mathbb F_3$. Let $\pr_1: \Sym^2 (\Lambda^2V)\oplus k[V]\rightarrow \Sym^2 (\Lambda^2V)$ and $\pr_2: \Sym^2 (\Lambda^2V)\oplus k[V]\rightarrow k[V]$ be the projections. 
We have: 
\begin{itemize}
\item[(i)]
$\dim \ker\pr_{2|O}=1$.
\item[(ii)]
$\dim \ker\pr_{1|O}=31$.
\end{itemize}
\end{lemme}
\begin{proof}
We have $\pr_{1}(O)=U$ and $\pr_{2}(O)=(N)$. Using the dimension tables from Propositions~\ref{CombinedSymplectic} and~\ref{cleffinclassesdiv}, we get
\begin{gather*}
\dim \ker\pr_{1|O} = \dim O - \dim U  = 31,\\
\dim \ker\pr_{2|O} = \dim O - \dim (N) = 1.
\end{gather*}

\end{proof}

\subsection{Recall of Hassett and Tschinkel's results}\label{HassetTschinkelSection} \label{Middle}
\begin{notation}\label{TheZs}
For each $\tau \in A$, denote $W_\tau$ the Brian\c con subscheme of $A\hilb{3}$ consisting of the elements supported entirely at the point $\tau$. If $\tau\in A[3]$ is a point of three-torsion, $W_\tau$ is actually a subscheme of $\X$. We will also use the symbol $W_\tau$ for the corresponding class in $H^4(\X,\Z)$. Further, set 
$$
W := \sum_{\tau\in A[3]} W_\tau.
$$
For $p\in A$, denote $Y_p$ the locus of all $\{x,y,p\}$ in $\X$. The corresponding class $Y_p \in H^4(\X,\Z)$ is independent of the choice of the point $p$. Then set $Z_\tau := Y_p - W_\tau$ and denote $\Pi$ the lattice generated by all $Z_\tau$, $\tau \in A[3]$.
\end{notation}
\begin{proposition}
Denote by $\Sym := \Sym^2\left(H^2\left(\X,\Z\right)\right) \subset H^4\left(\X,\Z\right) $ the span of products of integral classes in degree $2$.
Then 
$$
\Sym+\Pi \ \subset\  H^4\left(\X,\Z\right)
$$
is a sublattice  of full rank.  
\end{proposition}
\begin{proof}
This follows from \cite[Proposition 4.3]{Hassett}.
\end{proof}

In Section 4 of \cite{Hassett}, one finds the following formula:
\begin{equation}
Z_{\tau}\cdot D_{1}\cdot D_{2}=2\cdot B_{\X}(D_{1},D_{2}),
\label{ZT}
\end{equation}
for all $D_{1}$, $D_{2}$ in $H^{2}(\X,\Z)$, $\tau\in A[3]$ and $B_{\X}$ the Beauville-Bogomolov form on $\X$.

\begin{definition}\label{defiPi}
We define $\Pi' := \Pi \cap \Sym^{\perp}$. It follows from (\ref{ZT}) that $\Pi'$ can be described as the span of all classes of the form $Z_\tau -Z_0$ or alternatively as the set of all
$
\sum_\tau \alpha_\tau Z_\tau $, such that $ \sum_\tau \alpha_\tau =0$.
Note that in \cite{Hassett} the symbol $\Pi'$ denotes something different.
\end{definition}
\begin{remark}
Since $\rk \Sym = 28$ and $\rk \Pi' = 80$, the lattice $\Sym\oplus\Pi'\subset H^4\left(\X,\Z\right)$ has full rank.
\end{remark}

\begin{proposition}
The class $W$ can be written with the help of the square of half the diagonal as
\begin{align} 
W &= \theta^*\Big( \q_3(1)\vac \Big) \\
\label{WFormula}
&= 9 Y_p + e^2.
\end{align}
The second Chern class is non-divisible and given by 
\begin{align}
\label{sumZ}
\cc &= \frac{1}{3}\sum_{\tau\in A[3] } Z_\tau \\
\label{ChernY}
&= \frac{1}{3} \Big(72Y_p-e^2 \Big). 
\end{align}
\end{proposition}
\begin{proof} 
In Section 4 of \cite{Hassett} one finds the equations
\begin{gather}
W = \frac{3}{8}\Big(\cc + 3e^2\Big),\label{WW} \\
Y_p = \frac{1}{72}\Big(3\cc + e^2\Big),
\end{gather}
from which we deduce (\ref{WFormula}) and (\ref{ChernY}).
Equation (\ref{sumZ}) and the non-divisibility are from \cite[Proposition 5.1]{Hassett}.
\end{proof}
\subsection{Properties of $\Sym^2H^2(K_2(A),\Z)$ in $H^4(K_2(A),\Z)$}\label{SyminH4}
\begin{proposition}\label{ImSym}
The image of $H^4(A\hilb{3},\Q)$ under the morphism $\theta^*$ is equal to $\Sym^2H^2(\X,\Q)$.
\end{proposition}
\begin{proof}
We start by giving a set of universal generators of $H^4(A\hilb{n},\Q)$, $n\geq 0$. Theorem 5.30 of \cite{LiQinWang} ensures that it is possible to do this in terms of multiplication operators. To enumerate elements of $H^*(A,\Q)$, we follow Notation \ref{TorusClasses}. Basis elements of $H^2(A,\Q)$ will be denoted by $b_i$ for $1\leq i\leq 6$. Then our set of generators is given by:
\begin{center}
\begin{tabular}{l|c}
multiplication operator & number of classes \\ \hline
$\G_0(a_1)\G_0(a_2)\G_0(a_3)\G_0(a_4) $  & $1$ \\
$\G_0(a_i)\G_0(a_j)\G_0(b_k)$ for $i<j$  & $\binom{4}{2}\cdot 6 = 36$ \\
$\G_0(a_i)\G_0(a_j^*)$ & $4\cdot 4 = 16$ \\
$\G_0(b_i)\G_0(b_j)$ for $i\leq j$& $\binom{6+1}{2}= 21$ \\
$\G_0(x)$ &  $1$  \\ \hline
$\G_0(a_i)\G_0(a_j)\G_1(1)$ for $i<j$ & $\binom{4}{2} = 6$ \\
$\G_0(a_i)\G_1(a_j)$ & $4\cdot 4=16$ \\
$\G_0(b_i)\G_1(1)$ & $6$ \\
$\G_1(b_i)$ & $6$ \\
$\G_1(1)^2 $ & $1$ \\ \hline
$\G_2(1) $ &$1$
\end{tabular}
\end{center}
Any multiplication operator of degree $4$ can be written as a linear combination of these $111$ classes. Likewise, the dimension of $H^4(A\hilb{n},\Q)$ is $111$ for all $n\geq 4$, according to G\"ottsche's formula \cite[p.~35]{Gottsche}. However, for smaller $n$, there must be relations of linear dependence. For $n=3$, the $8$ classes $\G_0(x)$, $\G_1(b_i)$ and $\G_2(1) $ can be expressed as linear combinations of the others, so we are left with $103$ linearly independent classes that form a basis of $H^4(A\hilb{3},\Q)$. Multiplication with the class $[\X]$ is given by the operator $\G_0(a_1)\G_0(a_2)\G_0(a_3)\G_0(a_4) $ and annihilates every class that contains an operator of the form $\G_0(a_i)$. There are $75$ such classes, so by Proposition \ref{annihilator}, $\ker \theta^*\subset H^4(A\hilb{3},\Q)$ has dimension at least $75$ and $\im\theta^*$ has dimension at most $103-75=28$. However, since the image of $\theta^*$ must contain $\Sym^2H^2(\X,\Q)$, which is $28$-dimensional, equality follows.
\end{proof}

\begin{proposition} \label{ChernSym}
We have:
\begin{equation}
\cc= 4 u_1u_2 + 4v_1v_2 + 4 w_1 w_2 - \frac{1}{3} e^2. 
\end{equation}
In particular, $\cc\in \Sym \otimes\Q $.
\end{proposition}
\begin{proof}
First note that the defining diagram (\ref{square}) of the Kummer manifold is the pullback of the inclusion of a point, so the normal bundle of $\X$ in $A\hilb{3}$ is trivial. The Chern class of the tangent bundle of $\X$ is therefore given by the pullback from $A\hilb{3}$: $c(\X) = \theta^*\left(c(A\hilb{3})\right)$. Proposition \ref{ImSym} allows now to conclude that $\cc\in \Sym \otimes \Q$.

To obtain the precise formula, we use a result of Boissi\`ere, \cite[Lemma 3.12]{Boissiere}, giving a commutation relation for the Chern character multiplication operator on the Hilbert scheme. We get:
\begin{align*}
c_2(A\hilb{3}) & = 3\q_1(1)\mathfrak L_2(1)\vac - \frac{1}{3} \q_3(1)\vac \\
 & =\frac{8}{3}\q_1(1)\mathfrak L_2(1)\vac - \frac{1}{3} \delta^2 .
\end{align*}
With Corollary \ref{KummerEquality} one shows now, that $\cc$ is given as stated.
\end{proof}
\begin{rmk}
Proposition \ref{ChernSym} can also be proven using (\ref{sumZ}) and (\ref{ZT}).
\end{rmk}

\begin{corollary}\label{Pi'}
The intersection $\Sym\cap \Pi$ is one-dimensional and spanned by $3\cc$. 
\end{corollary}
\begin{proof}
By Proposition \ref{ChernSym} and (\ref{sumZ}), $3\cc\in \Sym\cap \Pi$. Since the ranks of $\Sym$, $\Pi$ and $H^4(\X,\Z)$ are $28$, $81$ and $108$, respectively, the intersection cannot contain more.
\end{proof}

\begin{corollary}\label{Classuvw}
\begin{equation} \label{YSym}
Y_p =  \frac{1}{6}\Big(u_1u_2 + v_1v_2 +  w_1 w_2 \Big).
\end{equation}
\end{corollary}
\begin{remark}\label{afterClassuvw}
Using Nakajima operators, we may write
\begin{gather}
Y_p = \frac{1}{9} \theta^*\Big( \q_1(1)\mathfrak L_2(1) \vac \Big) =  \frac{1}{2}\theta^*\Big(\q_1(x)\q_1(1)^2\vac\Big).
\end{gather}
\end{remark}
Now, we can summarize all the divisible classes found in $\Sym$ in the following proposition. 
\begin{proposition}\label{classedivisibleSym}
Let $\{e,u_1,u_2,v_1,v_2,w_1,w_2\}$ be the integral basis of $H^2(\X,\Z)$ as defined in Notation \ref{BasisH2KA}. 
We have:
\begin{enumerate}
\item
the class $e^2$ is divisible by $3$, 
\item
the class $u_1u_2 + v_1v_2+w_1w_2$ is divisible by 6,
\item
for $y\in\{u_1,u_2,v_1,v_2,w_1,w_2\}$, the class 
$
e \cdot y
$
is divisible by $3$ and 
$
 y^2 - \frac{1}{3} e\cdot y
$
is divisible by $2$.
\end{enumerate}
\end{proposition}
\begin{proof}
\begin{enumerate}
\item
From Proposition \ref{ChernSym} we see that
$e^2$ is divisible by $3$ and by Corollary \ref{Classuvw}
the class $u_1u_2 + v_1v_2+w_1w_2$ is divisible by 6.
\item
We have $y= \theta^*\left(\q_1(a)\q_1(1)^2\vac\right)$ for some $a\in H^2(A,\Z)$. A computation yields:
$$
e\cdot y = 3\cdot \theta^*\Big( \q_2(a)\q_1(1)\vac \Big)
\quad \text{and}\quad
y^2 = \theta^*\Big(\q_1(a)^2\q_1(1)\vac\Big)
$$
so $e\cdot y$ is divisible by $3$. Furthermore, by Corollary \ref{IntegralOperatorsTorus}, the class 
$
\frac{1}{2} \q_1(a)^2\q_1(1)\vac - \frac{1}{2} \q_2(a)\q_1(1)\vac 
$
is contained in $H^4(A\hilb{3},\Z)$, so its pullback
$
 \frac{1}{2}y^2 - \frac{1}{6} e\cdot y
$
is an integral class, too.
\end{enumerate}
\end{proof}
\subsection{Integral basis of $H^4(K_2(A),\Z)$}\label{integralbasisH4}

From the intersection properties $Z_\tau \cdot Z_{\tau'} = 1$ for $\tau\neq \tau'$ and $Z_\tau^2 = 4$ from Section 4 of \cite{Hassett}, we compute
\begin{equation}
 \discr \Pi' = 3^{84}.
 \label{discrPi}
\end{equation}
On the other hand, a formula developed in \cite{Kapfer} evaluates
\begin{equation} \label{DiscrSym}
\discr \Sym \ = 2^{14}\cdot 3^{38},
\end{equation}
so the lattices cannot be primitive. Denote $\Sym^{sat}$ and $\Pi'^{sat}$ the respective primitive overlattices of $\Sym$ and $\Pi'$ in $H^4(\X,\Z)$. $\Sym \oplus\Pi'$ is a sublattice of $H^4(\X,\Z)$ of index $2^{7}\cdot 3^{61}$ and we claim that $\Sym^{sat}\oplus \Pi'^{sat}$ has index $3^{22}$. 
We have already found in Proposition~\ref{classedivisibleSym}
\begin{itemize}
 \item $7$ linearly independent classes in $\Sym$ divisible by $2$,
 \item $8$ linearly independent classes in $\Sym$ divisible by $3$,
\end{itemize}
To obtain a basis of $H^4(\X,\Z)$, 
we are now going to find
\begin{itemize}
 \item $31$ linearly independent classes in $\Pi'$ divisible by $3$ and
 \item $20$ linearly independent classes in $\Sym^{sat}\oplus \Pi'^{sat}$, one divisible by $3^3$ and $19$ divisible by $3$.
\end{itemize}

The first thing to note is that $\Pi'$ is defined topologically for all deformations of $\X$ and the primitive overlattice of $\Pi'$ is a topological invariant, too.  
By applying a suitable deformation, we may therefore assume without loss of generality that $A$ is the product of two elliptic curves $A=E_1\times E_2$. Here according to Notation \ref{BasisH2KA}, $u_1:=j(a_1 a_2)$ where $\{a_1,a_2\}$ can be seen as a basis of $H^1(E_1)$ (it is necessary to obtain the following expression (\ref{LagrangianPlaneClass})).

Hassett and Tschinkel in Proposition 7.1 of \cite{Hassett} provide the class of a Lagrangian plane (i.e. a sub-varieties Lagrangian with respect to the holomorphic 2-form of $\X$ and isomorphic to the projective plane $\mathbb{P}^2$) $P\subset K_2(A)$, which can be expressed as follows:
\begin{equation}
\left[P\right]=\frac{1}{216}(6u_1-3e)^2+\frac{1}{8}\cc-\frac{1}{3}\sum_{\tau\in \Lambda'} Z_{\tau},
\label{LagrangianPlaneClass}
\end{equation}
where $\Lambda'=E_1[3]\times 0\subset A[3]$.
Hence by translating this plane by an element $\tau'\in A[3]$, we obtain another plane $P'$ that can be written:
$$\left[P'\right]=\frac{1}{216}(6u_1-3e)^2+\frac{1}{8}\cc-\frac{1}{3}\sum_{\tau\in \Lambda'+\tau'} Z_{\tau}.$$
By substracting these two expressions, we obtain a first class divisible by 3 in $\Pi'$:
\begin{equation}
 \frac{1}{3}\sum_{\tau\in\Lambda'} \Big(Z_{\tau} - Z_{\tau+\tau'}\Big)
 \label{first31}
\end{equation}

By Proposition \ref{Hassettmonodromy}, the image of the monodromy representation on the $Z_\tau$ contains the symplectic group $\Sp(4,\mathbb F_3)$. 
We know by Proposition \ref{transitively} that $\Sp(4,\mathbb F_3)$ acts transitively on the non-isotropic planes of $A[3]$. Hence, modulo $\Pi'$, the orbit by $\Sp(4,\mathbb F_3)$ of the classes (\ref{first31}) is a $\mathbb F_3$-vector space naturally isomorphic to $D$ as introduced in Definition \ref{SymplecticIdeal}, so by Proposition \ref{SymplecticIdealsDimension}, we get a subspace of $\Pi'$ of rank $31$ of classes divisible by $3$.
We add the thirds of these classes to $\Pi'$ and we get an over-lattice $\Pi'^{over}$ of $\Pi'$ such that:
\begin{equation}
 \frac{\Pi'^{over}}{\Pi'}=\left(\frac{\Z}{3\Z}\right)^{\oplus 31}.
 \label{Piover}
\end{equation}
This subspace can be determined by a computer. We describe it in Proposition~\ref{XXXI} (we will see later in this section that $\Pi'^{over}=\Pi'^{sat}$).  

Now we are going to find the classes divisible by 3 in $\Sym^{sat}\oplus \Pi'^{sat}$.
The class $Z_0$ is not contained in $\Sym$ nor in $\Pi'$.
It can be written as follows:
$$
Z_0=\frac{\sum_{\tau\in A[3]}Z_\tau-\sum_{\tau\in A[3]}(Z_\tau-Z_0)}{81}
\stackrel{(\ref{sumZ})}{=} \frac{c_2(K_2(A))-\frac{1}{3}\sum_{\tau\in A[3]}(Z_\tau-Z_0)}{27},
$$
where $\frac{1}{3}\sum_{\tau\in A[3]}(Z_\tau-Z_0)$ can be expressed as a linear combination of the 31 classes of Proposition \ref{XXXI} by Remark \ref{c2}.
Hence $Z_0$ is the class in $\Sym^{sat}\oplus \Pi'^{sat}$ divisible by $27$. 
Let us now find the remaining $19$ classes divisible by $3$.

We rearrange (\ref{LagrangianPlaneClass}) using (\ref{WW}):
\begin{align*}
\left[P\right]&=\frac{1}{216}(6u_1-3e)^2+\frac{1}{8}\cc-\frac{1}{3}\sum_{\tau\in \Lambda'} Z_{\tau}\\
&=\frac{36u_1^2+9e^2-36u_1\cdot e}{216} +\frac{W}{3}-\frac{3}{8}e^2-\frac{1}{3}\sum_{\tau\in \Lambda'} Z_{\tau}\\
&=\frac{u_1^2-2e^2-u_1\cdot e}{6} +\frac{W}{3}-\frac{1}{3}\sum_{\tau\in \Lambda'} Z_{\tau}.
\end{align*}
By Proposition \ref{classedivisibleSym}, the classes $e^2$ and $u_1\cdot e$ are both divisible by 3 and by (\ref{WW}), $W$ is divisible by 3, so the following class is integral:
$$
\mathfrak{N}:=\frac{u_1^2+\sum_{\tau\in \Lambda'} (Z_{\tau}-Z_0)}{3}
.
$$
From the considerations in Section \ref{monodromyexplication}, we know that the group $\Sp(A[3])\ltimes A[3]\subset \Mon(\Pi)$ is a natural extension of 
$\Sp(H_1(A,\Z))\ltimes A[3]\subset \Mon(H_1(A,\Z))$. Isometries of the image of the monodromy representation on $H_1(A,\Z)$ extend naturally to isometries of the image of the monodromy representation of $H^4(K_2(A),\Z)$ acting on $\Pi$ as describe in Proposition \ref{Hassettmonodromy} and acting on $\Sym$ by commuting with the map $j$ defined in Notation \ref{BasisH2KA}.
Hence the group $\Sp(H_1(A,\Z))\ltimes A[3]$ can be seen as a subgroup of $\Mon(H^4(K_2(A),\Z))$. 

Now we will conclude using this monodromy action of $\Sp(H_1(A,\Z))\ltimes A[3]$ on the element $\mathfrak{N}$ and the considerations from Section \ref{Section_Symplectic}. 
Proposition \ref{CombinedSymplectic} states now that
the orbit of $\mathfrak{N}$ under the action of $\Sp(A[3])\ltimes A[3]$ is spanning a space of rank $51$ modulo $\Sym\oplus\Pi'$. However, by Lemma \ref{cleffinclassesdiv}, the intersection of that space with $\Sym^{sat}$ is one-dimensional and the intersection with $\Pi'^{sat}$ has dimension $31$, so we are left with $19$ linearly independent elements which provide 19 elements in $\frac{H^4(K_2(A),\Z)}{\Sym^{sat}\oplus\Pi'^{sat}}$.
These 19 independent classes can be enumerated using a computer, see Proposition~\ref{XIX}. 

Now, we will check that we found all the classes in $H^4(\X,\Z)$.
Let $\Sym^{over}$ be the overlattice of $\Sym$ obtained by adding all the classes from Proposition \ref{classedivisibleSym}:
\begin{equation}
\frac{\Sym}{\Sym^{over}}=(\Z/2\Z)^{\oplus 7}\oplus(\Z/3\Z)^{\oplus 8}. 
\label{symover} 
\end{equation}
Let denote by $F$ the over-lattice of $\Sym^{sat}\oplus \Pi'^{sat}$ obtained by adding $Z_0$ and the thirds of all classes of Proposition \ref{XIX}:
\begin{equation}
\frac{F}{\Sym^{sat}\oplus \Pi'^{sat}}=(\Z/27\Z)\oplus(\Z/3\Z)^{\oplus 19}.
\label{Fpietsym}
\end{equation}
We have to show that 
\begin{equation}
\Sym^{over}=\Sym^{sat},\ \Pi'^{over}=\Pi'^{sat}
\label{PietSym}
\end{equation}
and $F=H^4(\X,\Z)$.
It can be seen calculating the descriminents. 
We have $\Sym^{over}\subset \Sym^{sat}$ and $\Pi'^{over}\subset\Pi'^{sat}$, hence to prove (\ref{PietSym}), we only have to show that 
$\discr \Sym^{over}=\discr \Sym^{sat}$ and $\discr \Pi'^{over}=\discr\Pi'^{sat}$.
By (\ref{DiscrSym}), (\ref{symover}) and (\ref{squareDiscr}), the lattice $\Sym^{over}$ has discriminant $3^{22}$. Moreover by (\ref{discrPi}), (\ref{Piover}) and (\ref{squareDiscr}), 
the lattice $\Pi'^{over}$ has discriminant $3^{22}$. Therefore:
\begin{equation}
\discr\left(\Sym^{over}\oplus \Pi'^{over}\right)=3^{44}.
\label{discrover}
\end{equation}
It follows that $\discr \left(\Sym^{sat}\oplus \Pi'^{sat}\right) | 3^{44}$ (here the symbol "$|$" is the divisibility relation).
Hence, by (\ref{squareDiscr}) and \ref{Fpietsym}, $\discr F | 1$, so necessarily $\discr F=1$. Hence necessarily $F=H^4(K_2(A),\Z)$. Moreover, $\discr \left(\Sym^{sat}\oplus \Pi'^{sat}\right)=3^{44}$ 
which show by (\ref{discrover}) that $\Sym^{over}=\Sym^{sat} \text{ and } \Pi'^{over}=\Pi'^{sat}.$ 

We summarize the description of the integral basis of $H^{4}(K_2(A),\Z)$ in the following theorem.
\begin{thm}\label{integralbasistheorem}
Let $A$ be an abelian variety. We use Notation \ref{BasisH2KA} and \ref{TheZs}. 
\begin{enumerate}
\item 
Let $\Sym^{sat}$ be the primitive overlattice of $\Sym^2\left(H^2\left(\X,\Z\right)\right)$ in $H^4(K_2(A),\Z)$.
The group $\frac{\Sym^{sat}}{\Sym^2\left(H^2\left(\X,\Z\right)\right)}=(\Z/2\Z)^{\oplus 7}\oplus(\Z/3\Z)^{\oplus 8}$ is generated by the elements:
$$\frac{e \cdot y}{3},\ \frac{y^2 - \frac{1}{3} e\cdot y}{2} \text{ for } y\in\{u_1,u_2,v_1,v_2,w_1,w_2\},\ 
\frac{e^2}{3} \text{ and } \frac{u_{1}\cdot u_{2}+v_{1}\cdot v_{2}+w_{1}\cdot w_{2}}{6}.$$
\item
Let $\Pi'$ be the lattice from Definition \ref{defiPi} and let $\Pi'^{sat}$ be the primitive over lattice of $\Pi'$ in $H^4(K_2(A),\Z)$.
The group $\frac{\Pi'^{sat}}{\Pi'\ \ \ }=(\Z/3\Z)^{\oplus 31}$ is generated by the classes:
$$\frac{1}{3}\sum_{\tau\in\Lambda} \Big(Z_{\tau} - Z_{\tau+\tau'}\Big),
$$
with $\Lambda$ a non-isotropic group and $\tau'\in A[3]$. Moreover a basis of $\frac{\Pi'^{sat}}{\Pi'\ \ \ }$ is provided by the 31 classes described in Proposition \ref{XXXI}. 
\item
We have 
$$\frac{H^4(K_2(A),\Z)}{\Sym^{sat}\oplus\Pi'^{sat}}=\left(\frac{\Z}{27\Z}\right)\oplus\left(\frac{\Z}{3\Z}\right)^{\oplus 19}.$$
Moreover, this group is generated by the class $Z_0$ and the 19 classes described in Proposition \ref{XIX}.
\end{enumerate}
\end{thm}
Moreover since $\Sym^{over}=\Sym^{sat}$, from the proofs of Proposition \ref{ChernSym}, \ref{classedivisibleSym} and Remark \ref{afterClassuvw}, we obtain the following corollary.
\begin{corollary}\label{SymSatImage}
The image of $H^4(A\hilb{3},\Z)$ under $\theta^*$ is equal to $\Sym^{sat}$. \qed
\end{corollary}

\subsection{Conclusion on the morphism to the Hilbert scheme}
Let us summarize our results on $\theta^*$:
\begin{theorem}\label{thetaTheorem}
Let $A$ be an abelian variety and $(b_i)\subset H^2(A,\Z)$ an integral basis. Let $\theta: \kum{A}{2}\hookrightarrow A^{[3]}$ be the embedding. We also use Notation \ref{TorusClasses}.

The homomorphism $\theta^*:H^*(A\hilb{3},\Z)\rightarrow H^*(\kum{A}{2},\Z)$ of graded rings is surjective in every degree except $4$. Moreover, the image of $H^4(A\hilb{3},\Z)$ is the primitive overlattice of $\Sym^2(H^2(\kum{A}{2},\Z))$. 
The kernel of $\theta^*$ is the ideal generated by $H^1(A\hilb{3},\Z)$.
The image by $\theta^*$ of the following integral classes provide a basis of $\im\theta^*$:
\begin{center}
\begin{tabular}{c|l|l}
degree & preimage of class & alternative name  \\
\hline
0 & $\frac{1}{6} \q_1(1)^3\vac$ & 1 \\
\hline
2 &  $\frac{1}{2}\q_1(b_i) \q_1(1)^2\vac$ for $1\leq i\leq 6$ & $j(b_i)$ \\
 & $\frac{1}{2} \q_2(1)\q_1(1)\vac $  & $e$\\
\hline
3 & $\frac{1}{2}\q_1(a^*_i)\q_1(1)^2\vac$ & \\
  & $\frac{1}{2}\q_2(a_i)\q_1(1)\vac$ & \\
\hline
4 & $\q_1(b_i)\q_1(b_j)\q_1(1)\vac$ for $1\leq i\leq j\leq 6$, but $(b_i,b_j)\neq(a_1a_2,a_3a_4)$ &\\
  & $\frac{1}{2}\q_1(x)\q_1(1)^2\vac$ (instead of the missing case above)  & $Y_p$\\
  & $\frac{1}{2}\left(\q_1(b_i)^2-\q_2(b_i)\right)\q_1(1)\vac$ & \\
  & $\frac{1}{3} \q_3(1)\vac$ & $W$ \\
\hline
5 & $\q_1(a_ia_j)\q_1(a_j^*)\q_1(1) \vac$ for any choice of $j\neq i$ &\\
  & $\q_2(a^*_i)\q_1(1)\vac $ &\\
\hline
6 & $\q_1(a_i^*)\q_1(a_j^*)\q_1(1)\vac$ for $1\leq i< j\leq 4$ & \\
  & $\q_2(x)\q_1(1)\vac$ & \\
\hline
8 & $\q_1(x)^3\vac$ & top class
\end{tabular}
\end{center}
\end{theorem}
\begin{proof}
The table is established by the following results:
For degree 2, see Proposition \ref{H2Sur}. Since the Poincar\'e duality pairing on $\kum{A}{2}$ can be evaluated using projection formula (\ref{projectionFormula}), the dual classes of degree 6 are easily computed.
The odd degrees are treated by Proposition \ref{oddcohomology}. Classes of degree 4 are studied in Sections \ref{SyminH4} and \ref{integralbasisH4}. The classes are chosen in a way that they give a basis of $\Sym^{sat}$, which is possible by Corollary \ref{SymSatImage}. The condition $(b_i,b_j)\neq(a_1a_2,a_3a_4)$ is more or less arbitrary, but we had to remove one class to avoid a relation of linear dependence.

The kernel of $\theta^*$ is described by the Propositions \ref{annihilator} and \ref{Annihideal}.
\end{proof}

\section{Symplectic involutions on $K_{2}(A)$}\label{Involution} 
By Section \ref{oddcohoK2} and \cite{MongWanTari}, it is now possible to classify the symplectic involutions on $K_{2}(A)$.

Let $X$ be an irreducible symplectic manifold. Denote $$\nu: \Aut (X)\rightarrow \Aut (H^{2}(X,\Z))$$
the natural morphism. Hassett and Tschinkel (Theorem 2.1 in \cite{Hassett}) have shown that $\Ker \nu$ is a deformation invariant. 
Let $X$ be an irreducible symplectic fourfold of Kummer type. Then Oguiso in \cite{Oguiso} has shown that $\Ker \nu =(\Z/3\Z)^{\oplus3}\rtimes\Z/2\Z$.

Let $A$ be an abelian variety and $g$ an automorphism of $A$. Let us denote by $T_{A[3]}$ the group of translations of $A$ by elements of $A[3]$. 
If $g\in  T_{A[3]}\rtimes\Aut_{\Z} (A)$, then $g$ induces a natural automorphism on $K_2(A)$. 
We denote the induced automorphism by $g^{[[3]]}$. If there is no ambiguity, we also denote the induced automorphism by the same letter $g$
to avoid too complicated formulas.

When $X=K_2(A)$, $\Ker \nu$ can be precisely described using: 
\begin{cor}\cite[Corollary 3.3]{BNS2}\label{BNS2cor}
 Let $A$ be a complex torus and $\nu: \Aut (K_2(A))\rightarrow \Aut (H^{2}(K_2(A),\Z))$ the natural map.
 Then $$\Ker \nu = T_{A[3]}\rtimes(-\id_A)^{[[3]]}.$$ 
 \end{cor}
\subsection{Torelli theorem}
To prove Theorem \ref{SymplecticInvo} (i) we will need to use the global Torelli theorem for IHS manifolds stated by Markman in \cite{Markmansurvey}. We recall this theorem in this section.

Let $X_1$ and $X_2$ be IHS manifolds. We say that the isomorphism $f: H^*(X_1,\Z) \stackrel{\cong}{\rightarrow} H^*(X_2,\Z)$ is a \emph{parallel-transport operator} if there exist a smooth and proper family $\pi : \mathcal{X}\rightarrow B$ of IHS manifolds over an analytic base $B$, points $b_i\in B$, isomorphisms $\psi_i : X_i \rightarrow \mathcal{X}_{b_i}$ for $i=1,2$, and a continuous path $\gamma: \left[ 0, 1 \right] \rightarrow B$, satisfying $\gamma(0)=b_1$, $\gamma(1)=b_2$, such that the parallel transport in the local system $R\pi_* \Z$ along $\gamma$ induces the homomorphism $\psi_{2*}\circ f\circ \psi_1^{*}: H^*(\mathcal{X}_{b_1},\Z)\stackrel{\cong}{\rightarrow} H^*(\mathcal{X}_{b_2},\Z)$. 
An isomorphism $g: H^{2}(X_1,\Z)\stackrel{\cong}{\rightarrow} H^{2}(X_2,\Z)$ is said to be a \emph{parallel-transport operator} if it is the $2$-th graded summand of a parallel-transport operator $f$ as above. 
Remark that an automorphism $g: H^{2}(X,\Z)\stackrel{\cong}{\rightarrow} H^{2}(X,\Z)$ of the second cohomology group of an IHS manifolds $X$ which is a parallel-transport operator is a monodromy operator. We denote by $\Mon^2(X)\subset \mathcal{O}(H^2(X,\Z)$ the subgroup of monodromy operators of $X$.

\begin{thm}\cite[Theorem 1.3]{Markmansurvey} \label{Torelli}
Let $X$ and $Y$ be two IHS manifolds, which are deformation equivalent.
Then,
$X$ and $Y$ are bimeromorphic if and only if there exists a parallel transport operator $f : H^2(X,\Z) \to H^2(Y,\Z)$, which is a Hodge isometry.
\end{thm}

To use this theorem is important to know the group $\Mon^2(X)$. In the  case of an irreducible symplectic manifold of Kummer type $X$, the group $\Mon^2(X)$ was described by Mongardi in \cite{MongardiMono}. Let $O^{+}(H^2(X,\Z))$ be the sub-group of $O(H^2(X,\Z)$ that preserve the orientation of the positive cone. Let $\mathcal{W}(X)$ be the sub-group of $O^{+}(H^2(X,\Z))$ acting as $\pm1$ on the discriminant group $A_{H^{2}(X,\Z)}:=\frac{H^{2}(X,\Z)^{*}}{H^{2}(X,\Z)}$. Let $\chi$ be the character corresponding to the action on $A_{H^{2}(X,\Z)}$. We denote by $\mathcal{N}(X)$ the kernel of $\det\circ \chi:\mathcal{W}(X)\rightarrow \left\{\pm1\right\}$.
\begin{thm}\cite[Theorem 2.3]{MongardiMono}\label{MonodromyMong}
let $X$ be an irreducible symplectic manifold of Kummer type. Then $\Mon^2(X)=\mathcal{N}(X)$.
\end{thm}
Let denote by $\Lambda_n$ the lattice $U^{\oplus 3}\oplus(-2-2n)$. Let $X$ be an irreducible symplectic $2n$-fold of Kummer type, an isometry $\varphi:H^2(X,\Z)\rightarrow \Lambda_n$ is called a \emph{mark} and the couple $(X,\varphi)$ is called a \emph{marked irreducible symplectic $2n$-fold of Kummer type}. We denote by $\mathcal{M}_{\Lambda_n}$ the moduli space of marked irreducible symplectic $2n$-fold of Kummer type. Moreover, we recall that the period map is defined as follows:

$$\xymatrix@R0pt{
\mathscr{P}:& \mathcal{M}_{\Lambda_n}\ar[r] & \Omega_{\Lambda}:=\left\{\left.x\in\mathbb{P}(\Lambda_n\otimes\C)\right|\ x^2=0\ \text{and}\ x\cdot\overline{x}>0\right\}\\
&(X,\varphi)   \ar[r] & \varphi(H^{0}(X,\Omega_X^2)).
   }$$
\begin{cor}\label{Torellicoro}
Let $X$ be an irreducible symplectic $2n$-fold of Kummer type, with $n+1$ a prime power. Let $A$ be a 2-dimensional complex torus, we denote by $A^*$ its dual torus.
If there exists a Hodge isometry $f : H^2(X,\Z) \to H^2(K_n(A),\Z)$, then $X$ is bimeromorphic to $K_n(A)$ or to $K_n(A^*)$.
\end{cor}
\begin{proof}
We are going to use Theorem \ref{Torelli}. To do so, we have to understand when $f$ is a parallel transport operator.
Let $\varphi$ be a mark of $X$, since $X$ is of Kummer type, we can find a mark $\psi$ of $K_n(A)$ such that $(K_n(A),\psi)$ and $(X,\varphi)$ are in the same connected component $\mathcal{M}_{\Lambda_n}^o$ of the moduli space $\mathcal{M}_{\Lambda_n}$. In particular, it means that $\psi^{-1}\circ\varphi $ is a parallel transport operator. Now, we are going to consider $f\circ\varphi^{-1}\circ \psi\in O(H^2(K_n(A),\Z))$. We can assume that $f\circ\varphi^{-1}\circ \psi\in O^+(H^2(K_n(A),\Z))$, by changing $f$ by $-f$ if necessary.
Moreover, since $n+1$ is a prime power, by Lemma 4.3 of \cite{Markmanou}, we can find $\nu\in \mathcal{N}(K_n(A))=\Mon(K_n(A))$ such that $f\circ\varphi^{-1}\circ \psi\circ\nu(\delta)=\delta$, where $\delta$ is half the class of the diagonal divisor in $K_n(A)$. Hence we can exchange the mark $\psi$ for the mark $\psi':=\psi\circ\nu$ keeping $(K_n(A),\psi')$ in $\mathcal{M}_{\Lambda_n}^o$.

Now, as done in Section 4 of \cite{Markmanou}, we consider $A^*$ the dual complex torus of $A$. Then $H^1(A^*,\Z)$ is isomorphic to $H^1(A,\Z)^*$. Let $\overline{\tau}$ be the composition of natural isomorphisms $H^2(A,\Z)\simeq H^2(A,\Z)^*\simeq H^2(A^*,\Z)$, where the first isomorphism is induced by the intersection pairing. Let $\tau: H^2(K_2(A),\Z)\rightarrow H^2(K_2(A^*),\Z)$ be the isomorphism restricting to $\delta^\bot$ as $-\overline{\tau}$ and mapping the class $\delta$ to half the class of the diagonal divisor in $K_2(A^*)$. The isometry $\overline{\tau}$ is also constructed in Lemma 3 of \cite{Shioda} and it is shown that it preserves the period. Moreover by proposition 4.6 of \cite{Markmanou}, 
$f\circ\varphi^{-1}\circ \psi'\in \mathcal{N}(K_n(A))=\Mon(K_n(A))$ or $\tau\circ f\circ\varphi^{-1}\circ \psi'\in \mathcal{N}(K_n(A))=\Mon(K_n(A))$. So $f$ or $\tau\circ f$ is a parallel transport operator. Then, we conclude the proof with Theorem \ref{Torelli}.
\end{proof}
\subsection{Uniqueness and fixed locus}
\begin{thm}\label{SymplecticInvo}
Let $X$ be an irreducible symplectic fourfold of Kummer type and $\iota$ a symplectic involution on $X$. Then:
\begin{enumerate}
\item
We have $\iota\in \Ker \nu$.
\item
Let $A$ be an abelian surface. Then 
the couple $(X,\iota)$ is deformation equivalent to $(K_2(A),t_\tau \circ (-\id_A)^{[[3]]})$,
where $t_\tau$ is the morphism induced on $K_2(A)$ by the translation by $\tau\in A[3]$.
\item
The fixed locus of $\iota$ is given by a K3 surface and 36 isolated points.
\end{enumerate}
\end{thm}
\begin{proof}[Proof of (i)]

If $\iota\notin \Ker \nu$, by the classification of Section 5 of \cite{MongWanTari}, the unique possible action of $\iota$ on $H^{2}(X,\Z)$ is given by $H^{2}(X,\Z)^{\iota}=U\oplus (2)^{\oplus2}\oplus(-6)$. We will show that it is impossible. Let us assume that $H^{2}(X,\Z)^{\iota}=U\oplus (2)^{\oplus2}\oplus(-6)$, we will find a contradiction. Let denote by $\Lambda_\iota$ the sublattice $U\oplus (2)^{\oplus2}\oplus(-6)$ of $\Lambda_2$. The proof is organized as follows. First, we will show that $(X,\iota)$ deforms to a couple $(K_2(A),i)$ where $i$ is a natural involution (that is an involution induced by an involution on $A$). Then we will see that this is impossible using Section 4 of \cite{MongWanTari0} and Corollary \ref{actionH3}. 

As explained after Remark 2 of \cite{MongardiDef}, we can find $(Y,\iota')$ which is deformation equivalent to $(X,\iota)$ such that there exist two marks $\varphi$, $\varphi'$ and a generic complex torus $A$ with 
$$\mathscr{P}(Y,\varphi)=\mathscr{P}(K_2(A),\varphi').$$
We recall in few words how this is shown in \cite{MongardiDef}. First, since $\iota$ is symplectic, a period of $X$ is contained in the sub-period domain $\Omega_{\Lambda_\iota}:= \left\{\left.x\in\mathbb{P}(\Lambda_\iota\otimes\C)\right|\ x^2=0\ \text{and}\ x\cdot\overline{x}>0\right\}$.
Moreover we can find $x\in \Omega_{\Lambda_\iota}$ such that $x^\bot\cap\Lambda_\iota=\Z d$ with $d^2=-6$ and $d\cdot \Lambda_2=6\Z$. Furthermore we can link $x$ to a period of $X$ by a chain of twisted lines in $\Omega_{\Lambda_\iota}$ (see for instance Proposition 3.7 of \cite{HuybrechtsSurvey}). 
Then Remark 1 of \cite{MongardiDef} explains that there exists a couple $(Y,\iota')$ deformation equivalent to $(X,\iota)$ and a mark $\varphi$ such that $\mathscr{P}(Y,\varphi)=x$. 
Moreover, $d^\bot$ in $\Lambda_2$ is isomorphic to the torus lattice, so by surjectivity of the period map of 2-dimensional torus, we can find a torus $A$ and a mark $\varphi'$ such that $\mathscr{P}(K_2(A),\varphi')=x$. 
Remark, in addition, that we have chosen $x$ such that the $NS(A)$ is minimal, that is $NS(A)\simeq \mathcal{S}_{\iota'}:=(H^2(Y,\Z)^{\iota'})^\bot$. To be more precise, if we denote by $j:H^2(K_2(A),\Z)\rightarrow H^2(A,\Z)\oplus\Z\delta$ the natural Hodge isometry with $\delta$ half the class of the diagonal divisor in $K_2(A)$, then 
\begin{equation}
\NS (A)=j\circ\varphi'^{-1}\circ \varphi(\mathcal{S}_{\iota'}).
\label{MongNS}
\end{equation}

Then Corollary \ref{Torellicoro} implies that $Y$ is bimeromorphic to $K_2(A)$ or to $K_2(A^*)$.
Let assume that we have a bimeromorphism $r:Y\rightarrow K_2(A)$ if $Y$ is bimeromorphic to $K_2(A^*)$, the proof is similar. 
Then the involution $\iota'$ provides an involution $i:=r\circ\iota'\circ r^{-1}$ not necessarily regular on $K_2(A)$.

On the other hand by (\ref{MongNS}), $\NS(A)\simeq (-2)^{\oplus2}$.
Now we construct an involution $g$ on $H^{2}(A,\Z)$ given by $-\id$ on $(-2)^{\oplus2}$ and $\id$ on $((-2)^{\oplus2})^{\bot}$ and extended to an involution on $H^{2}(A,\Z)$ by Corollary 1.5.2 of \cite{Lattice}. Then by Theorem 1 of \cite{Shioda}, $g$ provides a symplectic automorphism on $A$ with: $H^{2}(A,\Z)^{g}=((-2)^{\oplus2})^\bot=U\oplus (2)^{\oplus2}$.
It follows from the classification of Section 4 of \cite{MongWanTari0}, that $A=\C/\Lambda$ with $\Lambda=\left\langle (1,0),(0,1),(x,-y),(y,x)\right\rangle$, $(x,y)\in \C^2\smallsetminus \R^2$ and $g=\left(
\begin{array}{cc}
0 & -1\\
1 & 0 
\end{array} \right)$. 
We are going to show that $g\circ i$ acts trivially on $H^2(K_2(A),\Z)$. The automorphisms $g$ and $i$ are both symplectic, so acts trivialy on $T_A:=\NS(K_2(A))^\bot$ the transcendental lattice of $K_2(A)$. Hence, we only have to prove that $g\circ i$ acts trivially on $\NS(K_2(A))$. 
We have $\NS(K_2(A))=j^{-1}(\NS(A))\oplus\Z\delta\simeq (-2)^{\oplus2}\oplus(-6)$. We know that $g$ acts on $\NS(K_2(A))$ by fixing $\delta$ and by $-\id$ on $j^{-1}(\NS(A))$. Moreover, we know that $\mathcal{S}_i:=(H^2(K_2(A),\Z)^i)^\bot\subset \NS(K_2(A))$ and $\mathcal{S}_i\simeq (-2)^{\oplus2}$. Let $(\alpha,\beta)$ be a basis of $\mathcal{S}_i$, we can write $i(\delta)=\lambda\delta+\mu_1\alpha+\mu_2\beta$ with $\lambda$, $\mu_1$, $\mu_2$ integers. Then $i(\delta)^2=-6=-6\lambda^2-2\mu_1^2-2\mu_2^2$, so necessarily, $\lambda=1$ and $\mu_1=\mu_2=0$.
This implies that $i(\delta)=\delta$ and $\mathcal{S}_i=j^{-1}(\NS(A))$. That proves that $g\circ i$ acts trivially on $H^2(K_2(A),\Z)$.

Hence by Corollary 3.3 and Lemma 3.4 of \cite{FujikiK}, $g\circ i$ extends to an automorphism of $K_2(A)$. In particular, $i=g^{-1}\circ g\circ i$ extends to a symplectic involution on $K_2(A)$ and $g\circ i\in \Ker \nu$. It allows us to compare the action of $i$ and $g$ on $H^{3}(K_2(A),\Z)$.

By Corollary \ref{actionH3}, $t_\tau$ acts trivially on $H^{3}(K_2(A),\Z)$. Hence by Corollary \ref{BNS2cor}, we have necessarily:
$$g^{*}_{|H^{3}(K_2(A),\Z)}=i^{*}_{|H^{3}(K_2(A),\Z)}\circ (-\id_A)^{*}_{|H^{3}(K_2(A),\Z)}\ \text{or}\ g^{*}_{|H^{3}(K_2(A),\Z)}=i^{*}_{|H^{3}(K_2(A),\Z)}.$$
But by Corollary \ref{actionH3}, $g^{*}_{|H^{3}(K_2(A),\Z)}$ has order 4 and $i^{*}_{|H^{3}(K_2(A),\Z)}\circ (-\id_A)^{*}_{|H^{3}(K_2(A),\Z)}$ and $i^{*}_{|H^{3}(K_2(A),\Z)}$ have order 2, which is a contradiction.
\vspace{3pt}\\
\emph{Proof of (ii).} 
Let $X$ be a irreducible symplectic fourfold of Kummer type and $\iota$ a symplectic involution on $X$. By (i) of the above theorem, we have $\iota\in\Ker \nu$. Then by Theorem 2.1 of \cite{Hassett}, the couple $(X,\iota)$ deforms to a couple $(K_2(A),\iota')$ with $A$ an abelian surface and $\iota'\in\Ker \nu$ a symplectic involution on $ K_2(A)$. Then we conclude with Corollary \ref{BNS2cor}.
\vspace{3pt}\\
\emph{Proof of (iii).} 
Let $A$ be an abelian surface. By Section 1.2.1 of \cite{Tari}, the fixed locus of $t_\tau \circ (-\id_A)^{[[3]]}$ on $K_2(A)$ consists of a K3 surface an 36 isolated points. 
Now let $X$ be an irreducible symplectic fourfold of Kummer type and $\iota$ a symplectic involution on $X$. 
By (ii) of the above theorem, $\Fix\iota$ deforms to the disjoint union of a K3 surface and 36 isolated points. 
Moreover, $\iota$ is a symplectic involution, so the holomorphic 2-form of $X$ restricts to a non-degenerate holomorphic 2-form on $\Fix\iota$. So $\Fix\iota$ can only contain K3 surfaces, complex tori and isolated points.
It follows, necessarily for topological reason, that $\Fix\iota$ consists of a K3 surface and 36 isolated points.
\end{proof}
\begin{rmk}
With the same ideas as in proof of Theorem \ref{SymplecticInvo} (i), when $n+1$ is a prime power, we can show that a numerically standard symplectic automorphism on an irreducible symplectic $2n$-fold of Kummer type is standard (see \cite{MongardiDef} for the definition of standard and numerically standard).
\end{rmk}
\begin{rmk}\label{RemarkSymplecticInv}
\begin{itemize}
\item[(1)]
We also remark that the K3 surface fixed by $(t_\tau \circ (-\id_A))$ is given by the sub-manifold $$Z_{-\tau}=\overline{\left\{\left.(a_1,a_2,a_3)\ \right|\ a_1=-\tau,\ a_2=-a_3+\tau,\ a_2\neq -\tau\right\}}$$ defined in Section 4 of \cite{Hassett}. 
\item[(2)]
Considering the involution $-\id_A$,
the set $$\mathcal{P}:=\left\{\left.\xi\in K_{2}(A)\right|\ \Supp \xi= \left\{a_{1},a_{2},a_{3}\right\},\ a_{i}\in A[2]\smallsetminus \left\{0\right\}, 1\leq i\leq 3 \right\}$$ provides 35 fixed points and the vertex of $$W_{0}:=\left\{\left.\xi\in K_{2}(A)\right|\ \Supp \xi=\left\{0\right\}\right\}$$ supplies the 36th point. We denote by $p_1,...,p_{35}$ the points of $\mathcal{P}$ and by $p_{36}$ the vertex of $W_{0}$.
\end{itemize}
\end{rmk}

\subsection{Action on the cohomology}\label{actioncoh}
By Theorem \ref{SymplecticInvo}, we can assume that $X=K_2(A)$ and $\iota=-\id_A$. To consider $t_\tau \circ (-\id_A)$ instead of  $-\id_A$ only has the effect of exchanging the role of $[Z_0]$ and $[Z_{-\tau}]$.
Hence we do not lose any generality assuming that $\iota=-\id_A$. Now we calculate the invariants $l_i^j(K_{2}(A))$ defined in Definition-Proposition \ref{defiprop}. It will be used in Section \ref{BeauvilleForm}.

From Theorem \ref{SymplecticInvo} (1), the involution $\iota$ acts trivially on $H^{2}(K_{2}(A),\Z)$.
It follows that
\begin{equation}
l_{2}^2(K_{2}(A))=l_{1,-}^2(K_{2}(A))=0 \text{ and } l_{1,+}^2(K_{2}(A))=7.
\label{l22}
\end{equation}
From Corollary \ref{actionH3}, the involution $\iota$ acts as $-\id$ on $H^{3}(K_{2}(A),\Z)$.
It follows that
\begin{equation}
l_{2}^3(K_{2}(A))=l_{1,+}^3(K_{2}(A))=0 \text{ and } l_{1,-}^3(K_{2}(A))=8.
\label{l3}
\end{equation}
By Definition \ref{defiPi}, we have:
$$H^{4}(K_{2}(A),\Q)=\Sym^{2} H^{2}(K_{2}(A),\Q)\oplus^{\bot} \Pi'\otimes\Q,$$
where $\Pi'=\left\langle Z_{\tau}-Z_{0},\ \tau\in A[3]\smallsetminus \left\{0\right\}\right\rangle$. The involution
$\iota^*$ fixes $\Sym^{2} H^{2}(K_{2}(A),\Z)$ and $\iota^*(Z_{\tau}-Z_{0})=Z_{-\tau}-Z_{0}$. It provides the following proposition.
\begin{prop}\label{invariants}
We have $l_{1,-}^4(K_{2}(A))=0$, $l_{1,+}^4(K_{2}(A))=28$ and $l_{2}^4(K_{2}(A))=40$.
\end{prop}
\begin{proof}

Let $\mathcal{S}$ be the over-lattice of $\Sym^{2} H^{2}(K_{2}(A),\Z)$ where we add all the classes divisible by 2 in $H^{4}(K_{2}(A),\Z)$.
From Section \ref{integralbasisH4}, we know that the discriminant of $\mathcal{S}$ is not divisible by 2.
Hence, we have:
$$H^{4}(K_{2}(A),\F)=\mathcal{S}\otimes\F\oplus \Pi'\otimes\F.$$
Moreover, we have: $$\iota^{*}(Z_{\tau}-Z_{0})=Z_{-\tau}-Z_{0},$$
for all $\tau\in A[3]\smallsetminus \left\{0\right\}$.
Hence $\Vect_{\F}(Z_{\tau}-Z_{0},Z_{-\tau}-Z_{0})$ is isomorphic to $N_{2}$ as a $\F[G]$-module  (see the notation in Definition-Proposition \ref{defiprop}).
Moreover $H^{2}(K_{2}(A),\Z)$ is invariant by the action of $\iota$, hence $\Sym^{2} H^{2}(K_{2}(A),\Z)$ and $\mathcal{S}$ is also invariant by the action of $\iota$. 
It follows that $\mathcal{S}\otimes\F=\mathcal{N}_{1}$ and $\Pi'\otimes\F=\mathcal{N}_{2}$.
Since $\rk \mathcal{S}=28$, we have $l_{1,+}^{4}+l_{1,-}^{4}=28$.
However, $\mathcal{S}$ is invariant by the action of $\iota$, it follows that $l_{1,-}^{4}=0$ and $l_{1,+}^{4}=28$.
On the other hand $\rk \Pi'=80$, it follows that $l_{2}^{4}=40$.
\end{proof}
\section{Proof of Theorem \ref{theorem}}\label{BeauvilleForm}
Since all generalized Kummer fourfolds are deformation equivalent, by Theorem \ref{SymplecticInvo} all the couples $(X,\iota)$, where $X$ is a fourfold of Kummer type and $\iota$ a symplectic involution, are deformation equivalent.
Moreover, the Beauville--Bogomolov form is a topological invariant, 
hence without loss of generality it is enough to prove Theorem \ref{theorem} for a particular couple $(X,\iota)$.
We can assume that $X$ is a generalized Kummer fourfold and $\iota=-\id_A$. As it will be useful in proving Lemma \ref{Ddelta}, we can assume moreover that $A=E_\xi\times E_\xi$, where 
$$E_\xi:=\frac{\C}{\Lambda_0},$$
with $\xi:=e^{\frac{2i\pi}{6}}$ and $\Lambda_0 := \left<1,\xi\right>$.
This abelian surface has the interest to carry enough automorphisms.
\begin{definition}\label{elliptic6}
Define a group $G_\xi$ of automorphisms of $E_\xi\times E_\xi$ by the following generators in $\GL(2,\End(\Lambda_0))$:
\begin{align*}
g_1 &= \left( {\begin{array}{cc}
   \xi & 0 \\       0 & 1      
   \end{array} } \right),
 &
g_2 &= \left( {\begin{array}{cc}
   0 & 1 \\       1 & 0      
   \end{array} } \right),
 &
g_3 &= \left( {\begin{array}{cc}
   1 & 1 \\       0 & 1     
   \end{array} } \right).
\end{align*}
\end{definition}
For $A=E_\xi\times E_\xi$, let $V =A[2]$ be the (fourdimensional) $\mathbb F_2$-vector space of $2$-torsion points on $A$ and let $\mathfrak T$ be the set of planes in $V$. Note that by Remark~\ref{PlaneTriple} a plane in $V$ can be identified with an unordered triple $\{x,y,z\}$ with $0\neq x,y,z\in V$ and $x+y+z=0$. The action of $G_\xi$ on $A$ induces actions of $G_\xi$ on $A[2]$ and $\mathfrak T$. 
The following lemma will be used to prove Theorem \ref{fin}.
\begin{lemma}\label{orbitesG}
There are two orbits of $G_\xi$ on $\mathfrak T$, of cardinalities $5$ and $30$.
\end{lemma}
\begin{proof}
Note that the generators $g_2$ and $g_3$ exist because $A$ is of the form $E\times E$, while $g_1$ exists only in the special case $E=E_\xi$. Indeed, multiplication with $\xi$ induces a cyclic permutation on $E_\xi[2]$. 
The orbits can be explicitely determined by a suitable computer program. For verification, we give one of the orbits explicitely.
Denote $x_1,x_2,x_3$ the non-zero points in $E_\xi[2]$. The orbit of cardinality five is then given by
\begin{gather*}
\{(0,x_1),(0,x_2),(0,x_3)\} , \quad  \{(x_1,0),(x_2,0),(x_3,0)\},\quad  \{(x_1,x_1),(x_2,x_2),(x_3,x_3)\}  \\
 \{(x_1,x_2),(x_2,x_3),(x_3,x_1)\}, \quad  \{(x_1,x_3),(x_2,x_1),(x_3,x_2)\}. \qedhere
\end{gather*}
\end{proof}
\subsection{Overview on the proof of Theorem \ref{theorem} and notation}\label{nota}
The proof is divided into the following steps:
\begin{itemize}
\item[(1)]
First (\ref{l22}), (\ref{l3}), Proposition \ref{invariants} and Theorem \ref{utile'} will provide the $H^4$-normality of $(K_2(A),\iota)$ in Section \ref{H4}. The notion of $H^k$-normality is recalled in the beginning of the section.
\item[(2)]
The knowledge of the elements divisible by 2 in $\Sym^{2} H^{2}(K_2(A),\Z)$ from Section \ref{integralbasisH4} and the $H^4$-normality allow us to prove the $H^2$-normality of $(K_2(A),\iota)$ in Section \ref{H2}. 
\item[(3)]
An adaptation of the $H^2$-normality (Lemma \ref{primitive2}) and several lemmas in Section \ref{basisinteK'} will provide an integral basis of $H^{2}(K',\Z)$ (Theorem \ref{fin}).
\item[(4)]
Knowing an integral basis of $H^{2}(K',\Z)$, we end the calculation of the Beauville--Bogomolov form in Section \ref{beauK'} using intersection theory and the Fujiki formula.
\end{itemize}

Now we provide some notation that we will be used during the proof.
Let $K_{2}(A)$ be a generalized Kummer fourfold endowed with the symplectic involution $\iota$ induced by $-\id_A$.
We denote by $\pi$ the quotient map $K_{2}(A)\rightarrow K_{2}(A)/\iota$.
From Theorem \ref{SymplecticInvo}, we know that the singular locus of the quotient $K_{2}(A)/\iota$ is the K3 surface, image by $\pi$ of $Z_{0}$, and 36 isolated points. We denote $\overline{Z_{0}}:=\pi(Z_{0})$. 
We consider $r':K'\rightarrow K_{2}(A)/\iota$ the blow-up of $K_{2}(A)/\iota$ in $\overline{Z_{0}}$ and we denote by $\overline{Z_{0}}'$ the exceptional divisor.
We also denote by $s_{1}:N_{1}\rightarrow K_{2}(A)$ the blowup of $K_{2}(A)$ in $Z_{0}$; and denote by $Z_{0}'$ the exceptional divisor in $N_{1}$. Denote by $\iota_{1}$ the involution on $N_{1}$ induced by $\iota$. We have $K'\simeq N_{1}/\iota_{1}$, and we denote $\pi_{1}:N_{1}\rightarrow K'$ the quotient map.

Consider the blowup $s_{2}:N_{2}\rightarrow N_{1}$ of $N_{1}$ in the 36 points $p_1,...,p_{36}$ fixed by $\iota_{1}$ and the blowup $\widetilde{r}:\widetilde{K}\rightarrow K'$ of $K'$ in its 36 singulars points. We denote the exceptional divisors by $E_{1},...,E_{36}$ and $D_{1},...,D_{36}$ respectively. We also denote $\widetilde{\overline{Z_{0}}}=\widetilde{r}^{*}(\overline{Z_{0}}')$ and $\widetilde{Z_{0}}=s_{2}^{*}(Z_{0}')$.
Denote $\iota_{2}$ the involution induced by $\iota$ on $N_{2}$ and $\pi_{2}:N_{2}\rightarrow N_{2}/\iota_{2}$ the quotient map. 
We have $N_{2}/\iota_{2}\simeq \widetilde{K}$. To finish, we denote $V=K_{2}(A)\smallsetminus \Fix \iota$ and $U=V/\iota$. We collect this notation in a commutative diagram
\begin{equation}
\xymatrix{
 \widetilde{K}\ar[r]^{\widetilde{r}} & K' \ar[r]^{r'}&  K_{2}(A)/\iota&U\ar@{_{(}->}[l]\\
  N_{2}\ar@(dl,dr)[]_{\iota_{2}} \ar[r]^{s_{2}} \ar[u]^{\pi_{2}}& N_{1}\ar@(dl,dr)[]_{\iota_{1}} \ar[r]^{s_{1}} \ar[u]^{\pi_{1}} & K_{2}(A)\ar@(dl,dr)[]_{\iota}\ar[u]^{\pi}&V\ar[u]\ar@{_{(}->}[l]
   }
   \label{commutativediagram}
\end{equation} 
Also, we set $s=s_2\circ s_1$ and $r=\widetilde{r}\circ r'$. We denote also $e$ the half of the class of the diagonal in $H^{2}(K_2(A),\Z)$ as in Notation \ref{BasisH2KA}.

\begin{rmk}\label{commut2}
We can commute the push-forward maps and the blow-up maps as proved in Lemma 3.3.21 of \cite{Lol}.
Let $x\in H^{2}(N_1,\Z)$, $y\in H^{2}(K_2(A),\Z)$, we have:
$$\pi_{2*}(s_2^{*}(x))=\widetilde{r}^{*}(\pi_{1*}(x)),$$
$$\pi_{1*}(s_1^{*}(y))=r'^{*}(\pi_{*}(y)),$$
\end{rmk}
Moreover, we will also use the notation provided in Notation \ref{BasisH2KA} and in Section \ref{Middle}.
\subsection{The couple $(K_{2}(A),\iota)$ is $H^{4}$-normal}\label{H4}
We will use the notion of $H^k$-normality from Definition 3.3.4 of \cite{Lol} that we recall here.
\begin{defi}
Let $X$ be a compact complex manifold and $\iota$ be an involution. 
Let $0\leq k\leq 2n$, and assume that $H^{k}(X,\Z)$ is torsion free. 
Then if the map $\pi_{*}:H^{k}(X,\Z)\rightarrow H^{k}(X/G,\Z)/\tors$ is surjective, we say that $(X,\iota)$ is \emph{$H^{k}$-normal}.
\end{defi}
\begin{rmk}\label{Hnormal}
The $H^{k}$-normal property is equivalent to the following property.
For $x\in H^{k}(X,\Z)^{\iota}$, $\pi_{*}(x)$ is divisible by 2 if and only if there exists $y\in H^{k}(X,\Z)$ such that 
$x=y+\iota^{*}(y)$.
\end{rmk}
\begin{defi}\label{negligible}
Let $X$ be a compact complex manifold of dimension $n$ and $G$ an automorphism group of prime order $p$. 
\begin{itemize}
\item[1)]
We will say that $\Fix G$ is negligible if the following conditions are verified:
\begin{itemize}
\item[$\bullet$]
$H^{*}(\Fix G,\Z)$ is torsion-free.
\item[$\bullet$]
$\codim \Fix G\geq \frac{n}{2}+1$.
\end{itemize}
\item[2)]
We will say that $\Fix G$ is almost negligible if the following conditions are verified:
\begin{itemize}
\item[$\bullet$]
$H^{*}(\Fix G,\Z)$ is torsion-free.
\item[$\bullet$]
$n$ is even and $n\geq 4$.
\item[$\bullet$]
$\codim \Fix G =\frac{n}{2}$, and the purely $\frac{n}{2}$-dimensional part of $\Fix G$ is connected and simply connected. We denote the $\frac{n}{2}$-dimensional component by $Z$.
\item[$\bullet$]
The cocycle $\left[Z\right]$ associated to $Z$ is primitive in $H^{n}(X,\Z)$.
\end{itemize}
\end{itemize}
\end{defi}
We will use the following theorem (Corollary 3.5.18 of \cite{Lol}) to prove the $H^4$-normality of $(K_{2}(A),\iota)$.
\begin{thm}\label{utile'}
Let $G$ be a group of order 2 acting by automorphisms on a Kähler manifold $X$ of dimension $2n$. 
We assume:
\begin{itemize}
\item[i)]
$H^{*}(X,\Z)$ is torsion-free,
\item[ii)]
$\Fix G$ is negligible or almost negligible,
\item[iii)]
$l_{1,-}^{2k}(X)=0$ for all $1\leq k \leq n$, and
\item[iv)]
$l_{1,+}^{2k+1}(X)=0$ for all $0\leq k \leq n-1$, when $n>1$.
\item[v)]

$l_{1,+}^{2n}(X)+2\left[\sum_{i=0}^{n-1}l_{1,-}^{2i+1}(X)+\sum_{i=0}^{n-1}l_{1,+}^{2i}(X)\right]
= \sum_{k=0}^{\dim \Fix G}h^{2k}(\Fix G,\Z).$

\end{itemize}
Then $(X,G)$ is $H^{2n}$-normal.
\end{thm}
\begin{prop}\label{H4norm}
The couple $(K_{2}(A),\iota)$ is $H^{4}$-normal.
\end{prop}
\begin{proof}
We apply Theorem \ref{utile'}.
\begin{itemize}
\item[i)]
By Theorem \ref{torsion}, $H^{*}(K_{2}(A),\Z)$ is torsion-free. 
\item[ii)]
From Remark \ref{RemarkSymplecticInv} (1), we know that the connected component of dimension 2 of $\Fix \iota$  is given by $Z_{0}$ which is a K3 surface, hence is simply connected. 
Moreover by Proposition 4.3 of \cite{Hassett} $Z_{0}\cdot Z_{\tau}=1$ for all $\tau\in A[3]\smallsetminus \left\{0\right\}$. Hence the class of $Z_{0}$ in $H^{4}(K_{2}(A),\Z)$ is primitive. It follows that $\Fix \iota$ is almost negligible (Definition \ref{negligible}). 
\item[iii)]
By (\ref{l22}) and Proposition \ref{invariants}, we have $l_{1,-}^{2}(K_{2}(A))=l_{1,-}^{4}(K_{2}(A))=0.$
\item[iv)]
By (\ref{l3}), we have $l_{1,+}^{3}(K_{2}(A))=0.$ Moreover $H^{1}(K_{2}(A))=0$, so $l_{1,+}^{1}(K_{2}(A))=0.$
\item[v)]
We have to check the following equality:
\begin{align*}
&l_{1,+}^{4}(K_{2}(A))+2\left[l_{1,-}^{1}(X)+l_{1,-}^{3}(X)+l_{1,+}^{0}(X)+l_{1,+}^{2}(X)\right]\\
&= 36h^{0}(pt)+h^{0}(Z_{0})+h^{2}(Z_{0})+h^{4}(Z_{0}).
\end{align*}
By (\ref{l22}), (\ref{l3}) and Proposition \ref{invariants}:
$$l_{1,+}^{4}(K_{2}(A))+2\left[l_{1,-}^{1}(X)+l_{1,-}^{3}(X)+l_{1,+}^{0}(X)+l_{1,+}^{2}(X)\right]=28+2(8+1+7)=60.$$
Moreover since $Z_{0}$ is a K3 surface, we have:
$$36h^{0}(pt)+h^{0}(Z_{0})+h^{2}(Z_{0})+h^{4}(Z_{0})=36+1+22+1=60.$$
\end{itemize}
It follows from Corollary \ref{utile'} that $(K_{2}(A),\iota)$ is $H^4$-normal.
\end{proof} 
\begin{rmk}\label{primitive1}
As explained in Proposition 3.5.20 of \cite{Lol}, the proof of Theorem \ref{utile'} provide first that $\pi_{2*}(s^*(H^4(K_2(A),\Z)))$ is primitive in $H^4(\widetilde{K},\Z)$ and then the $H^4$ normality. 
So, the lattice $\pi_{2*}(s^*(H^4(K_2(A),\Z)))$ is primitive in $H^4(\widetilde{K},\Z)$.
\end{rmk}
\subsection{The couple $(K_{2}(A),\iota)$ is $H^{2}$-normal}\label{H2}
\begin{prop}
The couple $(K_{2}(A),\iota)$ is $H^{2}$-normal.
\end{prop}
\begin{proof}
We want to prove that the pushforward  
$\pi_{*}:H^{2}(K_{2}(A),\Z)\rightarrow H^{2}(K_{2}(A)/\iota,\Z)/\tors$ is surjective. 
By Remark \ref{Hnormal}, it is equivalent to prove that for all $x\in H^{2}(K_{2}(A),\Z)^{\iota}$,
$\pi_{*}(x)$ is divisible by 2 if and only if there exists $y\in H^{2}(K_{2}(A),\Z)$ such that $x=y+\iota^{*}(y)$. 

Let $x\in H^{2}(K_{2}(A),\Z)^{\iota}=H^{2}(K_{2}(A),\Z)$ such that $\pi_{*}(x)$ is divisible by 2, we will show that there exists $y\in H^{2}(K_{2}(A),\Z)$ such that $x=y+\iota^{*}(y)$.
By Proposition \ref{commut}, $\pi_{*}(x^2)$ is divisible by 2.
However, $x^2\in H^{4}(K_{2}(A),\Z)^{\iota}$; since $(K_{2}(A),\iota)$ is $H^{4}$-normal by Proposition \ref{H4norm}, it means that there is $z\in H^{4}(K_{2}(A),\Z)$ such that
$x^2=z+\iota^{*}(z)$.

Let $\mathcal{S}$ be, as before, the over-lattice of $\Sym^{2} H^{2}(K_{2}(A),\Z)$ where we add all the classes divisible by 2 in $H^{4}(K_{2}(A),\Z)$.
By Definition \ref{defiPi} and (\ref{discrPi}), there exist $z_{s}\in \mathcal{S}$, $z_{p}\in \Pi'$ and $\alpha\in \mathbb{N}$ such that:
$3^\alpha\cdot z= z_{s}+z_{p}$. 
Hence, we have:
$$3^\alpha\cdot x^2=2z_{s}+z_{p}+\iota^{*}(z_{p}).$$
Since $x^2\in \Sym$, $z_{p}+\iota^{*}(z_{p})=0$.
It follows:
\begin{equation}
3^\alpha\cdot x^2=2z_{s}.
\label{haha}
\end{equation}
let $(u_{1},u_{2},v_{1},v_{2},w_{1},w_{2},e)$ be the integral basis of $H^{2}(K_{2}(A),\Z)$ introduced in Notation~\ref{BasisH2KA}.
We can write:
$$x=\alpha_{1}u_{1}+\alpha_{2}u_{2}+\beta_{1}v_{1}+\beta_{2}v_{2}+\gamma_{1}w_{1}+\gamma_{2}w_{2}+de.$$
Then $$3^\alpha\cdot x^{2}=\alpha_{1}^{2}u_{1}^{2}+\alpha_{2}^{2}u_{2}^{2}+\beta_{1}^{2}v_{1}^{2}+\beta_{2}^{2}v_{2}^{2}+\gamma_{1}^{2}w_{1}^{2}+\gamma_{2}^{2}w_{2}^{2}+d^{2}e^{2}\mod 2H^{4}(K_{2}(A),\Z).$$
It follows by (\ref{haha}) that 
$\alpha_{1}^{2}u_{1}^{2}+\alpha_{2}^{2}u_{2}^{2}+\beta_{1}^{2}v_{1}^{2}+\beta_{2}^{2}v_{2}^{2}+\gamma_{1}^{2}w_{1}^{2}+\gamma_{2}^{2}w_{2}^{2}+d^{2}e^{2}$ is divisible by 2.
However, by Theorem \ref{integralbasistheorem} (i), we have:
\begin{equation}
\mathcal{S}=\left\langle \Sym^2 H^{2}(K_{2}(A),\Z); \frac{u_{1}\cdot u_{2}+v_{1}\cdot v_{2}+w_{1}\cdot w_{2}}{2};\frac{u_{i}^{2}-\frac{1}{3}u_{i}\cdot e}{2};\frac{v_{i}^{2}-\frac{1}{3}v_{i}\cdot e}{2};\frac{w_{i}^{2}-\frac{1}{3}w_{i}\cdot e}{2},i\in\left\{1,2\right\}\right\rangle.
\label{generatorsS}
\end{equation}
The $\frac{1}{2}(\alpha_{1}^{2}u_{1}^{2}+\alpha_{2}^{2}u_{2}^{2}+\beta_{1}^{2}v_{1}^{2}+\beta_{2}^{2}v_{2}^{2}+\gamma_{1}^{2}w_{1}^{2}+\gamma_{2}^{2}w_{2}^{2}+d^{2}e^{2})$ is in $\mathcal{S}$ and so can be expressed as a linear combination of the generators of $\mathcal{S}$.
Then, it follows from (\ref{generatorsS}) that all the coefficients of $\alpha_{1}^{2}u_{1}^{2}+\alpha_{2}^{2}u_{2}^{2}+\beta_{1}^{2}v_{1}^{2}+\beta_{2}^{2}v_{2}^{2}+\gamma_{1}^{2}w_{1}^{2}+\gamma_{2}^{2}w_{2}^{2}+d^{2}e^{2}$ are divisible by 2.
It means that $x$ is divisible by 2. This is what we wanted to prove.
\end{proof}
With exactly the same proof working in $H^4(\widetilde{K},\Z)$ and using Remark \ref{primitive1}, we provide the following lemma.
\begin{lemme}\label{primitive2}
The lattice $\pi_{2*}(s^*(H^2(K_2(A),\Z)))$ is primitive in $H^2(\widetilde{K},\Z)$.
\end{lemme}
\subsection{Calculation of $H^{2}(K',\Z)$}\label{basisinteK'}
This section is devoted to prove the following theorem.
\begin{thm}\label{fin}
Let $K'$, $\pi_1$, $s_1$ and $\overline{Z_0}'$ be respectively the variety, the maps and the class defined in Section \ref{nota}. 
We have $$H^{2}(K',\Z)=\pi_{1*}(s_{1}^{*}(H^2(K_2(A),\Z)))\oplus\Z\left(\frac{\pi_{1*}(s_{1}^{*}(e))+\overline{Z_{0}}'}{2}\right)\oplus\Z\left(\frac{\pi_{1*}(s_{1}^{*}(e))-\overline{Z_{0}}'}{2}\right).$$
\end{thm}

First we need to calculate some intersections.
\begin{lemme}\label{Fulton}
\begin{itemize}
\item[(i)]
We have $E_{l}\cdot E_{k}=0$ if $l\neq k$, $E_{l}^{4}=-1$ and $E_{l}\cdot z=0$ for all $(l,k)\in \left\{1,...,28\right\}^{2}$
and for all $z\in s^{*}(H^{2}(K_2(A),\Z))$.
\item[(ii)]
We have $e^4=324$.\newline
\end{itemize}
We already have some properties of primitivity:

\begin{itemize}
\item[(iii)]
$\pi_{1*}(s_1^{*}(H^{2}(K_{2}(A),\Z)))$ is primitive in $H^{2}(K',\Z)$,
\item[(iv)]
The group $\widetilde{\mathcal{D}}=\left\langle \widetilde{\overline{Z_0}},D_{1},...,D_{36},\frac{\widetilde{\overline{Z_0}}+D_{1}+...+D_{36}}{2}\right\rangle$ is primitive in $H^{2}(\widetilde{K},\Z)$.
\item[(v)]
$\overline{Z_{0}}'$ is primitive in $H^{2}(K',\Z)$,
\end{itemize}
\end{lemme}
\begin{proof}
\begin{itemize}
\item[(i)]
It is the same statement as Proposition 4.6.16 1) of \cite{Lol}, proven using adjunction formula.
\item[(ii)]
It follows directly from the Fujiki formula (\ref{fujiki}).
\item[(iii)]
By Lemma \ref{primitive2}, $\pi_{2*}(s^*(H^2(K_2(A),\Z)))$ is primitive in $H^{2}(\widetilde{K},\Z)$. Then by Remark \ref{commut2}, $r'^*(\pi_{*}(H^2(K_2(A),\Z)))$ is primitive in $H^{2}(K',\Z)$. Using again Remark \ref{commut2}, we get the result.
\item[]
The proof of (iv) and (v) is the same as Lemma 4.6.14 of \cite{Lol} and will be omitted. 
\end{itemize}
\end{proof}
With Lemma \ref{Fulton} (iii) and (v), it only remains to prove that $\pi_{1*}(s_{1}^{*}(e))+\overline{Z_{0}}'$ is divisible by 2 which will be done in Lemma \ref{dernierlemme}. To prove this lemma, we first prove that $\pi_{2*}(s^{*}(e))+\widetilde{\overline{Z_{0}}}$ is divisible by 2. Knowing that $\widetilde{\overline{Z_0}}+D_{1}+...+D_{36}$ is divisible by 2, we only have to show that $\pi_{2*}(s^{*}(e))+D_{1}+...+D_{36}$ is divisible by 2 which is done by Lemma \ref{exist} and \ref{Ddelta}.

First we need to know the group $H^{3}(\widetilde{K},\Z)$.
\begin{lemme}\label{H3}
We have $H^{3}(\widetilde{K},\Z)=0.$
\end{lemme}
\begin{proof}
We have the following exact sequence: 
$$\xymatrix@C=10pt@R=0pt{H^{3}(K_2(A),V,\Z)\ar[r] &H^{3}(K_2(A),\Z)\ar[r]^{f} &H^{3}(V,\Z)\ar[r]& 
H^{4}(K_2(A),V,\Z)\ar[r]^{\rho} &H^{4}(K_2(A),\Z).}$$
By Thom isomorphism, $H^{3}(K_2(A),V,\Z)=0$ and $H^{4}(K_2(A),V,\Z)=H^{0}(Z_0,\Z)$.
Moreover $\rho$ is injective, so $H^{3}(V,\Z)=H^{3}(K_2(A),\Z)$. 

Hence by (\ref{l22}), (\ref{l3}), Proposition 3.2.8 of \cite{Lol} and since $H^{3}(K_2(A),\Z)^\iota=0$, we find that $H^{3}(U,\Z)=0$.
Then the result follows from the exact sequence
$$\xymatrix@C=10pt@R=0pt{H^{3}(\widetilde{K},U,\Z)\ar[r] &H^{3}(\widetilde{K},\Z)\ar[r] &H^{3}(U,\Z)}$$
and from the fact that $H^{3}(\widetilde{K},U,\Z)=0$ by Thom isomorphism.
\end{proof}
To prove the next lemma, 
we will need a proposition from Section 7 of \cite{BNS} about Smith theory. Let $X$ be a topological space and let $G=\left\langle \iota\right\rangle$ be an involution acting on $X$. 
Let $\sigma:=1+\iota\in \mathbb{F}_{2}[G]$. We consider the chain complex $C_{*}(X)$ of $X$ with coefficients in $\mathbb{F}_{2}$ and its subcomplex $\sigma C_{*}(X)$. We denote by $X^{G}$ the fixed locus of the action of $G$ on $X$. 
\begin{prop}\label{SmithProp}
\begin{itemize}
\item[(1)] (\cite{Bredon}, Theorem 3.1). There is an exact sequence of complexes:
$$\xymatrix@C=20pt{0\ar[r] &\sigma C_{*}(X)\oplus C_{*}(X^{G})\ar[r]^{\ \ \ \ \ \ f}&C_{*}(X) \ar[r]^{\sigma}&\sigma C_{*}(X) \ar[r]&0
},$$ where $f$ denotes the sum of the inclusions.
\item[(2)] (\cite{Bredon}, (3.4) p.124). There is an isomorphism of complexes:
$$\sigma C_{*}(X)\simeq C_{*}(X/G,X^{G}),$$
where $X^{G}$ is identified with its image in $X/G$.
\end{itemize}
\end{prop}
\begin{lemme}\label{exist}
There exists $D_e$ which is a linear combination of the $D_i$ with coefficient 0 or 1 such that $\pi_{2*}(s^{*}(e))+D_e$ is divisible by 2.
\end{lemme}
\begin{proof}
First, we have to use Smith theory as in Section 4.6.4 of \cite{Lol}.

Look at the following exact sequence:
$$\xymatrix@C=10pt@R0pt{0\ar[r] &H^{2}(\widetilde{K},\widetilde{\overline{Z_{0}}}\cup(\cup_{k=1}^{36}D_{k}),\mathbb{F}_{2}))\ar[r]&H^{2}(\widetilde{K},\mathbb{F}_{2}) \ar[r]& H^{2}(\widetilde{\overline{Z_{0}}}\cup(\cup_{k=1}^{36}D_{k}),\mathbb{F}_{2}))\\
\ar[r]&H^{3}(\widetilde{K},\widetilde{\overline{Z_{0}}}\cup(\cup_{k=1}^{36}D_{k}),\mathbb{F}_{2})\ar[r] &0.\ \ \ \ \ \ \ \ \ &
}$$
First, we will calculate the dimension of the vector spaces $H^{2}(\widetilde{K},\widetilde{\overline{Z_{0}}}\cup(\cup_{k=1}^{36}D_{k}),\mathbb{F}_{2})$ and $H^{3}(\widetilde{K},$ $\widetilde{\overline{Z_{0}}}\cup(\cup_{k=1}^{36}D_{k}),\mathbb{F}_{2})$.
By (2) of Proposition \ref{SmithProp}, we have 
$$H^{*}(\widetilde{K},\widetilde{Z_{0}}\cup(\cup_{k=1}^{36}D_{k}),\mathbb{F}_{2})\simeq H^{*}_{\sigma}(N_{2}).$$

The previous exact sequence gives us the following equation:

$$h^{2}_{\sigma}(N_{2})-h^{2}(\widetilde{K},\mathbb{F}_{2})+h^{2}(\widetilde{Z_{0}}\cup(\cup_{k=1}^{36}D_{k}),\mathbb{F}_{2})-h^{3}_{\sigma}(N_{2})=0.$$
As $h^{2}(\widetilde{K},\mathbb{F}_{2})=8+36=44$ and $h^{2}(\widetilde{Z_{0}}\cup(\cup_{k=1}^{36}D_{k}),\mathbb{F}_{2})=23+36=59$, we obtain:
$$h^{2}_{\sigma}(N_{2})-h^{3}_{\sigma}(N_{2})=-15.$$

Moreover by 2) of Proposition \ref{SmithProp}, we have the exact sequence
$$\xymatrix@C=10pt@R0pt{0\ar[r] &H^{1}_{\sigma}(N_{2})\ar[r]&H^{2}_{\sigma}(N_{2}) \ar[r]&H^{2}(N_{2},\mathbb{F}_{2}) \ar[r]&H^{2}_{\sigma}(N_{2})\oplus H^{2}(\widetilde{Z_{0}}\cup(\cup_{k=1}^{36}E_{k}),\mathbb{F}_{2})\\
\ar[r]&H^{3}_{\sigma}(N_{2})\ar[r]&\coker\ar[r] &0.\ \ \ \ \ \ \ \ \ \ \  &
}$$
By Lemma 7.4 of \cite{BNS}, $h^{1}_{\sigma}(N_{2})=h^{0}(\widetilde{Z_{0}}\cup(\cup_{k=1}^{36}E_{k}),\mathbb{F}_{2})-1$.
Then we get the equation
\begin{align*}
&h^{0}(\widetilde{Z_{0}}\cup(\cup_{k=1}^{36}E_{k}),\mathbb{F}_{2})-1-h^{2}_{\sigma}(N_{2})+h^{2}(N_{2},\mathbb{F}_{2})\\
&-h^{2}_{\sigma}(N_{2})-h^{2}(\widetilde{Z_{0}}\cup(\cup_{k=1}^{36}E_{k}),\mathbb{F}_{2})+h^{3}_{\sigma}(N_{2})-\alpha=0,\\
\end{align*}
where $\alpha=\dim \coker$.
So
$$21-\alpha-2h^{2}_{\sigma}(N_2)+h^{3}_{\sigma}(N_2)=0.$$
From the two equations, we deduce that
$$h^{2}_{\sigma}(N_{2})=36-\alpha,\ \ \ \ \ h^{3}_{\sigma}(N_{2})=51-\alpha.$$

Come back to the exact sequence
$$\xymatrix@C=15pt{0\ar[r] &H^{2}(\widetilde{K},\widetilde{\overline{Z_{0}}}\cup(\cup_{k=1}^{36}D_{k}),\mathbb{F}_{2})\ar[r]&\ar[r]^{\varsigma^{*}\ \ \ \ \ \ \  }H^{2}(\widetilde{K},\mathbb{F}_{2}) & H^{2}(\widetilde{\overline{Z_{0}}}\cup(\cup_{k=1}^{36}D_{k}),\mathbb{F}_{2}),
}$$
where $\varsigma:\widetilde{\overline{Z_{0}}}\cup(\cup_{k=1}^{36}D_{k})\hookrightarrow \widetilde{K}$ is the inclusion.
Since $h^{2}(\widetilde{K},\widetilde{\overline{Z_{0}}}\cup(\cup_{k=1}^{36}D_{k}),\mathbb{F}_{2})=h^{2}_{\sigma}(N_{2})=36-\alpha$, we have $\dim_{\mathbb{F}_{2}} \varsigma^{*}(H^{2}(\widetilde{K},\mathbb{F}_{2}))=(8+36)-36+\alpha=8+\alpha$.
We can interpret this as follows.
Consider the homomorphism
\begin{align*}
\varsigma^{*}_{\Z}:H^{2}(\widetilde{K},\Z)&\rightarrow H^{2}(\widetilde{\overline{Z_{0}}},\Z)\oplus (\oplus_{k=1}^{36} H^{2}(D_{k},\Z))\\
 u&\rightarrow (u\cdot\widetilde{\overline{Z_{0}}},u\cdot D_{1},...,u\cdot D_{36}).
\end{align*}
Since this is a map of torsion free $\Z$-modules (by Lemma \ref{H3} and universal coefficient formula), we can tensor by $\mathbb{F}_{2}$,
$$\varsigma^{*}=\varsigma^{*}_{\Z}\otimes \id_{\mathbb{F}_{2}}: H^{2}(\widetilde{K},\Z)\otimes\mathbb{F}_{2}\rightarrow H^{2}(\widetilde{\overline{Z_{0}}},\Z)\oplus (\oplus_{k=1}^{36} H^{2}(D_{k},\Z))\otimes\mathbb{F}_{2},$$
and we have $8+\alpha$ independent elements such that the intersection with the $D_{k}$, $k\in\left\{1,...,36\right\}$ and $\widetilde{\overline{Z_{0}}}$ are not all zero.
But, $\varsigma^{*}(\pi_{2*}(H^{2}(N_2,\Z)))=0$ and $\varsigma^{*}(\widetilde{\overline{Z_0}},\left\langle D_1,...,D_{36}\right\rangle)$, (it follows from Proposition \ref{commut}). By Lemma \ref{Fulton} (iv), the element $\widetilde{\overline{Z_0}}+D_{1}+...+D_{36}$ is divisible by 2. 
Hence necessary, it remain $7+\alpha$ elements $\frac{u_1+d_1}{2}$,...,$\frac{u_{7+\alpha}+d_{7+\alpha}}{2}$ in 
$H^2(\widetilde{K},\Z)$, linearly independent as elements of $H^2(\widetilde{K},\mathbb{F}_2)$ and with $u_i\in \pi_{2*}(s^{*}(H^{2}(K_{2}(A),\Z)))$, $d_i\in\left\langle D_1,...,D_{36}\right\rangle$.

By Lemma \ref{Fulton} (iv), $\left\langle D_1,...,D_{36}\right\rangle$ is primitive in $H^2(\widetilde{K},\Z)$. Hence necessary, the element $u_1,...,u_{7+\alpha}$ viewed as element in $\pi_{2*}(s^{*}(H^{2}(K_{2}(A),\mathbb{F}_{2})))$ are also linearly independent. 
Since $\dim_{\mathbb{F}_{2}}$ $\pi_{2*}(s^{*}(H^{2}(K_{2}(A),\mathbb{F}_{2})))=7$, it follows that $\alpha=0$ and $\Vect_{\mathbb{F}_{2}}(u_1,...,u_{7})=\pi_{2*}(s^{*}(H^{2}(K_{2}(A),\mathbb{F}_{2})))$.
Hence $e\in\Vect_{\mathbb{F}_{2}}(u_1,...,u_{7})$ and there exists $D_e$ which is a linear combination of the $D_i$ with coefficient 0 or 1 such that $\pi_{2*}(s^{*}(e))+D_e$ is divisible by 2.
\end{proof}
\begin{lemme}\label{Ddelta}
We have:
$$D_e=D_1+...+D_{36}.$$
\end{lemme}
\begin{proof}
We know from Remark \ref{SPA2} that the image of the monodromy representation on $A[2]$ contains the symplectic group $\Sp A[2]$. 
We recall from Remark \ref{RemarkSymplecticInv} (2), that the $D_1,...,D_{35}$ are given by $\pi_2(s^{-1}(\mathcal{P}))$.
It follows that the image of the monodromy representation on $H^2(\widetilde{K},\Z)$ contains the isometries which act on $D_1,...,D_{35}$
as the elements $f$ of $\Sp A[2]$:
$$f\cdot \pi_2(s^{-1}(\{a_1,a_2,a_3\})=\pi_2(s^{-1}(\{f(a_1),f(a_2),f(a_3)\}),$$
and act trivially on $D_{36}$ and $\pi_{2*}(s^{*}(e))$. 
As explained by Remark \ref{PlaneTriple}, Remark \ref{RemarkSymplecticInv} (2) and Proposition \ref{transitively}, the 2 orbits of the action of $\Sp A[2]$ on the set $\mathfrak{D}:=\left\{D_1,...,D_{35}\right\}$ correspond to the two sets of isotropic and non-isotropic planes in $A[2]$. Hence by Proposition \ref{OrbitesSp} (3), (4)  the action of $\Sp A[2]$ on the set $\mathfrak{D}$ has 2 orbits: one of 15 elements and another of 20 elements. 

On the other hand, as we have mentioned, we can assume that $A=E_\xi\times E_\xi$ where $E_\xi$ is the elliptic curve introduced in Definition \ref{elliptic6}. Hence there is the following automorphism group acting on $A$:
$$G:=\left\langle \left(
\begin{array}{cc}
\xi & 0\\
0 & 1 
\end{array} \right), 
\left(
\begin{array}{cc}
0 & 1\\
1 & 0 
\end{array} \right),
\left(
\begin{array}{cc}
1 & 1\\
0 & 1 
\end{array} \right)
 \right\rangle.$$
The group $G$ extends naturally to an automorphism group of $N_2$ (we recall that $N_2$ is defined in Section (\ref{nota})) 
which we also denote $G$.  Moreover, the action of $G$ restricts to the set $\mathfrak{D}$. Then by Lemma \ref{orbitesG} the action of $G$ on $\mathfrak{D}$ has 2 orbits: one of 5 elements and one of 30 elements. Also the group $G$ acts trivially on $D_{36}$ and on $\pi_{2*}(s^{*}(e))$. 

Hence the combined action of $G$ and $\Sp A[2]$ acts transitively on $\mathfrak{D}$.
Since $\pi_{2*}(s^{*}(e))$ is fixed by the action of $G$ and $\Sp A[2]$, $D_e$ has also to be fixed by the action of $G$ and $\Sp A[2]$ else it will contradict Lemma \ref{Fulton} (iv).
It follows that there are only 3 possibilities for $D_e$: 
\begin{itemize}
\item[(1)]
$D_e=D_{36}$,
\item[(2)]
$D_e=D_1+...+D_{35}$,
\item[(3)]
or $D_e=D_1+...+D_{36}$.
\end{itemize}
Let $d$ be the number of $D_i$ with coefficient equal to 1 in the linear decomposition of $D_e$. The number $d$ can be 1, 35 or 36.

Then from Lemma \ref{Fulton} (i), (ii) and Proposition \ref{commut}
$$\left(\frac{\pi_{2*}(s^{*}(e))+D_e}{2}\right)^4=\frac{324-d}{2}.$$ 
Hence $d$ has to be divisible by 2.
It follows that $D_e=D_1+...+D_{36}$.
\end{proof} 
\begin{lemme}\label{dernierlemme}
The class $\pi_{1*}(s_1^{*}(e))+\overline{Z_{0}}'$ is divisible by 2.
\end{lemme}
\begin{proof}
We know that $\pi_{2,*}(s^{*}(e))+\widetilde{\overline{Z_{0}}}$ is divisible by 2.
Indeed by Lemma \ref{Fulton} (iv), $\widetilde{\overline{Z_{0}}}+D_1+...+D_{36}$ is divisible by 2 and by Lemma \ref{exist} and \ref{Ddelta},
$\pi_{2,*}(s^{*}(e))+D_1+...+D_{36}$ is divisible by 2. 

We can find a Cartier divisor on $\widetilde{K}$ which corresponds to $\frac{\pi_{2*}(s^{*}(e))+\widetilde{\overline{Z_0}}}{2}$ and which does not meet 
$\cup_{k=1}^{36} D_k$.
Then this Cartier divisor induces a Cartier divisor on $K'$ which necessarily corresponds to half the cocycle $\pi_{1*}(s_{1}^{*}(e))$ $+\overline{Z_0}'$.
\end{proof}
\subsection{Computing $B_{K'}$}\label{beauK'}
We finish the proof of Theorem \ref{theorem}, computing the Beauville--Bogomolov form $B_{K'}$ of $K'$. We continue using the notation provided in Section \ref{nota}. 
One of the main ingredient will be the Fujiki formula (see Section 1.2.2 of \cite{Lol}).
If $X=K'$ or $K_2(A)$, we have for all $\alpha\in H^2(X,\Z)$:
\begin{equation}
\alpha^{4}=c_{X}B_{X}(\alpha,\alpha)^2,
\label{Fujiki}
\end{equation}
$c_X\in \mathbb{Q}$ is the Fujiki constant. 
Moreover if 
$0\neq \omega$ is the holomorphic 2-from on $X$, we have 
\begin{equation}
B_{X}(\omega+\overline{\omega},\omega+\overline{\omega})>0.
\label{Fujikiposi}
\end{equation}
There also exists a polarized version of the Fujiki formula.
\begin{equation}
\alpha_{1}\cdot \alpha_{2}\cdot\alpha_{3}\cdot\alpha_{4}=\frac{c_{X}}{24}\sum_{\sigma\in \mathfrak{S}_{4}}B_{X}(\alpha_{\sigma(1)},\alpha_{\sigma(2)})\cdot B_{X}(\alpha_{\sigma(3)},\alpha_{\sigma(4)}).
\label{beauville}
\end{equation}
for all $\alpha_{i}\in H^{2}(X,\Z)$.

\begin{lemme}\label{Zinter}
We have $$\overline{Z_{0}}'^{2}=-2r^{*}(\overline{Z_{0}}).$$
\end{lemme}
\begin{proof}
We use the same technique as in Lemma 4.6.12 of \cite{Lol}.
Consider the following diagram:
$$\xymatrix{
Z_{0}' \ar[d]^{g}\ar@{^{(}->}[r]^{l_{1}}& N_{1} \ar[d]^{s_{1}}\\
   Z_{0} \ar@{^{(}->}[r]^{l_{0}}  & K_{2}(A),
   }$$
where $l_{0}$ and $l_{1}$ are the inclusions and $g:=s_{1|Z_0'}$.
By Proposition 6.7 of \cite{Fulton}, we have:
$$s_{1}^{*}l_{0*}(Z_{0})=l_{1*}(c_{1}(E)),$$
where $E:=g^{*}(\mathscr{N}_{Z_{0}/K_{2}(A)})/\mathscr{N}_{Z_{0}'/N_{1}}$.
Hence $$s_{1}^{*}l_{0*}(Z_{0})=c_{1}(g^{*}(\mathscr{N}_{Z_{0}/K_{2}(A)}))-Z_{0}'^{2}.$$
Since $K_{2}(A)$ is hyperkähler and $Z_{0}$ is a K3 surface, we have $c_{1}(\mathscr{N}_{Z_{0}/K_{2}(A)})=0$.
So
$$Z_{0}'^{2}=-s_{1}^{*}l_{0*}(Z_{0}).$$
Then the result follows from Proposition \ref{commut}.
\end{proof}
\begin{prop}\label{passage}
We have the formula $$B_{K'}(\pi_{1*}(s_{1}^{*}(\alpha),\pi_{1*}(s_{1}^{*}(\beta)))=6\sqrt{\frac{2}{c_{K'}}}B_{K_2(A)}(\alpha,\beta),$$
where $c_{K'}$ is the Fujiki constant of $K'$ and $\alpha$, $\beta$ are in $H^{2}(K_2(A),\Z)$ and $B_{K_2(A)}$ is the Beauville--Bogomolov form of $K_2(A)$.
\end{prop} 
\begin{proof}
The ingredient for the proof is the Fujiki formula.

By (\ref{Fujiki}), we have $$(\pi_{1*}(s_{1}^{*}(\alpha)))^{4}=c_{K'}B_{K'}(\pi_{1*}(s_{1}^{*}(\alpha),\pi_{1*}(s_{1}^{*}(\alpha)))^{2}\ \ \text{and}$$
$$\alpha^{4}=9B_{K_2(A)}(\alpha,\alpha)^{2}.$$
Moreover, by Proposition \ref{commut}, $$(\pi_{1*}(s^{*}(\alpha)))^{4}=8s^{*}(\alpha)^{4}=8\alpha^{4}.$$
By statement (\ref{Fujikiposi}), we get the result.
\end{proof}
In particular, it follows:
\begin{equation}
B_{K'}(\pi_{1*}(s_{1}^{*}(e),\pi_{1*}(s_{1}^{*}(e)))=-36\sqrt{\frac{2}{c_{K'}}}
\label{deltahaha}
\end{equation}
\begin{lemme}\label{ortho}
$$B_{K'}(\pi_{1*}(s_{1}^{*}(\alpha)),\overline{Z_0}')=0,$$
for all $\alpha\in H^{2}(\X,\Z)$.
\end{lemme}
\begin{proof}
We have $\pi_{1*}(s_{1}^{*}(\alpha))^{3}\cdot\overline{Z_0}'=8s_{1}^{*}(\alpha)^{3}\cdot\Sigma_{1}$ by Proposition \ref{commut},
and $s_{1*}(s_{1}^{*}(\alpha^{3})\cdot Z_0')=\alpha^{3}\cdot s_{1*}(Z_0')=0$ by the projection formula.
We conclude with (\ref{beauville}).
\end{proof}
\begin{lemme}\label{Z2}
We have:
$$B_{K'}(\overline{Z_0}',\overline{Z_0}')=-4\sqrt{\frac{2}{c_{K'}}}.$$
\end{lemme}
\begin{proof}
By (\ref{beauville}) and (\ref{deltahaha}), we have:
\begin{align*}
\overline{Z_0}'^{2}\cdot\pi_{1*}(s_{1}^{*}(e))^{2}
&=\frac{c_{K'}}{3}B_{M'}(\overline{Z_0}',\overline{Z_0}')\times B_{K'}(\pi_{1*}(s_{1}^{*}(e)),\pi_{1*}(s_{1}^{*}(e)))\\
&=\frac{c_{K'}}{3}B_{K'}(\overline{Z_0}',\overline{Z_0}')\times\left(-36\sqrt{\frac{2}{c_{K'}}}\right)\\
\end{align*}
\vspace{-1cm}
\begin{equation}
=-12\sqrt{2c_{K'}}B_{K'}(\overline{Z_0}',\overline{Z_0}')\ \ \ \ \ \ 
\label{jenesaispas1}
\end{equation}
By Proposition \ref{commut}, we have 
\begin{equation}
\overline{Z_0}'^{2}\cdot\pi_{1*}(s_{1}^{*}(e))^{2}=8Z_0'^{2}\cdot (s_{1}^{*}(e))^{2}.
\label{jenesaispas2}
\end{equation}
By the projection formula, 
$Z_0'^{2}\cdot (s_{1}^{*}(e))^{2}=s_{1*}(Z_0'^{2})\cdot e^{2}$.
Moreover by lemma \ref{Zinter}, $s_{1*}(Z_0'^{2})=-Z_0$. 
Hence
\begin{equation}
Z_0'^{2}\cdot (s_{1}^{*}(e))^{2}=-Z_0\cdot e^{2}.
\label{jenesaispas3}
\end{equation}
It follows from (\ref{jenesaispas1}), (\ref{jenesaispas2}) and (\ref{jenesaispas3}) that
\begin{equation}
-8Z_0\cdot e^{2}=-12\sqrt{2c_{K'}}B_{K'}(\overline{Z_0}',\overline{Z_0}').
\label{jenesaispas4}
\end{equation}
Moreover from Section 4 of \cite{Hassett}, we have:
\begin{equation}
Z_0\cdot e^{2}=-12.
\label{jenesaispas5}
\end{equation}
So by (\ref{jenesaispas4}) and (\ref{jenesaispas5}):
$$B_{K'}(\overline{Z_0}',\overline{Z_0}')=-8\sqrt{\frac{1}{2c_{K'}}}.$$
\end{proof}
Now we are able to finish the calculation of the Beauville--Bogomolov form on $H^{2}(K',\Z)$.
By (\ref{deltahaha}), Propositions \ref{passage}, Lemma \ref{ortho}, \ref{Z2} and Theorem \ref{fin},
the Beauville--Bogomolov form on $H^{2}(K',\Z)$ gives the lattice:
$$ U^{\oplus3}\left(6\sqrt{\frac{2}{c_{K'}}}\right) \oplus -\frac{1}{4}\sqrt{\frac{2}{c_{K'}}}\left(
\begin{array}{cc}
40 & 32\\
32 & 40 
\end{array} \right)$$
$$=U^{\oplus3}\left(6\sqrt{\frac{2}{c_{K'}}}\right) \oplus -\sqrt{\frac{2}{c_{K'}}}\left(
\begin{array}{cc}
10 & 8\\
8 & 10 
\end{array} \right)$$
Then it follows from the integrality and the indivisibility of the Beauville--Bogomolov form that $c_{K'}=8$, and we get Theorem \ref{theorem}.
\subsection{Betti numbers and Euler characteristic of $K'$}
\begin{prop}\label{b}
The Betti numbers and the Euler characteristic are
\begin{align*}
 b_2(K')&=8,&
b_{3}(K')&=0,&
b_{4}(K')&=90,&
\chi(K')&=108.
\end{align*}
\end{prop}
\begin{proof}
It is the same proof as Proposition 4.7.2 of \cite{Lol}.
From Theorem 7.31 of \cite{Voisin}, (\ref{l22}), (\ref{l3}) and Proposition \ref{invariants}, we get the Betti numbers.
Then $\chi(K')=1-0+8-0+90-0+8-0+1=108$.
\end{proof}

\appendix
\section{Divisible classes in $H^4(\kum{A}{2},\Z)$}
Here we give the divisible classes from Section~\ref{Middle} that were determined by using a computer.
\begin{prop}\label{XXXI}
The 31 following classes of $\Pi'$ are divisible by 3 in $H^{4}(K_2(A),\Z)$ and their thirds span a $\mathbb F_3$-vector space of dimension 31 in $\frac{\Pi'^{sat}}{\Pi'}$.
$$\sum_{\tau\in\Lambda} \Big(Z_{\tau} - Z_{\tau+\tau'}\Big), \text{ with }$$
\begin{itemize}
\item[(i)]
$\Lambda=\plan{1\\0\\0\\0}{0\\1\\0\\0} \text{ and } 0\neq \tau'\in P^\perp = \plan{0\\0\\1\\0}{0\\0\\0\\1} $,

\item[(ii)] 
$\Lambda=\plan{0\\0\\1\\0}{0\\0\\0\\1}  \text{ and } 0\neq \tau' \in P^\perp = \plan{1\\0\\0\\0}{0\\1\\0\\0} \setminus \vect{1\\0\\0\\0}$,

\item[(iii)] 
$\Lambda=\plan{1\\0\\0\\1}{0\\1\\2\\1} \text{ and } \tau' \in \left\{ \vect{0\\1\\1\\2},\vect{1\\0\\0\\2},\vect{1\\1\\1\\1},\vect{2\\2\\2\\2} \right\}$,

\item[(iv)] 
$\Lambda=\plan{1\\0\\0\\0}{0\\1\\0\\1} \text{ and } \tau' \in \left\{ \vect{0\\0\\0\\1},\vect{2\\0\\1\\2},\vect{1\\0\\2\\0},\vect{1\\0\\2\\1} \right\}$,
\item[(v)]
$\Lambda=\plan{1\\0\\0\\0}{0\\1\\1\\1} \text{ and } \tau' \in \left\{ \vect{0\\0\\1\\1},\vect{1\\0\\0\\1} \right\}$,

\item[(vi)] 
$\Lambda=\plan{1\\0\\1\\1}{0\\1\\0\\1} \text{ and } \tau' \in \left\{ \vect{0\\1\\0\\2},\vect{1\\0\\2\\2} \right\}$,

\item[(vii)]
$\Lambda=\plan{1\\0\\1\\0}{0\\1\\0\\1} \text{ and } \tau' \in \left\{ \vect{0\\1\\0\\2},\vect{1\\0\\2\\0} \right\}$,

\item[(viii)]
$\Lambda=\plan{1\\0\\0\\0}{0\\1\\0\\2} \text{ and } \tau' = \vect{1\\0\\1\\0}$,

\item[(ix)]
$\Lambda=\plan{1\\0\\1\\1}{0\\1\\2\\2} \text{ and } \tau' = \vect{1\\1\\0\\2}$.
\end{itemize}
\end{prop}

\begin{prop}\label{XIX}
We use Notation \ref{BasisH2KA}.
The 19 following classes are divisible by 3 in $H^{4}(K_2(A),\Z)$ and their thirds provide a sub-vector space of dimension 19 of $\frac{H^4(K_2(A),\Z)}{\Sym^{sat}\oplus\Pi'^{sat}}$.
\begin{itemize}
\item[(i)]
$u_2^2+\sum_{\tau\in \Lambda} Z_\tau-Z_0$, for $\Lambda = \plan{0\\0\\0\\1}{0\\0\\1\\0}$,
\item[(ii)]
$v_2^2+v_2u_2+u_2^2+\sum_{\tau\in \Lambda} Z_\tau-Z_0$, for $\Lambda= \plan{0\\0\\0\\1}{0\\1\\1\\0}$,
\item[(iii)]
$w_2^2+w_2u_2+u_2^2+\sum_{\tau\in \Lambda} Z_\tau-Z_0$, for $\Lambda= \plan{0\\0\\1\\0}{0\\1\\0\\1}$,
\item[(iv)]
$w_2^2-w_2u_2+u_2^2+\sum_{\tau\in \Lambda} Z_\tau-Z_0$, for  $\Lambda= \plan{0\\0\\1\\0}{0\\1\\0\\2}$,
\item[(v)]
$w_2^2-w_2v_2+w_2u_2+v_2^2+v_2u_2+u_2^2+\sum_{\tau\in \Lambda} Z_\tau-Z_0$, for $\Lambda= \plan{0\\0\\1\\2}{0\\1\\0\\1}$,
\item[(vi)]
$w_1^2+w_1u_2+u_2^2+\sum_{\tau\in \Lambda} Z_\tau-Z_0$, for  $\Lambda= \plan{0\\0\\0\\1}{1\\0\\2\\0}$,
\item[(vii)]
$w_1^2-w_1u_2+u_2^2+\sum_{\tau\in \Lambda} Z_\tau-Z_0$, for $\Lambda= \plan{0\\0\\0\\1}{1\\0\\1\\0}$,
\item[(viii)]
$v_1^2+v_1u_2+u_2^2+\sum_{\tau\in \Lambda} Z_\tau-Z_0$, for $\Lambda= \plan{0\\0\\1\\0}{1\\0\\0\\1}$,
\item[(ix)]
$v_1^2-v_1u_2+u_2^2+\sum_{\tau\in \Lambda} Z_\tau-Z_0$, for $\Lambda= \plan{0\\0\\1\\0}{1\\0\\0\\2}$,
\item[(x)]
$v_1^2+v_1w_1-v_1u_2+w_1^2+w_1u_2+u_2^2+\sum_{\tau\in \Lambda} Z_\tau-Z_0$, for $\Lambda = \plan{0\\0\\1\\2}{1\\0\\0\\2}$,
\item[(xi)]
$v_1^2+v_1w_1-v_1w_2-v_1v_2+v_1u_2+w_1^2+w_1w_2+w_1v_2-w_1u_2+w_2^2-w_2v_2+w_2u_2+v_2^2+v_2u_2+u_2^2+\sum_{\tau\in \Lambda} Z_\tau-Z_0$, for $\Lambda = \plan{0\\0\\1\\2}{1\\1\\0\\1}$,
\item[(xii)]
$v_1^2-v_1w_1+v_1w_2-v_1v_2+v_1u_2+w_1^2+w_1w_2-w_1v_2+w_1u_2+w_2^2+w_2v_2-w_2u_2+v_2^2+v_2u_2+u_2^2+\sum_{\tau\in \Lambda} Z_\tau-Z_0$, for $\Lambda = \plan{0\\0\\1\\1}{1\\2\\0\\1}$,
\item[(xiii)]
$u_1^2+\sum_{\tau\in \Lambda} Z_\tau-Z_0$, for  $\Lambda = \plan{0\\1\\0\\0}{1\\0\\0\\0}$,
\item[(xiv)]
$u_1^2-u_1v_2+v_2^2+ \sum_{\tau\in \Lambda} Z_\tau-Z_0$, for  $\Lambda = \plan{0\\1\\0\\0}{1\\0\\0\\1}$,
\item[(xv)]
$u_1^2+u_1v_2+v_2^2+\sum_{\tau\in \Lambda} Z_\tau-Z_0$, for $\Lambda= \plan{0\\1\\0\\0}{1\\0\\0\\2}$,
\item[(xvi)]
$u_1^2+u_1w_1+w_1^2+\sum_{\tau\in \Lambda} Z_\tau-Z_0$, for $\Lambda = \plan{0\\1\\0\\2}{1\\0\\0\\0}$,
\item[(xvii)]
$u_1^2+u_1w_1-u_1v_2+w_1^2+w_1v_2+v_2^2+\sum_{\tau\in \Lambda} Z_\tau-Z_0$, for $\Lambda = \plan{0\\1\\0\\2}{1\\0\\0\\1}$,
\item[(xviii)]
$u_1^2-u_1w_1+u_1w_2-u_1u_2+w_1^2+w_1w_2-w_1u_2+w_2^2+w_2u_2+u_2^2+\sum_{\tau\in \Lambda} Z_\tau-Z_0$, for  $\Lambda = \plan{0\\1\\0\\1}{1\\0\\1\\0}$,
\item[(xix)]
$u_1^2+u_1v_1-u_1w_1+v_1^2+v_1w_1+w_1^2+\sum_{\tau\in \Lambda} Z_\tau-Z_0$, for $\Lambda = \plan{0\\1\\2\\1}{1\\0\\0\\0}$.
\end{itemize}
\end{prop}

\noindent
\textbf{Acknowledgements. }
We want to thank Samuel Boissi\`ere, David Chataur, Brendan Hassett, Giovanni Mongardi, Marc Nieper-Wi\ss kirchen and Ulrike Rie\ss\ for useful discussions.
We also thank Samuel Boissi\`ere, Daniel Huybrechts and Marc Nieper-Wi\ss kirchen for hospitality.
GM is supported by Fapesp grant 2014/05733-9. SK was partially supported by a DAAD grant.

\bibliographystyle{amsplain}

\end{document}